\documentclass[11pt]{amsart}

\usepackage{eucal}
\usepackage{upgreek}
\usepackage{pdfsync}
\usepackage[all, cmtip]{xy}
\usepackage{amsfonts}
\usepackage{amsmath}
\usepackage{amssymb}
\usepackage{amscd}
\usepackage{color,xcolor}
\usepackage{amsthm}
\usepackage{amsxtra}
\usepackage{bbm}
\usepackage{graphicx}
\usepackage{mathrsfs}
\usepackage[numbers,sort]{natbib}
\usepackage{isomath}
\usepackage{enumitem}
\usepackage{latexsym} 

\newtheorem{thm}{Theorem}

\newtheorem{prop}{Proposition}
\newtheorem{lma}[prop]{Lemma}

\newtheorem{cor}[prop]{Corollary}

\theoremstyle{definition}

\newtheorem{df}[prop]{Definition} 

\theoremstyle{remark}

\newtheorem{rmk}[prop]{Remark} 
\newtheorem{exa}[prop]{Example}

\def\mrm#1{{\mathrm{#1}}}

\newcommand{\R}{{\mathbb{R}}}
\newcommand{\Z}{{\mathbb{Z}}}
\newcommand{\C}{{\mathbb{C}}}

\newcommand{\N}{{\mathbb{N}}}
\newcommand{\D}{{\mathbb{D}}}
\newcommand{\bK}{{\mathbb{K}}}
\newcommand{\HH}{{\mathbb{H}}}
\newcommand{\bH}{{\mathbb{H}}}
\newcommand{\bF}{{\mathbb{F}}}

\newcommand{\bs}{\bigskip}

\newcommand{\del}{\partial}

\newcommand{\sm}[1]{C^\infty(#1)}

\newcommand\vol{\operatorname{vol}}

\newcommand{\A}{\mathcal{A}}

\newcommand{\cL}{\mathcal{L}}

\newcommand{\til}[1]{\widetilde{#1}}

\newcommand{\codim}{\text{codim}}

\newcommand{\om}{\omega}

\newcommand{\Om}{\Omega}

\newcommand{\eps}{\epsilon}

\newcommand{\cA}{\mathcal{A}}
\newcommand{\cB}{\mathcal{B}}

\newcommand{\cD}{\mathcal{D}}

\newcommand{\cG}{\mathcal{G}}
\newcommand{\cH}{\mathcal{H}}

\newcommand{\cJ}{\mathcal{J}}

\newcommand{\cO}{\mathcal{O}}

\newcommand{\cP}{\mathcal{P}}
\newcommand{\cR}{\mathcal{R}}
\newcommand{\cS}{\mathcal{S}}
\newcommand{\cM}{\mathcal{M}}

\newcommand{\tmin}{{\text{min},\bK}}
\newcommand{\tmon}{{\text{mon},\bK}}
\newcommand{\tuniv}{{\text{univ},\bK}}

\newcommand{\CP}{\mathbb{C} {P}}
\newcommand{\RP}{\mathbb{R} {P}}
\newcommand{\tHam}{\widetilde{\text{Ham}}}
\newcommand{\Id}{{\mathbbm{1}}}

\renewcommand{\bar}[1]{\overline{#1}}

\DeclareMathOperator{\id}{\mathrm{id}}

\DeclareMathOperator{\ind}{\mathrm{ind}}
\DeclareMathOperator{\supp}{\mathrm{supp}}

\DeclareMathOperator{\val}{\mathrm{val}}

\DeclareMathOperator{\Ham}{\mathrm{Ham}}

\DeclareMathOperator{\Symp}{\mathrm{Symp}}

\DeclareMathOperator{\ima}{\mathrm{im}}

\DeclareMathOperator{\Spec}{\mathrm{Spec}}

\DeclareMathOperator{\pemod}{{\mathbf{pmod}}}
\DeclareMathOperator{\barc}{{\mathbf{Barcodes}}}
\DeclareMathOperator{\inte}{{\mathrm{int}}}

\def\H2{H^{(2)}}

\newenvironment{Properties}
{\begin{list}{}{
\setlength{\topsep}{6pt}%
\setlength{\itemsep}{4pt}%
\setlength{\labelsep}{0pt}%
\setlength{\leftmargin}{0pt}%
\setlength{\labelwidth}{0pt}%
\setlength{\listparindent}{0pt}}%
\setlength{\parskip}{0pt}%
}
{\end{list}}

\usepackage[pagebackref,hyperindex]{hyperref}

\newcommand{\akemail}{asafkisl@post.tau.ac.il}
\newcommand{\esemail}{shelukhin@dms.umontreal.ca}

\begin{document}

\title{Bounds on spectral norms and barcodes}
\date{\today}

\author{Asaf Kislev}
\address{Asaf Kislev, School of Mathematical Sciences, Tel Aviv University, Israel}
\email{\akemail}

\author{Egor Shelukhin}
\address{Egor Shelukhin, Department of Mathematics and Statistics,
  University of Montreal, C.P. 6128 Succ.  Centre-Ville Montreal, QC
  H3C 3J7, Canada}
\email{\esemail}

\bibliographystyle{abbrv}

\begin{abstract}

We investigate the relations between algebraic structures, spectral invariants, and persistence modules, in the context of monotone Lagrangian Floer homology with Hamiltonian term. Firstly, we use the newly introduced method of filtered continuation elements to prove that the Lagrangian spectral norm controls the barcode of the Hamiltonian perturbation of the Lagrangian submanifold, up to shift, in the bottleneck distance. Moreover, we show that it satisfies Chekanov type low-energy intersection phenomena, and non-degeneracy theorems. Secondly, we introduce a new averaging method for bounding the spectral norm from above, and apply it to produce precise uniform bounds on the Lagrangian spectral norm in certain closed symplectic manifolds. Finally, by using the theory of persistence modules, we prove that our bounds are in fact sharp in some cases. Along the way we produce a new calculation of the Lagrangian quantum homology of certain Lagrangian submanifolds, and answer a question of Usher.

\end{abstract}

\maketitle

\tableofcontents

\section{Introduction and main results}\label{Sect:intro}

In both Morse and Floer theory, one deals with the homology of a natural filtered complex, whose differential is given by counting trajectories of the negative gradient vector field of a certain functional connecting certain pairs of critical points. Since the Morse function, or the action functional in the Floer case, decreases along these trajectories, for each $s \in \R$ the generators whose critical values are strictly smaller than $s$ form a subcomplex. In the case of a Morse function $f$ on a closed manifold $X$, the homology $V^s_*(f)$ of this subcomplex with coefficients in a field $\bK$ is isomorphic to the homology $H_*(\{f<s\},\bK)$ of the sublevel set $\{f<s\}$ with coefficients in $\bK.$ This is a vector space of finite dimension, which is trivial for all $s \ll 0.$ Inclusions $\{f<s\} \subset \{f<t\}$ for $s \leq t$ yield maps $\pi_{s,t}: V^s_*(f) \to V^t_*(f)$ such that $\pi_{s,t} \circ \pi_{r,s} = \pi_{r,t}$ for $r \leq s \leq t.$ This family of vector spaces parametrized by a real parameter forms an algebraic structure called a {\em persistence module} which was introduced and studied extensively since the early 2000s in the data analysis community (see e.g. \cite{CEH-stability,CdSGO-structure,ELZ-simplification, BauLes,Carlsson,CZCG,CrawBo,Ghrist,CarlZom}). Recently, persistence modules found applications in symplectic topology, see for example \cite{Team,PolShe,PolSheSto,UsherZhang,Zhang,Fraser,stevenson}, with precursors in \cite{UsherBD1,UsherBD2,FOOO-polydiscs,Barannikov,CorneaRanicki}.

One of the main features of the theory of persistence modules, is the normal form theorem, whereby a persistence module $V$ is determined, up to isomorphism, by a multiset $\cB(V) = \{(I_j,m_j)\}$ of intervals $I_j \subset \R,$ that are either finite, or infinite bounded from below, coming with multiplicity $m_j \in \Z_{>0}$. The isometry theorem for pesistence modules states that the interleaving distance $d_{\mrm {inter}}(V,W)$ between two persistence modules $V,W,$ that is controlled from above by $|f-g|_{C^0}$ in the case $V=V(f), W = V(g),$ of persistence modules of Morse functions, is completely recovered by the bottleneck distance between the associated barcodes $\cB(V),\cB(W),$ which is roughly speaking the minimal distance by which one can move the endpoints of the bars in $\cB(V)$ to obtain the barcode $\cB(W).$ We refer to Section \ref{Sect:pmod} for further details.  We denote by $\pemod$ the category of persistence modules over the base field $\bK,$ which we sometimes consider as a metric space endowed with $d_{\mrm {inter}}$
and by $\barc$ the space of barcodes with the bottleneck distance. It shall also be convenient to consider the quotients $(\pemod',d'_{\mrm {inter}})$ and $({\barc}',d'_{\mrm {bottle}})$ of $(\pemod,d_{\mrm {inter}})$ and $(\barc,d_{\mrm {bottle}})$ by the isometric $\R$-action by shifts: $V \mapsto V[c], \cB \mapsto \cB[c]$ for $c \in \R,$ defined by $V[c]^t = V[c+t],$ $\cB[c] = \{(I_j-c,m_j)\}$ for $\cB = \{(I_j,m_j)\}.$ The induced metrics are the quotient metrics: for example $d'_{\mrm {inter}}([V],[W]) = \inf_{c\in \R} d_{\mrm {inter}} (V,W[c]).$ We will formulate our results in the language of $(\pemod,d_{\mrm{inter}})$ primarily, however we keep in mind that by the isometry theorem, they are equivalent to the analogous statements for $(\barc,d_{\mrm{bottle}}).$ We remark that elements of $\barc'$ have well-defined bar lengths.

In symplectic topology, for a symplectic manifold $(M,\om),$ persistence modules have been associated to Hamiltonians $H \in \cH = \sm{[0,1] \times M,\R}$ in the absolute case, and more generally to pairs $(L,H)$ where $H \in \cH$ and $L \subset M$ is a Lagrangian submanifold in the relative case. We remark that the relative case can be alternatively described by associating persistence modules to certain pairs $(L,L')$ of Lagrangian submanifolds. Let us briefly describe the relative case.

A Lagrangian submanifold $L \subset M$ of a symplectic manifold $(M,\omega)$ is called weakly monotone if there exists a positive constant $\kappa = \kappa_L$ such the relative cohomology class $\omega_L \in H^2(M,L;\R)$ of $\omega$  and the Maslov class $\mu_L$ satisfy \[\omega_L = \kappa \cdot \mu_L\] on $H^D_2(M,L;\Z) = \ima(\pi_2(M,L) \to H_2(M,L;\Z)).$ We assume in addition that $N_L \geq 2$ for the positive generator $N_L$ of $\ima(\mu_L) \subset \Z.$ If $\mu_L = 0, \omega_L = 0$ on $H^D_2(M,L;\Z),$ we call $L$ weakly exact. Finally, we call $L$ monotone if it is weakly monotone, and not weakly exact. We call a symplectic manifold $(M,\omega)$ weakly monotone, symplectically aspherical, or monotone, if respectively, the Lagragian diagonal $\Delta_M \subset M \times M^{-}$ is weakly monotone, weakly exact, or monotone, where $M^{-}$ denotes the symplectic manifold $(M,-\omega).$ Finally, for a Hamiltonian $H \in \cH$ we denote by $\{\phi^t_H\}_{t \in [0,1]}$ the Hamiltonian isotopy generated by $H:$ it is given by integrating the time-dependent vector field $X^t_H$ determined uniquely by $\iota_{X^t_H} \omega = - dH_t,$ where $H_t(-) = H(t,-) \in \sm{M,\R}.$ The group of Hamiltonian diffeomorphisms $\Ham(M,\om)$ consists of the time-one maps of such isotopies (see \cite{P-book}).

Given $L \subset M$ weakly monotone, and $H \in \cH,$ such that the intersection $\phi^1_H L \cap L$ is transverse, filtered Floer theory in the contractible class of paths from $L \to L,$ can be considered to be essentially Morse theory on a suitable cover $\til{\cP}_{pt}(L,L)$ of the path space $\cP_{pt}(L,L)$ for the action functional \[\cA_{L,H}(x,\overline{x}) = \int_{0}^{1} H(t,x(t)) - \int_{\overline{x}}\om\] where $\overline{x}$ is a capping of $x$ relative to $L,$ that is a smooth map $\overline{x}:\D \to M$ with $\overline{x}(e^{i\pi t}) = x(t),$ for $t \in [0,1]$ and $\overline{x}(e^{i\pi t}) \in L$ for $t \in [1,2].$ It gives for each $r \in \Z$ a persistence module $V_r(L,H) \in \pemod.$ We can similarly consider the action functional $\cA_{H}$ on a suitable cover $\til{\cL}_{pt} M$ space of contractible loops $\cL_{pt} M.$ In this case, if the intersection $(\phi \times \id) \Delta_M \cap \Delta_M$ is transverse in $M \times M^{-},$ Floer theory provides for each $r \in \Z$ a persistence module $V_r(H) \in \pemod.$ Moreover, $V_r(L,H), V_r(H)$ depend up to isomorphism only on the class \[[H] = [\{\phi^t_H\}_{t \in [0,1]}] \in \til{\Ham}(M,\om)\] in the universal cover of the group of Hamiltonian diffeomorphisms of $(M,\om).$ Hence the barcode $\cB_r(L,[H])$ of $V_r(L,H)$ is defined canonically. We also note that in fact by \cite[Proposition 6.2]{UsherBD2}, if $\phi^1_F(L) = \phi^1_G(L),$ there exist constants $C \in \R$ and $I \in \Z$ such that for all $r\in \Z,$ $V_r(L,F) \cong V_{r+I}(L,G)[C].$ Similarly, by \cite[Proposition 5.3]{UsherBD2}, if $\phi^1_F = \phi^1_G$ then for all $r\in \Z,$ $V_r(F) \cong V_{r+I}(G)[C],$ for certain $C \in \R,$ and $I \in 2\Z.$ Finally, if we choose the covers $\til{\cP}_{pt}(L,L),$ and $\til{\cL}_{pt} M$ to be sufficiently large (see Section \ref{Sec:Floer-persistence}), then $V_{r+1}(L,H) \cong V_r(L,H)[-\kappa_L]$ and $V_{r+2}(H) \cong V_r(H)[-2 \kappa_{\Delta_M}].$ This way, for $L \subset M$ a weakly monotone Lagrangian submanifold, and $\phi \in \Ham(M,\om),$ satisfying suitable non-degeneracy assumptions, we obtain elements $\cB'(L,\phi(L)) \in \barc'$ and $\cB'_0(\phi),\cB'_1(\phi) \in \barc',$ that depend only on $\phi(L).$ In the latter case, we may take a larger cover and get $\cB'(\phi) \in \barc',$ which should be considered as the union of $\cB_0([H]) \sqcup \cB_1([H])[-\kappa_{\Delta_M}]$ considered modulo shifts. As explained in \cite{PolShe}, these definitions can be extended to the case when $(L,H)$ or $H$ are degenerate, by taking suitable sufficiently $C^2$-small perturbations of $H.$

Finally, we note that similarly to the case of Morse theory, natural distances between Hamiltonian diffeomorphisms control the interleaving distance between the associated persistence modules. For example \cite{PolShe} for $F,G \in \cH$ with zero mean, \begin{equation}\label{eq:bottle-Hofer}d_{\mrm{inter}} (V_r(F),V_r(G)) \leq d_{\mrm{Hofer}} ([F],[G]),\end{equation}  where $d_{\mrm{Hofer}}(-,-)$ is the bi-invariant pseudo-metric on $\til{\Ham}(M,\omega)$ given in the following two equivalent ways \[d_{\mrm{Hofer}} ([F],[G]) = \inf_{[H] = [F]^{-1}[G]} \int_0^1 (\max_M H_t - \min_M H_t)\,dt,\] 
\begin{equation}\label{eq:Hofer norm linear} d_{\mrm{Hofer}} ([F],[G])= \inf_{\substack{[F']=[F],\\ [G']=[G]}} \int_0^1 (\max_M (F'_t - G'_t) - \min_M (F'_t - G'_t))\,dt.\end{equation}
 Moreover for all $f,g \in \Ham(M,\om),$ \[d'_{\mrm{bottle}} (\cB'(f),\cB'(g)) \leq d_{\mrm{Hofer}}(f,g),\] where \[d_{\mrm{Hofer}}(f,g) = \inf_{\substack{\phi^1_F=f \\ \phi^1_G = g}} d_{\mrm{Hofer}}([F],[G])\] is the bi-invariant Hofer metric on $\Ham(M,\om)$ introduced in \cite{HoferMetric} and proved to be non-degenerate in \cite{Viterbo-specGF,Polterovich-isotopy,Lalonde-McDuff-Energy}. Moreover, one can replace $\cB'(-)$ by either $\cB'_0(-)$ or $\cB'_1(-).$ Similarly, one shows following closely the absolute case \[d_{\mrm{inter}} (V_r(L,F),V_r(L,G)) \leq d_{\mrm{Hofer}} ([F],[G]),\] where $F,G \in \cH$ have zero mean. In fact, reinterpreting the Floer complex $(L,H)$ for $H \in \cH$ up to a shift in action, via the Lagrangian Floer complex of $(L,\phi^1_H(L)),$ and replacing the Floer continuation maps by suitable Floer equations with moving boundary conditions (cf. \cite{Oh-Turkish} and \cite{OhBook,SeidelBook} and the references therein), one can show that:
\[d'_{\mrm{inter}} (V'_r(L,F),V'_r(L,G)) \leq d_{\mrm{Hofer}} ([F]\cdot L,[G] \cdot L),\] where for $H \in \cH$ we denote by $[H]\cdot L$ the class of $\{\phi^t_H(L)\}_{t \in [0,1]}$ in the universal cover $\til{\cO}_L$ of the orbit $\cO_L = \Ham(M,\om) \cdot L$ of $L$ under the action of the Hamiltonian group. The pseudo-metric is given by \[d_{\mrm {Hofer}}([F]\cdot L,[G]\cdot L) = \int_0^1 (\max_L H(t,-) - \min_L H(t,-)) \, dt,\] with the infimum now running over $H\in \cH$ with $[H]\cdot L = [F]^{-1}[G] \cdot L.$ We shall show this, as well as the following bound in a different way, when $L \subset M$ is wide, as a consequence of Theorem \ref{theorem: Lipschitz wide} and Inequality \eqref{eq: gamma and Hofer}. Finally, \[d'_{\mrm{bottle}} (\cB'(L,f \cdot L),\cB'(L,g \cdot L)) \leq d_{\mrm{Hofer}} (f \cdot L, g \cdot L),\] the latter being the Lagrangian Hofer metric \cite{Chekanov,ChekanovFinsler}, defined as \[d_{\mrm{Hofer}} (f \cdot L, g \cdot L) =  \inf_{\substack{ \phi^1_F L = f L,\\ \phi^1_G L = g L}} d_{\mrm{Hofer}}([F]\cdot L, [G]\cdot L).\] Equivalently \cite{Usher-trick}, for $L' = \phi^1_F(L) \in \cO_L,$ we have \begin{align*}d_{\mrm {Hofer}}(L,L') & = \inf \int_0^1 (\max_L H(t,-) - \min_L H(t,-)) \, dt, \\ & = \inf \int_0^1 (\max_{\phi^t_H(L)} H(t,-) - \min_{\phi^t_H(L)} H(t,-)) \, dt, \\ &= \inf \int_0^1 (\max_M H(t,-) - \min_M H(t,-)) \, dt,\end{align*} the infimum running over all $H \in \cH$ with $\phi^1_H(L) = L'.$

We call the maximal length of a finite bar in $\cB'(L,\phi(L)),$ respectively $\cB'(\phi),$ the {\em boundary depth} $\beta(L,\phi(L)) = \beta(L,\phi^1_H) = \beta(L,H),$ respectively $\beta(\phi) = \beta(H).$ It was introduced and studied by Usher in \cite{UsherBD1,UsherBD2} in a different way, and proven to depend only on $\phi^1 L,$ resp. $\phi^1_H$ in \cite[Theorem 1.7, Section 6]{UsherBD2}. 

Finally, by the PSS isomorphisms in Floer theory, $V_*(H)^{\infty} = \lim_{s \to \infty} V_*(H)^s \cong QH_*(M; \Lambda)$ the quantum homology of $M,$ with a suitable choice $\Lambda$ of Novikov coefficients, and similarly $V_*(L,H)^{\infty} = \lim_{s \to \infty} V_*(L,H)^s \cong QH_*(L; \Lambda),$ the Lagrangian quantum homology, as introduced in \cite{BiranCorneaLagrangianQuantumHomology,Bi-Co:qrel-long} with suitable coefficients $\Lambda.$ Considering a non-zero class $a \in QH(M;\Lambda),$ resp. $a \in QH(L;\Lambda)$ we obtain the numbers $c(a,H),$ and $c(L; a,H)$ by looking at the infimum of levels $s \in \R$ with $a$ being in the image of $V^s(H) \to QH(M;\Lambda),$ resp. in the image of $V^s(L,H) \to QH(L;\Lambda).$ Spectral invariants enjoy various remarkable properties, and allow one to prove many interesting results. They were first introduced in symplectic topology by Viterbo in \cite{Viterbo-specGF} by means of generating functions, then in Floer theory by Schwarz \cite{Schwarz:action-spectrum}, and Oh \cite{Oh-construction,Oh-specnorm}, and in Lagrangian Floer theory by Leclercq \cite{Leclercq-spectral} and Leclercq-Zapolsky \cite{LeclercqZapolsky} (see also \cite{Oh-spec-lagr,MonznerVicheryZapolsky,FO3-spec}). We refer to \cite{LeclercqZapolsky} for a review of the literature. 

In terms of Lagrangian (resp. Hamiltonian) spectral invariants, the relative, resp. absolute, spectral pseudo-norms are given for $H \in \cH$ by \[\gamma(L,H) = c(L; [L],H) + c(L; [L],\overline{H}),\] \[\gamma(H) = c([M],H) + c([M],\overline{H}).\] We will omit $L$ in the notation for Lagrangian spectral invariants, once it is clear which Lagrangian submanifolds are being discussed. Recall from \cite{BiranCorneaRigidityUniruling} that a weakly monotone Lagrangian submanifold $L \subset M$ is called {\em wide} if $QH(L,\Lambda) \cong H_*(L,\Lambda).$ Note that weakly exact Lagrangians are always wide. We refer to Section \ref{Sect:prelim} for further preliminaries related to Floer homology.

It has long been known \cite{EntovPolterovichCalabiQM}, via Lemma \ref{lma: Hamiltonian geq Lagrangian}, that \begin{align}\label{eq:bound of gamma on CP^n}\gamma(H) \leq A,\\ \notag \gamma(L,H) \leq A,\end{align} for all Hamiltonians $H$ on $\C P^n,$ and all wide weakly monotone Lagrangian submanifolds $L \subset \C P^n,$ where $A = \left< [\om_{FS}] , [\C P^1] \right>.$ Usher \cite{Usher-private} has asked whether the boundary depth $\beta(H)$ and $\beta(L,H)$ is similarly uniformly bounded in this case.

We give a preview of our results in the following statement, that answers the question of Usher. This statement is expanded and refined in Sections \ref{sec:intro-continuation} and \ref{sec:filteredYoneda}.

\begin{thm}
	\label{thmSpectralNormBoundsBoundaryDepth}
	Let $L \subset M$ be wide. Then for each Hamiltonian $H \in \cH$,
	\[ \beta(L,H) \leq \gamma(L,H) .\]
	\[ \beta(H) \leq \gamma(H) .\]
\end{thm}

\begin{rmk}

We note that it follows from \cite{LeclercqZapolsky} that $\gamma(L,H)$ depends only on the class $[H] \cdot L \in \til{\cO}_L,$ hence we write $\gamma(L,H) = \gamma(L,[H]) = \gamma(L, [H]\cdot L).$ In the absolute case, we get the number $\gamma(H) = \gamma([H]).$
	
Theorem \ref{thmSpectralNormBoundsBoundaryDepth} therefore implies in the relative case that if $L$ is wide, then \[\beta(L, \phi^1_H) \leq \gamma(L, \phi^1_H),\] where the latter is defined as follows. Denote by $\Ham_L$ the group of Hamiltonian diffeomorphisms $f$ of $M$ that satisfy $f(L) = L$. Let $\Gamma_L$ be the space of paths $\eta : [0,1] \to \Ham(M)$ such that $\eta^0 = \Id$ and $\eta^1 \in \Ham_L$. Let $\mathcal{P}_L = \pi_0(\Gamma_L)$. Then \[\gamma(L, \phi^1_H) = \inf_{\eta \in \cP_L} \gamma(L,\eta \cdot  [H]) \geq 0.\]
Of course the infimum in the right hand side is the same as \[\inf_{F} \gamma(L,[F])\] running over all $F \in \cH$ with $\phi^1_F(L) = \phi^1_H(L).$ In other words we minimize $\gamma(L,-)$ over all elements of $\til{\cO}_L $ that project to $\phi^1_H(L)$ under the natural covering map $\til{\cO}_L \to \cO_L.$
	
It is well-known that the absolute analogue $\gamma(\phi^1_H) = \inf_{F} \gamma([F]),$ over all $F \in \cH$ with $\phi^1_F = \phi^1_H,$ is non-degenerate, and hence defines a norm on $\Ham(M,\om),$ called the spectral norm \cite[Theorem A(1)]{Oh-specnorm} (see also \cite[Theorem 12.4.4]{McDuffSalamonBIG},\cite[Remark 2.2]{Usher-Sharp}). The Lagrangian spectral norm $\gamma(L,-)$ was introduced in \cite{Viterbo-specGF} for $L = 0_{Q},$ the zero section in the cotangent bundle $M = T^*Q,$ wherein its non-degeneracy was proven (see also \cite{MonznerVicheryZapolsky,LisiRieser,HLS-coisotropic}). Since the introduction of Lagrangian spectral invariants \cite{Leclercq-spectral,LeclercqZapolsky} for weakly monotone closed Lagrangian submanifolds with non-vanishing Floer homology (easily extended to the case of open manifolds tame at infinity or compact with convex boundary), the notion of Lagrangian spectral norm was extended to these cases, and its non-degeneracy has remained a mystery. We settle it in Theorem \ref{thm:spectral_norm non-deg} by proving a strong version of a relative energy capacity inequality.
	
Finally, we note that by the Lagrangian control property of Lagrangian spectral invariants, similarly to the absolute case, $\gamma(L,-)$ is bounded from above by the Lagrangian Hofer norm: \begin{equation}\label{eq: gamma and Hofer}\gamma(L,L') \leq d_{\mathrm{Hofer}}(L,L').\end{equation}

\end{rmk}

\subsection{Filtered continuation elements and their applications}\label{sec:intro-continuation}

We first recall the idea of the proof of Inequality \eqref{eq:bottle-Hofer}. Given $F,G \in \cH$ non-degenerate with zero mean, the Floer continuation maps $C(F,G): CF(F;\cD) \to CF(G;\cD),$ $C(G,F): CF(G;\cD) \to CF(F;\cD),$ with $\cD$ suitable perturbation data with zero Hamiltonian part, yield $\delta$-interleavings between the persistence modules $V_r(F)$ and $V_r(G)$ for $\delta = \int_0^1 (\max_M (F_t - G_t) - \min_M (F_t - G_t) )\, dt.$ Since $V_r(F), V_r(G)$ only depend on $[F],[G]$ up to isomorphism, this finishes the proof, by Equation \eqref{eq:Hofer norm linear}.

Following \cite{ShelukhinHZ}, inspired by \cite{BiranCorneaS-Fukaya, AK-simplehomotopy}, we observe that the continuation map $C(F,G)$ is chain homotopic to the multiplication operator $m_2(x,-): CF(F;\cD) \to CF(G;\cD)$ with a cycle $x \in CF(G \# \overline{F}; \cD)$ representing $PSS_{G \# \bar{F}; \cD}([M]).$ Now, instead of the Hofer norm, it is the minimal filtration levels, in other words, spectral invariants $c([M],G \# \bar{F}),$ $c([M], F \# \bar{G})$ that govern the shifts in filtrations of the maps, which allows to connect the spectral norm to the interleaving distance. Similar observations apply in the relative case. The method of filtered continuation elements hence consists of {\em replacing continuation maps by multiplication operators} with certain cycles of minimal filtration level in their homology class. We provide two main applications of this method. The homology classes of these cycles are called continuation elements in the literature.

\subsubsection{Interleaving distance up to shift, and spectral norm.}

The first application of this method  whose proof is described in Section \ref{sec:filteredYoneda}, is the following result. Theorem  \ref{thmSpectralNormBoundsBoundaryDepth} is its immediate consequence.

\begin{thm}\label{theorem: Lipschitz wide}
	Let $L$ in $M$ be wide. Then the barcode of $(L,H),$ as a function of $H,$ is Lipschitz in the spectral norm up to $\R$-shifts: for each $r \in \Z,$ $F\in \cH, \;G \in \cH$ there exists $c \in \R,$ such that
	\[d_{\mrm{inter}}(V_r(L,F),V_r(L,G)[c]) \leq \frac{1}{2} \gamma(L,G \# \overline{F}).\]  In fact $c = - \frac{1}{2} (- c([L],G \# \overline{F}) + c([L], F \# \overline{G})).$
\end{thm}

\medskip

Theorem \ref{theorem: Lipschitz wide} has the following corollary on the absolute Hamiltonian case. This is immediate from \cite[Sections 2.7, 4.2.2]{LeclercqZapolsky}.

\begin{thm}\label{theorem: Lipschitz Hamiltonian}
	Let $M$ be a weakly monotone symplectic manifold. Then for each $r \in \Z,$ $F\in \cH, \;G \in \cH$ there exists $c \in \R,$ such that \[d_{\mrm{inter}}(V_r(F),V_r(G)[c]) \leq \frac{1}{2} \gamma(G \# \overline{F}).\] In fact $c = - \frac{1}{2} (- c([M],G \# \overline{F}) + c([M], F \# \overline{G})).$
\end{thm}

\begin{rmk}
	In fact Theorem \ref{theorem: Lipschitz Hamiltonian} admits a direct proof \cite{ShelukhinHZ} that is somewhat shorter, since, in the absolute setup, Proposition \ref{prop: beta 0 for wide} is immediate.
\end{rmk}

\begin{rmk}\label{Remark:idempotents}
	If $L$ is wide, and in its quantum homology $QH(L;\Lambda_{\min})$ the unit $[L]$ decomposes as $[L] = e_1 + \ldots + e_N$ for idempotents $E = \{e_j\}$ in $QH_n(L)$, satisfying $e_j \ast e_k = \delta_{j,k} e_j,$ then a similar bound is attained with $\gamma_{E}(L,H) = \frac{1}{2} \max \{c(e_j; L, H) + c(e_j; L, \overline{H})\}_{1 \leq j \leq N} + R(E),$ with error term $R(E) = \max\{A(e_j)\}_{1 \leq j \leq N},$ in the right hand side. See \cite{ShelukhinHZ} for details in the absolute case: the relative case is proven analogously. Since in this paper we focus on sharp bounds, which seem unlikely to be attained with splittings into $N \geq 2$ idempotents, because of the error term $R(E) \geq A_L$, we do not discuss the details of this generalization here, deferring them elsewhere. 
\end{rmk}

\medskip

We note that Theorem \ref{thmSpectralNormBoundsBoundaryDepth} is a direct corollary of Theorems \ref{theorem: Lipschitz wide} and \ref{theorem: Lipschitz Hamiltonian}. Moreover, it seems fit to note that the two spectral norms $\gamma(H),\; \gamma(L; H)$ for $H \in \cH$ featured above are related as follows.

\medskip

\begin{lma}\label{lma: Hamiltonian geq Lagrangian}
	Let $L$ in $M$ be monotone, with $N_L \geq 2.$ Then for each $H \in \cH,$ \[\gamma(L,H) \leq \gamma(H).\]
\end{lma}

\begin{proof}
	Using Proposition 4 in \cite{LeclercqZapolsky}, we have the following relation between the relative and absolute spectral invariants.
	\[ c(L;[M] \ast [L],H) \leq c([M],H) + c(L;[L],0) = c([M],H).\]
	Since $[M]\ast [L] = [L]$ (see \cite[Theorem 2.5.2]{BiranCorneaLagrangianQuantumHomology}), we get
	\[ c(L;[L],H) \leq c([M],H),\]
	and hence
	\[ \gamma(L,H) = c(L;[L],H) + c(L;[L],\overline{H}) \leq c([M],H) + c([M],\overline{H}) = \gamma(H) .\]
\end{proof}

\begin{rmk}
	Lemma \ref{lma: Hamiltonian geq Lagrangian} can be generalized to show that for an idempotent $\tau \in QH_{2n}(M,\Lambda), \, \tau^2 = [M]$ with $\tau \ast [L] = [L],$ \[\gamma(L, H) \leq \gamma_\tau(H) = c(\tau,H) + c(\tau,\overline{H}).\]  This leads to additional uniform bounds on the Lagrangian spectral norm, which we have found, however, to be less sharp than the ones given by Proposition \ref{Prop:slimLagr}.
\end{rmk}

\medskip

We briefly discuss applications to $C^0$ symplectic topology. Consider a distance function $d$ on a closed symplectic manifold $(M,\omega),$ coming from an auxiliary Riemannian metric. Define the $C^0$ metric on $\Ham(M,\om)$ by \[d_{C^0}(f,g) = \min_{x\in M} d(f(x),g(x)).\] The topology this metric induces is independent of the Riemannian metric, and is called the $C^0$ topology. The completion  $\overline{\Ham}(M,\om)$ of $\Ham(M,\om)$ with respect to this topology is called the group of Hamiltonian homeomorphisms of $(M,\om).$ Theorem \ref{theorem: Lipschitz Hamiltonian} and Lemma \ref{lma: Hamiltonian geq Lagrangian} have the following immediate corollary:

\begin{cor}\label{Corollary: C^0}
	If the spectral norm $\gamma:\Ham(M,\omega) \to \R_{\geq 0}$ on a weakly monotone symplectic manifold is $C^0$-continuous so are the maps $(\Ham(M,\om),d_{C^0}) \to (\barc',d'_{bottle}),$ $\phi \mapsto \cB'(\phi),$ $\phi \mapsto \cB'(L,\phi),$ for each weakly monotone closed Lagrangian submanifold $L$ of $M.$ In the weakly aspherical resp. weakly exact case, $\cB'(\phi),$ resp. $\cB'(L,\phi)$ stand for the degree $r \in \Z$ barcodes $\cB_r(\phi), \cB_r(L,\phi),$ for each $r \in \Z,$ considered up to shifts. In the monotone case, $\cB'(\phi), \cB'(L,\phi)$ stand for the degree $0$ or $1,$ resp. degree $0$ barcodes with coefficients respectively in $\Lambda_{M,\tmon},\Lambda_{L,\tmon}$, with respective quantum variables $s$ of degree $2$ and $t$ of degree $1$, considered up to shifts. In particular these maps extend to $\overline{\Ham}(M,\om),$ taking values in $\overline{\barc}'.$
\end{cor}

\begin{rmk}
In view of \cite[Proposition 5.3, Proposition 6.2]{UsherBD2}
(see also Lemma \ref{propActionFormula}, Equation \eqref{eq:t mult iso} and the index calculation \cite[p. 245]{PolShe}) it is easily seen that $\cB'(\phi), \cB'(L,\phi)$ indeed depend only on $\phi,$ resp. $\phi(L).$
\end{rmk}

\begin{rmk}
	Corollary \ref{Corollary: C^0} immediately implies the $C^0$ continuity of barcodes of Hamiltonian diffeomorphisms of surfaces $(\Sigma,\sigma)$ by work of Seyfaddini \cite{SeyfaddiniC0Limits}. The paper \cite{LSV-conj} proves this result in a different way, and computes examples of barcodes of Hamiltonian homeomorphisms, allowing one to distinguish different conjugacy classes in $\overline{\Ham}(\Sigma,\sigma).$ This shows as a corollary that there exist Hamiltonian homeomorphisms on surfaces that are not conjugate to any Hamiltonian diffeomorphism. Moreover, as pointed out to us by S. Seyfaddini, it is shown in \cite{LSV-conj} that the map $(\Ham(\Sigma,\om),d_{C^0}) \to (\barc,d_{\mrm{bottle}}),$ where the barcodes are not considered up to shift, is $C^0$ discontinuous. This implies that the non-zero, in general, shift $c$ in Theorem \ref{theorem: Lipschitz Hamiltonian} (and hence in Theorem \ref{theorem: Lipschitz wide}) cannot be removed.
\end{rmk}

\begin{rmk}
	In \cite{BHS-spectrum} and \cite{ES-Viterbo}, the spectral norm is proven to be $C^0$-continuous, and Corollary \ref{Corollary: C^0} is consequently used to establish the $C^0$-continuity of barcodes of Hamiltonian diffeomorphisms, in the case of closed symplectically aspherical symplectic manifolds and of $(\C P^n, \om_{FS}),$ respectively, with applications to conjugacy classes in the group of Hamiltonian homeomorphisms. 
\end{rmk}

\subsubsection{Chekanov type theorems for the spectral norm}

Further, using filtered continuation elements, we sharpen two results of Chekanov \cite{Chekanov,ChekanovFinsler} as follows. We start with the following definitions.

\begin{df}\label{def: hbar}
	Let $(M,\omega)$ be a closed symplectic manifold and $L \subset M$ a closed connected Lagrangian submanifold. For an $\om$-compatible almost complex structure $J$ on $M,$ let $\hbar(J,L) > 0$ be the minimal area of a non-constant $J$-holomorphic disk on $L,$ or a non-constant $J$-holomorphic sphere in $M.$ If no such disks or curves exist, set $\hbar(J,L) = +\infty.$
\end{df}

\begin{rmk}\label{rmk: rational L}
	If $L$ is rational, that is $\omega_L (H^D_2(M,L;\Z)) = \rho \cdot \Z$ for $\rho > 0,$ then $\hbar(J,L) \geq \rho$ for all $J.$
\end{rmk}

\begin{df}\label{def: relative Gromov width}
	Let $(M,\omega)$ be a symplectic manifold, $L$ a Lagrangian submanifold of $M,$ and $A \subset M$ a subset. Then the relative Gromov width $w(L;A)$ of $L$ with respect to $A$ is defined as \[ w(L;A) = \sup\{\frac{\pi r^2}{2}\,|\, \exists \; e:B_r \to M, e^{-1}(L) = B_r \cap \R^n, e^{-1}(A) = \emptyset\}, \] where $B_r \subset \C^n$ is the standard ball of radius $r$ endowed with the standard symplectic form $\om_{st},$ the map $e:(B_r,\om_{st}) \to (M,\omega)$ is a symplectic embedding, and $\R^n \subset \C^n$ is the real part. Set the relative Gromov width of $L$ to be \[w(L):= w(L;\emptyset).\] As a matter of terminology, the symplectic embedding $e:(B_r,\om_{st}) \to (M,\omega)$ satisfying the condition $e^{-1}(L) = B_r \cap \R^n$ with respect to the Lagrangian submanifold $L \subset M,$ will be called a symplectic ball embedding {\em relative to $L$}.
\end{df}

\medskip

First, following Chekanov's argument for his non-displacement theorem \cite{Chekanov} (see also \cite{BarraudCorneaSerre,Char,CorneaS}) one has the following consequence of a Hamiltonian $H \in \cH$ having small spectral norm. Whenever it applies, this statement sharpens the previous results in this direction by Inequality \eqref{eq: gamma and Hofer}.

\medskip

\begin{thm}\label{thm:Chekanov}
	Let $(M,\omega)$ be weakly monotone, and $L \subset M$ a closed connected Lagrangian submanifold. If for an $\om$-compatible almost complex structure $J$ on $M,$ a Hamiltonian diffeomorphism $\phi$ satisfies $\gamma(\phi) < \hbar(J,L),$ then \[L \cap \phi(L) \neq \emptyset.\] If in addition $L$ and $\phi(L)$ intersect transversely, then \[\# (L \cap \phi(L)) \geq \dim_{\mathbb{F}_2}H_*(L; \mathbb{F}_2).\]
\end{thm}

\medskip

\begin{rmk}

	It was proven in \cite{GG-pseudorotations} that each Hamiltonian pseudo-rotation of $\C P^n,$ that is $\phi \in \Ham(\C P^n, \om_{FS})$ that has precisely $n+1$ periodic points (of all periods) satisfies the following property. There exists a constant $C>0,$ and integer $d\geq n$ depending on $\phi,$ such that for each $\epsilon >0,$ $\liminf_{k\to \infty} \frac{1}{k}\cdot\# \{1\leq l \leq k|\, \gamma(\phi^l) < \epsilon\}  \geq C \eps^d;$  moreover, for each subset $L \subset \C P^n$ with positive homological capacity, as defined {\em ibid.}, $c_{hom}(L)>0,$ the following Poincar\'{e} recurrence statement holds: $\liminf_{k\to \infty} \frac{1}{k}\cdot\# \{1\leq l \leq k|\, \phi(L) \cap L \neq \emptyset\} \geq C (c_{hom}(L))^d.$ Theorem \ref{thm:Chekanov} implies that for each closed Lagrangian submanifold $L$ of $\C P^n$  \[\liminf_{k\to \infty} \frac{1}{k}\cdot \# \{1\leq l \leq k|\, \phi^l(L) \cap L \neq \emptyset\}  \geq C \big(\hbar(J,L)\big)^d > 0,\] whereby there exists a sequence $k_j \in \Z, \displaystyle\lim_{j \to \infty} k_j = +\infty$ such that for all $j,$ $\phi^{k_j}(L) \cap L \neq \emptyset.$ This generalizes the Poincar\'{e} recurrence result from \cite{GG-pseudorotations} in the case of Lagrangian submanifolds, as it is not currently known that each closed Lagrangian submanifold of $\C P^n$ has positive homological capacity. Furthermore, if we assume the non-empty intersections $\phi^{k_j}(L) \cap L$ to be transverse, Theorem \ref{thm:Chekanov} gives the lower bound $\dim H_*(L;\mathbb{F}_2)$ on the multiplicity of the intersection.
\end{rmk}

We proceed to recall the Lagrangian spectral pseudo-distance, and a certain extrinsic version of this pseudo-distance.

\begin{df}
	Let $(M,\om)$ be a symplectic manifold, and $L \subset M$ be a closed connected Lagrangian submanifold, and let $L' \in \cO_L.$ 
	\begin{enumerate}
		\item If $M$ is closed weakly monotone we define \[\gamma_{ext}(L,L') = \inf \{\gamma([H])\,|\, \phi^1_H(L) = L'\}.\]  
		\item If $M$ is weakly monotone, closed, open tame at infinity, or compact with convex boundary, $L$ is closed weakly monotone, and $QH(L) \neq 0,$ recall that \[\gamma_{int}(L,L') = \gamma(L,L') =\inf \{\gamma(L,[H])\,|\, \phi^1_H(L) = L'\}. \]
	\end{enumerate}
	
	Both $\gamma_{ext}$ and $\gamma_{int},$ when they are defined, are invariant under shifts by Hamiltonian diffeomorphisms, and define pseudo-metrics $d_{\gamma,{ext}},$ $d_{\gamma,{int}}$ on the orbit $\cO_L = \Ham(M,\omega) \cdot L$ of $L$ under $\Ham(M,\omega).$
	
\end{df}

In this setting the following sharper version of Chekanov's non-degeneracy theorem \cite{ChekanovFinsler} (see also \cite{BarraudCorneaSerre,Char,CorneaS}) holds. It is indeed a sharpening by Inequality \ref{eq: gamma and Hofer}.

\medskip

\begin{thm}\label{thm:spectral_norm non-deg}
	The pseudo-metrics $d_{L,\gamma, ext}$ and $d_{L,\gamma,int},$ whenever defined, are non-degenerate, and hence define genuine $\Ham$-invariant metrics on $\cO_L.$ Moreover, for each $\epsilon > 0$ with $ \epsilon < w(L,L')$ there exists an $\om$-compatible almost complex structure $J_{\epsilon},$ such that \begin{align*}
	\gamma_{ext}(L,L') \geq \min\{w(L;L') - \epsilon, \hbar(J_{\epsilon},L), \hbar(J_{\epsilon},L') \} > 0,\\
	\gamma_{int}(L,L') \geq \min\{w(L;L') - \epsilon, \hbar(J_{\epsilon},L), \hbar(J_{\epsilon},L') \} > 0.
	\end{align*}
	
\end{thm}

We note that Theorem \ref{thm:spectral_norm non-deg} can be considered to be a new energy-capacity inequality for the Lagrangian spectral norm. Indeed, whenever $\omega_L(H_2^D(M,L;\Z)) \subset \rho \cdot \Z,$ for a constant $\rho \in (0,\infty],$ the lower bound in both cases takes the form\begin{align*}
\gamma_{ext}(L,L') \geq \min\{w(L,L'),\rho\},\\ 
\gamma_{int}(L,L') \geq \min\{w(L,L'),\rho\}. 
\end{align*}

Finally, combining the proof of Theorem \ref{theorem: Lipschitz wide}, and Theorem \ref{thm:spectral_norm non-deg}, we obtain the following alternative version of control on barcodes.

\begin{thm}\label{theorem: Lipschitz wide ext}
	Let $L$ in $M$ be weakly monotone. Then for each $r \in \Z,$ $F\in \cH, \;G \in \cH$ there exists $c \in \R,$ such that
	\[d_{\mrm{inter}}(V_r(L,F),V_r(L,G)[c]) \leq \frac{1}{2} \gamma(G \# \overline{F}).\]  In fact $c = - \frac{1}{2} (- c([M],G \# \overline{F}) + c([M], F \# \overline{G})).$
\end{thm}


\subsection{An averaging technique and sharp upper bounds on the spectral norm}

In this section we introduce a new technique for producing upper bounds on the spectral norm, and prove that in certain cases the upper bound on boundary depth that we obtain is sharp. Moreover, we find the unexpected fact that this bound is stricly smaller than the minimal area of a non-constant holomorphic disk with boundary on our given Lagrangian (or a holomorphic sphere in the Hamiltonian case)! 

It was observed in \cite{huLalonde,huLalondeLeclercq,SeidelBook,charetteCornea} that Hamiltonian paths with time-one map sending a Lagrangian submanifold $L$ to itself, induce automorphisms of the Floer homology of $L,$  similarly to the well-known Seidel representation \cite{seidelInvertibles} in Hamiltonian Floer theory. We use this notion to prove the following bound on the spectral norm. We refer to Section \ref{Sect:Seidel} for related notation and preliminaries.

\begin{prop}\label{Prop:slimLagr} 
	Let $L \subset M$ be wide. Set $n= \dim(L),$ and $A_L = \kappa N_L.$ Let $t$ be a monotone Novikov variable of degree $1.$  For an element $\eta \in \cP_L$ and section class $\sigma,$ let $S=S_{\eta,\sigma}([L])$ denote the associated Lagrangian Seidel element. Assume that 
	\begin{enumerate}
		\item \label{assumption 1} $N_L > n,$ and hence there is a well-defined class $[pt] \in QH_0(L),$
		\item \label{assumption 2} $S^k = c_1 \cdot [pt] t^n$ for some $k \geq 1,$ $c_1 \in \bK \setminus \{0\}$ and
		\item \label{assumption 3} $S^m = c_2 \cdot [L],$ for some $m > k,$ $c_2 \in \bK \setminus \{0\}.$ 
	\end{enumerate}
	
	Then for each Hamiltonian $H$ on $M,$ \[\beta(L,H) = \beta(L,\phi^1_H) \leq \gamma(L,\phi^1_H) \leq \frac{n}{N_L} \cdot A_L.\]
	In particular this bound is stricly smaller than $A_L.$

\end{prop}

We present the proof of this statement now, because it summarizes our new averaging method for bounding the spectral norm.

\begin{proof}[Proof of Proposition \ref{Prop:slimLagr}]
	Assumption \ref{assumption 1} implies by Poincar\'{e} duality for Lagrangian spectral invariants \cite{LeclercqZapolsky} and our choice of coefficients in $\mathbb{F}_2$  that for each $H \in \cH$ \[c([L],\overline{H}) = - c([pt],H).\] Indeed, since $N_L>n=\dim(L),$ and $L$ is wide, taking coefficients in the Novikov field $ \Lambda_{\tmin}$ with quantum variable of degree $N_L,$ we see that $[pt]$ is the unique degree $0$ element (up to $\bK\setminus \{0\}$ multiples). This yields by \cite[Theorem 27]{LeclercqZapolsky}, which we reproduce in Proposition \ref{prop:main_properties_Lagr_sp_invts} below, the desired inequality with coefficients in $\Lambda_{\tmin}.$ Therefore by \cite[Section 5.4]{BiranCorneaRigidityUniruling} the same equality holds for spectral invariants with coefficients in any field extension of $\Lambda_{\tmin},$ and in particular for coefficients in our Novikov field $\Lambda_{\tmon}.$ Strictly speaking, we should first work with non-degenerate pairs $(L,H),$ and then use continuity of spectral invariants \cite{LeclercqZapolsky} to extend this equality to all $H \in \cH.$ 
	
	Therefore \[\gamma(L,[H]) = c([L],H) - c([pt],H),\] where we denote by $[H]$ the class of the path $\{\phi^t_H\}_{t \in [0,1]}$ in $\til{\Ham}(M,\om).$ Consequently by Assumption \ref{assumption 2} we obtain \[\gamma(L,[H]) = c([L],H) - c(S^k,H) + \frac{n}{N_L} \cdot A_L.\] Therefore, by Corollary \ref{corSpecInv} below, for each $0 \leq j \leq m-1,$ we have \[\gamma(L,[H]\cdot\eta^{-j}) = c(S^j,H) - c(S^{k+j},H) + \frac{n}{N_L} \cdot A_L.\] Summing up these equalities, and using $S^m = c_2 \cdot [L],$ $c_2 \in \bK \setminus \{0\},$ we obtain \[\frac{1}{m} \sum_{j=0}^{m-1} \gamma(L,[H]\cdot \eta^{-j}) = \frac{n}{N_L} \cdot A_L.\] Since the minimum of a finite set does not exceed its average, we get \[\gamma(L,\phi^1_H) \leq \min_{0 \leq j \leq m-1} \gamma(L,[H]\cdot \eta^{-j}) \leq \frac{n}{N_L} \cdot A_L,\] and hence by Theorem \ref{thmSpectralNormBoundsBoundaryDepth} \[\beta(L,\phi^1_H) \leq \frac{n}{N_L} \cdot A_L.\]
	
\end{proof}

At this point we introduce the following two invariants measuring the best possible uniform bound on $\beta(L,-),$ and $\gamma(L,-):$ \[\overline{\beta}(M,L) := \sup_{H \in \cH} \beta(L,H),\] \[\overline{\gamma}(M,L) := \sup_{H \in \cH} \gamma(L,\phi^1_H).\] Similarly, in the absolute case, we define
\[\overline{\beta}(M) = \sup_{H \in \cH} \beta(H),\]
\[\overline{\gamma}(M) = \sup_{H \in \cH} \gamma(\phi^1_H).\]

\begin{rmk}\label{rmk:Usher-question-details}
	In \cite{UsherBD2} it was shown that $\overline{\beta}(M,L) = +\infty$ in a large collection of cases when $L$ is weakly monotone and $HF(L,L) \neq 0.$ However, no examples with $\overline{\beta}(M,L) < +\infty$ were known to date. In fact, Usher \cite{Usher-private} posed the question of whether $\overline{\beta}(\C P^1, \R P^1)$ is finite or not.
\end{rmk}

As a corollary of Proposition \ref{Prop:slimLagr}, using variations on the computations of the Lagrangian Seidel element in \cite{hyvrier}, we obtain uniform bounds on $\gamma(L,\phi^1_H)$ and $\beta(\phi^1_H)$ in the following four cases. Moreover, adapting arguments of Stevenson \cite{stevenson}, we show that in some of these cases, our bounds are the best ones possible, uniform in $H$. This provides a class of examples with finite $\overline{\beta}(M,L)$ and provides a more precise answer to Usher's question.

\begin{thm} \label{thm:slimLagrExamples}
Let the ground field be $\bK = \mathbb{F}_2.$ Then the following identities hold:
	\begin{enumerate}
 \item\label{thmBDbounds-case:RPn} $\displaystyle{\overline{\beta}(\R P^1, \C P^1) = \overline{\gamma}(\R P^n, \C P^n) = \frac{1}{4} \omega([\mathbb{C}{P}^1]),}$ \\
 $\displaystyle{\frac{1}{4}\omega([\mathbb{C}{P}^1]) \leq \overline{\beta}(\R P^n, \C P^n) \leq \overline{\gamma}(\R P^n, \C P^n) \leq \frac{n}{2n+2} \omega([\mathbb{C}{P}^1]),}$ for $n>1$ 
 \medskip
 \item\label{thmBDbounds-case:CPn} $\displaystyle{\overline{\beta}(\Delta_{\C P^n}, \C P^n \times (\C P^n)^{-}) = \overline{\gamma}(\Delta_{\C P^n}, \C P^n \times (\C P^n)^{-}) = \frac{n}{n+1} \omega([\mathbb{C}{P}^1]),}$ \\ $\displaystyle{\overline{\beta}(\C P^n) = \overline{\gamma}(\C P^n) = \frac{n}{n+1} \omega([\mathbb{C}{P}^1]),}$
\medskip 
 \item\label{thmBDbounds-case:Sn} $\displaystyle{\overline{\beta}(L, Q^n) = \overline{\gamma}(L, Q^n) = \frac{1}{2} \omega([\mathbb{C}{P}^1]),}$ for $n>1,$  
 
 \smallskip 
 \noindent where $Q^n \subset \C P^{n+1},$ is the complex $n$-dimensional smooth quadric (of real dimension $2n$), $L \cong S^n$ is a natural monotone Lagrangian sphere in $Q^n,$ and $\omega= \om_{FS}$ on $\C P^{n+1},$ 
\medskip
 \item\label{thmBDbounds-case:HPn} $\displaystyle{\frac{1}{2} \omega([\C P^1]) \leq \overline{\beta}(\bH P^n, Gr(2,2n+2)) \leq  \overline{\gamma}(\bH P^n, Gr(2,2n+2)) \leq \frac{n}{n+1} \omega([\mathbb{C}{P}^1]),}$ for $n \geq 1,$ 
 
 \smallskip 
 \noindent where $\omega$ is the standard symplectic form on $Gr(2,2n+2),$ the complex Grassmannian of complex $2$-planes in $\C^{2n+2}$, and $L = \bH P^n$ is the quaternionic projective space of real dimension $4n,$ embedded as the fixed point set of the involution on $Gr(2,2n+2)$ induced by multiplying $\C^{2n+2} \cong \bH^{n+1}$ on the right by the quaternionic root of unity $j.$ Here $\mathbb{C}{P}^1$ is a complex line in $Gr(2,2n+2)$ such that $[\mathbb{C}{P}^1]$ is a generator of $H_2(Gr(2,2n+2);\Z).$
	\end{enumerate}
\end{thm}

\medskip

\begin{rmk}
	It would be interesting to investigate the sharpness in Cases (\ref{thmBDbounds-case:RPn}), $n>1,$ and (\ref{thmBDbounds-case:HPn}) of Theorem \ref{thm:slimLagrExamples}.
\end{rmk}

\subsection{Calabi quasimorphisms with precise bounds on the defect}

We recall the following basic definition.

\begin{df}
	A quasimorphism $\sigma:G \to \R$ on a group $G$ is a function satisfying the bound \[D_{\sigma} = \sup_{x,y \in G}|\sigma(xy) - \sigma(x) - \sigma(y)| < \infty.\]
\end{df}

For each quasimorphism $\sigma$ there exists a unique {\em homogeneous}, that is additive on each abelian subgroup of $G,$ quasimorphism $\overline{\sigma},$ such that $\sigma - \overline{\sigma}$ is a bounded function. This homogeneization is given by the formula \[\overline{\sigma}(x) = \lim_{k \to \infty} \frac{\sigma(x^k)}{k}.\]  

Homogeneous quasimorphisms that are non-trivial, that is they are not homomorphisms, in general reflect non-commutative features of group and are an important notion of study in geometric group theory (see \cite{scl}). 

Quasimorphisms on the (universal cover of) the Hamiltonian group of closed symplectic manifolds were constructed in many cases, starting with the work \cite{EntovPolterovichCalabiQM} (see \cite{Entov-ICM} for a review of the literature). In particular, it was shown in \cite{EntovPolterovichCalabiQM} that $\til{\Ham}(\C P^n, \om_{FS})$ admits a non-trivial homogeneous quasimorphism, with the additional property that its composition with the natural map $\til{\Ham}_c(U,\om_{FS}|_U) \to \til{\Ham}(\C P^n, \om_{FS})$ coincides, for each displaceable open subset $U \subset \C P^n,$ with a non-zero constant multiple of the Calabi homomorphism \[\mrm{Cal}:\til{\Ham}_c(U,\om_{FS}|_U) \to \R\]\[ [H] \mapsto \int_0^1 \int_U H(t,x) \,(\om_{FS}|_U)^n \,dt,\] with the Hamiltonian normalized to have $\supp H \subset [0,1] \times U.$  In fact, more is true: this quasimorphism vanishes on $\pi_1(\Ham (\C P^n, \om_{FS})),$ and hence descends to a homogeneous quasimorphism $\Ham(\C P^n, \om_{FS}) \to \R.$

As a corollary of Case \ref{thmBDbounds-case:CPn} of Theorem \ref{thm:slimLagrExamples}, we improve the best known bound $D_{\sigma} \leq \om([\C P^1])$ \cite[Chapter 3]{PolterovichRosen} on the defect of the Calabi quasimorphism on $\tHam(\C P^n).$ However, we do not know whether or not it is optimal.

\begin{cor}
	\label{thmCalabiQMBound}
	Let $M = \C P^n$, and let $\sigma : \tHam(M,\omega) \to \R$ be the Calabi quasimorphism. The defect of $\sigma$ satisfies \[D_{\sigma} \leq \frac{n}{n+1} \omega([\CP^1]).\]
	
\end{cor}

\begin{proof}
    Since $[H] \to c([M],[H])$ is a conjugation invariant function $\tHam(M,\omega) \to \R,$ by \cite[Proposition 3.5.3]{PolterovichRosen} we obtain 
    \[|\sigma([G] [H]) - \sigma([G]) - \sigma([H]) | \leq \min\{\gamma([G]),\gamma([H])\},\] for all $[G],[H] \in \tHam(M,\omega).$
    Moreover, by \cite[Section 4.3]{EntovPolterovichCalabiQM}, $\sigma$ vanishes on $\pi_1(\tHam(M,\om)) \subset Z(\tHam(M,\om)).$ Hence, since $\sigma$ is a homogeneous quasimorphism, $\sigma([H])$ depends  only on $\phi^1_H.$  Hence for all $[G],[H] \in \tHam(M,\omega),$ \[|\sigma([G] [H]) - \sigma([G]) - \sigma([H]) | \leq \min\{\gamma(\phi^1_G),\gamma(\phi^1_H)\} \leq \frac{n}{n+1} \omega([\mathbb{C}\mathbb{P}^1]),\] by Theorem \ref{thm:slimLagrExamples}.
\end{proof}

In a similar direction, Theorem \ref{thm:slimLagrExamples} yields the following corollary on the existence of quasimorphisms with precise bounds on the defect.

\begin{cor}\label{Cor: quasi}
Let $(M,\om)$ be one of $\C P^{n} \times (\C P^{n})^{-},$ $ Gr(2,2n+2),$ $ Q^n.$ Then the universal cover  $\til{\Ham}(M,\om)$ admits a non-zero  quasimorphism $\sigma$ of defect bounded by $\frac{2n}{n+1} \om([\C P^1]),$ $ \frac{2n}{n+1} \om([\C P^1]),$ $\om([\C P^1])$ respectively. This quasimorphism satisfies the Calabi property: for $H \in \cH$ with $\supp H \subset [0,1] \times U,$ where $U \subset M$ is a displaceable open set, one has \[\sigma([H]) = - \frac{1}{\vol(M,\om)}\int_0^1\int_U H(x,t)\, \om^n \,dt.\] 
\end{cor}

\begin{proof}
The function $\sigma_0([H]) = c([L],H),$ for $H$ normalized to have zero mean, on $\til{\Ham}(M,\om)$ is a quasimorphism of defect $D$ at most $\frac{n}{n+1} \om([\C P^1]), \frac{n}{n+1} \om([\C P^1]), \frac{1}{2}\om([\C P^1])$ by Theorem \ref{thm:slimLagrExamples}, and the triangle inequality for Lagrangian spectral invariants. Indeed, it is immediate that \[c([L],F \#G) \leq c([L],F) + c([L],G),\] 
\[c([L],F) \leq c([L],F\#G) + c([L],\overline{G}),\] 
so that \[c([L],F) + c([L],G) - \gamma(L,[G]) \leq c([L],F \# G),\] whence the statement follows. Therefore the homogenization $\sigma$ of $\sigma_0$ has defect at most $2D,$ by \cite[Corollary 2.59]{scl}. The quasimorphism $\sigma$ is non-zero. Indeed by the Lagrangian control property of the spectral invariant $c([L],-),$ for a mean-normalized Hamiltonian $H$ that is  identically constant equal to $c$ on $L$ we have $c([L],H) = c,$ and hence $\sigma([H]) = c,$ since the constant property is invariant under iterations. We note that for the same reason $L$ is superheavy with respect to the corresponding symplectic quasi-state (see \cite{EntovPolterovich-rigid,LeclercqZapolsky}).
 
Finally, if $H = F - \left< F\right>,$ for a Hamiltonian $F$ supported in a displaceable open subset $U \subset M,$ and $\left<F\right>(t) = \frac{1}{\vol(M,\om)}\int_M F(x,t) \,\om^n,$ then arguing as in \cite[Proposition 3.3]{EntovPolterovichCalabiQM}, together with Lemma \ref{lma: Hamiltonian geq Lagrangian} we obtain for $K \in \cH$ is such that $\phi^1_K (U) \cap \overline{U} = \emptyset,$ that \[\frac{1}{m}c(L,[L],K \# H^{\# m}) \leq \frac{1}{m} c([M],K \# H^{\# m}) \xrightarrow{m \to \infty} - \left<\left< F\right>\right>,\] where $\left<\left< F\right>\right> = \int_{0}^{1} \left< F\right>(t)\,dt.$ As left hand side converges, by the quasimorphism property, to $\sigma([H]),$ we have \[\sigma([H]) \leq - \left<\left< F\right>\right>.\] Applying the same argument to $\overline{H},$ we have \[\sigma([H]^{-1}) \leq \left<\left< F\right>\right>,\] hence as $\sigma$ is homogeneous, we obtain the reverse inequality \[\sigma([H]) \geq - \left<\left< F\right>\right>.\] \end{proof}

\begin{rmk}
The existence of Calabi quasimorphisms for $\C P^{n} \times (\C P^{n})^{-}, Gr(2,2n+2)$ was known since \cite{EntovPolterovichCalabiQM}, and the one for $Q^n$ can be deduced from \cite[Remark after Theorem 3.1]{EntovPolterovich-semisimp} in a fairly straightforward way. However, the above upper bounds on the defect seem to be new.
\end{rmk}


\section*{Acknowledgements}
We thank Leonid Polterovich for bringing us together to work on this project, and for helpful conversations. We also thank Paul Biran, Octav Cornea, Viktor Ginzburg, Ba\c{s}ak G\"{u}rel, Pazit Haim-Kislev, Vincent Humiliere, Michael Khanevsky, Paul Seidel, Sobhan Seyfaddini, and Lara Simone Su\'{a}rez for their useful input. This work was initiated and was partially carried out during the stay of ES at the Institute for Advanced Study, where he was supported by NSF grant No. DMS-1128155. It was partially written during visits of ES to Tel Aviv University, and to Ruhr-Universit\"{a}t Bochum. He gratefully acknowledges the hospitality of these institutions and of Helmut Hofer,  Leonid Polterovich, and Alberto Abbondandolo. At the University of Montr\'{e}al, ES is supported by an NSERC Discovery Grant and by the Fonds de recherche du Qu\'{e}bec - Nature et technologies. AK is partially supported by European Research Council advanced grant 338809.

\section{Preliminaries}\label{Sect:prelim}

In this section we briefly recall the by-now standard definitions of various Floer complexes associated to Lagrangian submanifolds and Hamiltonian perturbations, maps relating them, basic algebraic structures, and associated numerical invariants: persistence modules and spectral invariants.

\subsection{Floer homology with Hamiltonian perturbations: absolute and relative}

In this paper we consider closed connected symplectic manifolds $(M,\omega).$ Certain arguments could be adapted to a setting of open symplectic manifolds tame at infinity, or compact with convex boundary. We work with Lagrangian submanifolds that we assume to be closed and connected. Moreover we assume that $(M,\omega)$ and $L$ are weakly monotone: there exists a constant $\kappa > 0,$ such that \[\left <[\omega],A\right >= \kappa \cdot \left <\mu_L, A \right >\] for all $A \in H^D_2(M,L; \Z),$ where $\mu_L \in H^2(M,L; \Z)$ is the Maslov class, and $H^D_2(M,L; \Z) \subset H_2(M,L; \Z)$ is the image of the Hurewicz homomorphism $\pi_2(M,L) \to H^D_2(M,L; \Z).$ It is easy to see that in this case $(M,\omega)$ is (spherically) weakly monotone: \[\left <[\omega],A\right > = 2\kappa \cdot \left < c_1(TM,\omega), A \right >\] for all $A \in H^S_2(M; \Z),$ the image of the Hurewicz map $\pi_2(M) \to H_2(M;\Z).$ Moreover, we assume that the positive generator $N_L \in \Z_{>0}$ of $\ima(\mu_L) \subset \Z,$ called the minimal Maslov number, satisfies \[N_L \geq 2.\] We call the respective positive generator $N_M \in \Z_{>0}$ of $\ima(c_1(TM,\omega)) \subset \Z$ the minimal Chern number. If $L \subset M$ is monotone, then $2 N_M = N_L.$ By a suitable rescaling of the symplectic form, we may assume that $\kappa = 1,$ which we shall do freely throughout the paper. Finally, define \[A_L = \kappa N_L >0.\]

If, for all $A \in H^D_2(M,L; \Z),$ \[\left<[\omega],A\right>= 0, \mu_L(A)= 0\]  we call $L$ weakly exact.  Similarly, we call a symplectic manifold for which \[\left <[\omega],A\right >= 0, \left <c_1(TM,\omega), A \right > = 0,\] for all $A \in H^S_2(M; \Z)$ symplectically aspherical. We shall call a Lagrangian, resp. a symplectic manifold, monotone if it is weakly monotone and not weakly exact, resp. monotone and not symplectically aspherical. Moreover, the above definitions apply to define Floer homology in the contractible class of paths in $M$ with boundary on $L$, resp. loops in $M$; if we wish to define Floer homology in a certain non-contractible class of loops, we can adapt those definitions to this class (cf. \cite{PolShe,PolSheSto}).

Fix a base field $\bK.$ We shall assume for this paper that $\bK = \mathbb{F}_2.$ However for part of our examples, including the case of absolute Hamiltonian Floer homology, one could in fact work with an arbitrary base field $\bK$, when a proper system of coherent or canonical (see \cite{Zap:Orient}, \cite{SeidelBook}, \cite{AbouzaidBook} and references therein) orientations has been set up. We shall discuss these cases when necessary.

\bs

{\em Lagragian Floer homology with Hamiltonian term:}

We shall define a few versions of Lagrangian Floer homology that differ primarily by the choice of Novikov rings. This material is quite standard, but since we use a different versions of Floer homology, we present it briefly in this section (we refer to \cite{OhBook},\cite{LeclercqZapolsky},\cite{UsherBD2} and references therein for details). We mostly restrict to the component of contractible chords in the suitable path space, since the fundamental class, when non-zero, is represeted by a Floer chain in this component. We now make our basic set up.

Let $\mathcal{H}_0$ denote the space of Hamiltonians $H \in \mathcal{H} = \sm{[0,1] \times M, \R}$ normalized by $\int_M H \, \omega^n = 0.$ We note that any Hamiltonian path $\{\phi^t\}_{t \in [0,1]}$ can be reparametrized in time to be constant for $t$ near $\{0\}\cup \{1\}.$ In this case its new generating Hamiltonian $H(t,x)$ will vanish for $t$ near $\{0\}\cup \{1\}.$ We use the convention that our Hamiltonian term $H$ is of this form, since it is useful for the construction of product structures.

Consider a free homotopy class $\eta \in \pi_0\cP(L,L)$ of paths in $M$ from $L$ to $L.$ Denote by $\cP_{\eta}(L,L)$ the corresponding connected component of the path space $\cP(L,L)$ in $M$ from $L$ to $L.$ We denote by $\eta = pt$ the component of a constant path in $L.$ Consider the set $\cO_{\eta}(L,H) \subset \cP_{\eta}(L,L)$ of chords $\gamma:[0,1] \to M$ of the Hamiltonian flow of $H$ from $L$ to $L.$ This means that $\gamma \in \cP_{\eta}(L,L)$ and \[\dot{\gamma}(t) = X^t_{H}(\gamma(t))\] for all $t \in [0,1],$ where $X^t_H$ is the time-dependent Hamiltonian vector field of $H$ defined by the Hamiltonian construction: for each $t \in [0,1]$ let $H_t \in \sm{M,\R}$ be defined by $H_t(x) = H(t,x),$ then $X^t_H$ is determined uniquely by \[\iota_{X^t_H}\omega = - dH_t.\] We denote by $\{ \phi^t_H \}_{t \in [0,1]}$ the flow on $M$ generated by $X^t_H,$ with $\phi^0_H = \id.$ Finally, we denote by $[H]$ the class $[\{ \phi^t_H \}_{t \in [0,1]}]$ of $\{ \phi^t_H \}_{t \in [0,1]}$ in the universal cover $\tHam(M,\omega)$ of $\Ham(M,\omega).$

We denote by $\cO(L,H) = \displaystyle{\sqcup_{\eta \in \pi_0\cP(L,L)}} \cO_{\eta}(L,H)$ the set of all Hamiltonian chords of $H$ from $L$ to $L.$ 

Given that all the intersections of $L$ with $\phi^1_H(L)$ are transverse, the Floer complex in each component $\eta$ will be given by a free module over a suitable coefficient ring, with basis isomorphic to $\cO_{\eta}(L,H),$ endowed with a differential given by counting solutions of the Floer equation connecting between different chords in the same free homotopy class. In the non-transverse case, we shall introduce a small additional Hamiltonian perturbation that makes the equation equivalent to the transverse case.

We define a coefficient ring \[\Lambda_{0,\text{univ},\bK} = \left \{ \displaystyle\sum_{j \geq 0} a_j \cdot T^{\lambda_j}\;:\;\;a_j \in \bK,\, 0 \leq \lambda_j \nearrow +\infty \right \}\] for a formal variable $T.$ This is the universal Novikov ring over $\bK.$ Its ring of fractions is the universal Novikov field over $\bK,$ given explicitly by \[\Lambda_{\text{univ},\bK} = \left \{ \displaystyle\sum_{j \geq 0} a_j \cdot T^{\lambda_j}\;:\;\;a_j \in \bK,\, \lambda_j \nearrow +\infty \right \}.\] In our monotone case, we let $t$ be a formal graded variable of degree $\deg(t) = 1$ and set \[\Lambda_{\text{mon},\bK} =\left \{\displaystyle\sum_{j \geq 0} a_j \cdot t^{-l_j}\;:\;\;a_j \in \bK,\, l_j \in \Z, l_j \nearrow +\infty \right\}.\] We note that $t \mapsto T^{-\kappa},$ where $\kappa = {A_L}/{N_L}>0,$ defines an ungraded embedding of fields $\iota_{1,0}: \Lambda_{\text{mon},\bK} \to \Lambda_{\text{univ},\bK}.$ Our different complexes shall be connected by this map. We record another coefficient ring, \[\Lambda_{\text{min},\bK} =\left \{\displaystyle\sum_{j \geq 0} a_j \cdot q^{-l_j}\;:\;\;a_j \in \bK,\, l_j \in \Z, l_j \nearrow +\infty \right\},\] where now $\deg(q) = N_L,$ and $q \mapsto t^{N_L}$ gives a finite graded extension of fields $\iota_{N_L,1}: \Lambda_{\text{min},\bK} \to \Lambda_{\text{mon},\bK}.$ When we wish to underline the fact that $\Lambda_{\text{min},\bK}$ is associated to $L,$ we write $\Lambda_{L,\text{min},\bK}.$

Finally we remark that for a general, not necessarily monotone, Lagrangian submanifold $L$ of a symplectic manifold $M,$ we may, in certain settings, define Floer homology with coefficients in the Novikov field $\Lambda_{L, \Gamma_{\omega},\bK},$ defined identically to the universal Novikov field $\Lambda_{\tuniv},$ with the exception that the exponents are taken in the subgroup $\Gamma_{\omega}$ of $\R$ given by $\Gamma_{\omega} = \ima(\om_L: H_2^D(M,L;\Z) \to \R).$ The Novikov ring $\Lambda_{0,L,\Gamma_{\omega},\bK}$ is defined similarly to $\Lambda_{0,\tuniv},$ and consists of those elements of $\Lambda_{L, \Gamma_{\omega},\bK}$ that have non-zero valuation.

We denote by $\cJ_M$ the space of $\omega$-compatible almost complex structures on $M$ depending smoothly on a parameter $t \in [0,1],$ and by $\cJ'_M$ its subspace of time-independent $\omega$-compatible almost complex structures. Now we fix a Floer perturbation datum $\mathcal{D}^{L,H} = (K^{L,H},J^{L,H})$ where $K^{L,H} \in \cH$ is a time-dependent Hamiltonian perturbation, and $J^{L,H} = \{J^{L,H}_t\}_{t \in [0,1]} \in \cJ_M$ is a time-dependent $\omega$-compatible almost complex structure on $M.$ When we  restrict our attention to the component $\eta,$ then in the case that all the intersections in $L \cap \phi^1_H(L)$ corresponding to endpoints of paths in $\cO_{\eta}(L,H)$ are transverse, we can assume that $K^{L,H} \equiv 0.$ Otherwise, we can take $K^{L,H}$ small in $C^2$-topology, and consider the Hamiltonian $H^{\cD} = H + K^{L,H}$ generating the flow $\{\phi^t_H \circ \phi^t_{\overline{K}^{L,H}} \}_{t \in [0,1]},$ for $\overline{K}^{L,H}(t,x) = K^{L,H} (t, \phi^t_H x)$ instead of $H,$ making all intersections in the class $\eta$ transverse. We denote by $\cO_{\eta}(L,H; \cD)$ the Hamiltonian chords of $H^{\cD}$ from $L$ to $L.$

Given $x_{-},x_{+} \in \cO_{\eta}(L,H; \cD),$ to any smooth map $u:Z \to M,$ where $Z = \R \times [0,1],$ such that $u|_{\{0\}\times \R}$ and $u|_{\{1\}\times \R}$ have values in $L,$ and $u(s,-) \xrightarrow{s \to \pm \infty} x_{\pm}$ uniformly in $C^1$-topology, one can associate an index $\ind(u,x_{-},x_{+}),$ called the Maslov-Viterbo index \cite{ViterboMaslov}. It is now standard that the space $\widehat{\mathcal{M}}(x_{-},x_{+})$ of solutions of the Floer equation \[(du - X^t_{H^{\cD}}(u) \otimes dt)^{(0,1)} = 0, \] or equivalently $\del_s u + J(u)(\del_t u - X^t_{H^{\cD}}(u)) = 0,$ for $(s,t)$ the natural coordinates on $Z = \R \times [0,1],$ with $u(s,-) \xrightarrow{s \to \pm \infty} x_{\pm}$ uniformly and $\ind(u,x_{-},x_{+}) = k$ is, for all $J^{L,H}$ in a co-meager set of $\cJ_M^{reg} \subset \cJ_M,$ a smooth manifold of dimension $k,$ admitting a free action of $\R.$ Moreover, for $k=1,$ the quotient space by the $\R$-action is a compact smooth zero-dimensional manifold ${\mathcal{M}}(x_{-},x_{+}) = \widehat{\mathcal{M}}(x_{-},x_{+})/ {\R}.$

Now we proceed with the definitions of Floer complexes. In our monotone case, when $\eta = pt,$ we consider the cover $\til{\cO}_{pt}(L,H; \cD)$ corresponding to $\ker([\omega]) \cap \ker(\mu_L) = \ker([\omega]),$ where $[\omega]: \pi_1(\cP_{pt}(L,L)) \cong \pi_2(M,L) \to \R$ is given by integrating with respect to the symplectic form, and $\mu_L:  \pi_2(M,L) \to \Z$ is the Maslov class. This means that up to a suitable equivalence relation, we look at pairs $(x,\bar{x})$ where $x \in {\cO_{pt}}(L,H; \cD)$ and $\overline{x}$ is a capping of $x,$ that is to say: a map $\bar{x}: \mathbb{D} \cap \mathbb{H} \to M$ with $\bar{x}|_{\mathbb{D} \cap \R},$ suitably reparametrized, coincides with $x,$ i.e. $\bar{x}(2t-1) = x(t)$ for all $t \in [0,1],$ and $\bar{x}|_{\partial \mathbb{D} \cap \mathbb{H}}$ has values in $L.$ Here $\mathbb{D} = \{z \in \C\,:\, |z| \leq 1\},$ and $\mathbb{H} = \{z \in \C\,:\, \mathrm{Im}(z) \geq 0\}.$ We consider two cappings $\bar{x},\bar{x}'$ equivalent if $v = \bar{x} \overline{\#} \bar{x}': (\mathbb{D},\partial \mathbb{D}) \to (M,L),$ defined by $\bar{x}(z)$ for $z \in \mathbb{D} \cap \mathbb{H},$ and $\bar{x}'(\bar{z})$ for   $z \in \mathbb{D} \cap \overline{\mathbb{H}},$ satisfies $\left < [\omega], [v] \right > = \left < \mu_L , [v] \right > = 0.$

For general $\eta,$ one chooses a reference path $\gamma_{0,\eta} \in \cP_{\eta}(L,L)$ and as cappings of $x \in \cO_{\eta}(L,H;\cD)$ considers equivalence classes of paths in $\cP_{\eta}(L,L)$ from $\gamma_{0,\eta}$ to $x,$ with respect to a similar equivalence relation: $(x,\bar{x}) \sim (x,\bar{x}')$ if and only if $\left < [\omega], [v] \right > = \left < \mu_L , [v] \right > = 0,$ where $v$ is now a map from an annulus to $M,$ with boundary mapped to $L,$ obtained from a loop in $\cP_{\eta}(L,L)$ based at $\gamma_{0,\eta}$ given by composing the path $\bar{x}$ with the path $\bar{x}'$ traversed in the reverse direction.

In fact it is easy to see that $\til{\cO}_{\eta} (L,H;\cD)$ are the critical points of the action functional \[\cA_{H;\cD}(z,\bar{z}) = \int_0^1 H^{\cD}(z(t)) - \int_{\bar{z}}\omega\] on a suitable cover $\til{\cP}_{\eta}(L,L)$ of ${\cP}_{\eta}(L,L)$ (given by the construction with the above equivalence relation). We shall use this functional to induce a filtration on our complex. We also set $\Spec(L,H; \cD) = \ima(\cA_{H;\cD}) \subset \R$ and $\Spec(L,H) = \Spec(L,H; 0).$ These are known \cite{Oh-chain} to be closed nowhere dense subsets of $\R.$

Finally, one can assign to each $(x,\bar{x}) \in \til{\cO}_{\eta} (L,H;\cD)$ an index $\mu_{CZ}(x,\bar{x}) \in \Z,$ by equipping the reference path $\gamma_{0,\eta} \in \cP_{\eta}(L,L)$ with a symplectic trivialization of $\gamma_{0,\eta}^*(TM),$ and a section of the Lagrangian Grassmannian $\mathcal{L}(\gamma_{0,\eta}^*(TM)) \to [0,1],$ connecting $T_{\gamma_{0,\eta}(0)}(L)$ and $T_{\gamma_{0,\eta}(1)}(L).$ This induces a symplectic trivialization of $\bar{x}^*(TM),$ and computing the Maslov index of a loop of Lagrangian subspaces of a now-fixed symplectic vector space, obtained by extending the path $l(H,x,\bar{x};\cD) = \{ D(\phi_{H^{\cD}}^{t})(x(0)) T_{x(0)}L \}$ in a canonical way. One can  describe $\mu_{CZ}$ alternatively as a suitably normalized Conley-Zehnder (or Robbin-Salamon \cite{RobbinSalamonIndex}) index of the path $l(H,x,\bar{x};\cD)$ under the above trivialization. We take the index normalized in such a way that for $H$ being a lift to a Weinstein neighborhood of a $C^2$-small Morse function $f$ on $L,$ and $K^{L,H} \equiv 0,$ we get for $(x,\bar{x})$ a constant path and capping at a critical point $q = x(0)$ of $f,$ $\mu_{CZ}(x,\bar{x}) = \ind_f(q),$ the Morse index of $q$ as a critical point of $f:L \to \R.$ Finally \[\ind(u,x_{-},x_{+}) = \mu_{CZ}(x_{-},\overline{x}_{-}) - \mu_{CZ}(x_{+},\overline{x}_{-} \# u ).\] 

\begin{rmk}\label{rmk:CZ-normalization}
We note that for purposes of computing $\ind(u,x_{-},x_{+})$ as such a difference, we may take {\em any} uniform shift of the above normalization of the Robbin-Salamon index. This will be useful later.
\end{rmk}

In our first definition, fixing $\eta,$ we consider for each $r \in \Z,$ the $\bK$-vector space $CF_r(L,H; \cD)$ freely generated by $(x,\bar{x}) \in \til{\cO}_{\eta} (L,H;\cD)$ that have $\mu_{CZ}(x,\bar{x}) = r.$ It is easy to see that in the monotone case this is a finite-dimensional vector space. This makes $CF(L,H; \cD)$ into a $\Z$-graded vector space over $\bK.$ We then define the differential $d_{L,H,\cD}:CF(L,H; \cD) \to CF(L,H; \cD)[-1]$ by extending by linearity from \[d_{L,H;\cD}(x_{-},\bar{x}_{-}) = \displaystyle \sum_{\substack{[u] \in \cM(x_{-},x_{+}),\\ \ind(u,x_{-},x_{+}) = 1}} (x_{+},\overline{x}_- \# u).\] By, again, a standard Gromov-Floer compactness argument, one obtains $d_{L,H;\cD}^2 = d_{L,H;\cD} \circ d_{L,H;\cD} = 0.$ 

In other words, the differential $d_{L,H;\cD}$ is given by counting the isolated unparametrized solutions to the Floer equation from $x_{-}$  to $x_{+},$ in the zero-dimensional component. In the above first version, the "weight" with which a solution is counted is given essentially by the homotopy class of the solution (up to identification based on the symplectic area, and Maslov index).

We note that $\Lambda_{\text{min},\bK}$ acts on $CF(L,H;\cD)$ by $q^{\mu_L([A])/N_L} \cdot (x,\bar{x}) = (x, \bar{x} \# A),$ turning $(CF(L,H; \cD), d_{L,H;\cD})$ into a finite rank free complex over $\Lambda_{\tmin}.$ Finally we introduce a filtration-level function on this complex by \[\cA(\sum \lambda_j (x_j,\bar{x}_j)) = \max_j\{ \cA_{H;\cD}(x_j,\bar{x}_j) - \nu(\lambda_j)\}\] for $\lambda_j \in \Lambda_{\tmin}$ and a basis $\{(x_j,\bar{x}_j)\} \subset \til{\cO}_{\eta} (L,H;\cD)$ of $CF(L,H; \cD)$ over $\Lambda_{\tmin}.$ Clearly this definition does not depend on the choice of such a basis.

We also record the versions with coefficients extended to $\Lambda_{\tmon},$ and $\Lambda_{\tuniv}$ given respectively by \[CF(L,H; \cD,\Lambda_{\tmon}) = CF(L,H; \cD) \otimes_{\Lambda_{\tmin}} \Lambda_{\tmon},\] \[CF(L,H; \cD,\Lambda_{\tuniv}) = CF(L,H; \cD) \otimes_{\Lambda_{\tmin}} \Lambda_{\tuniv},\] with differential $d_{L,H;\cD}$ extended by linearity.

The homology $HF_*(L,H;\cD),$ $HF_*(L,H;\cD,\Lambda_{\tmon})$ of the complexes $(CF_*(L,H; \cD),d_{L,H;\cD}),$ $(CF_*(L,H; \cD, \Lambda_{\tmon}),d_{L,H;\cD})$ respectively, is a graded module of finite rank over $\Lambda_{\tmin},$ respectively $\Lambda_{\tmon},$ with each graded component a finite-dimensional vector space over $\bK.$

Finally, given two Hamiltonians $H_{-},H_{+} \in \cH,$ and perturbations $\cD_{-},\cD_{+},$ by counting index $0$ isolated solutions for generic $J$ depending on $(s,t) \in Z,$ with $J_{s,t} = J^{L,H^{-}}, s \ll -1,$ and $J_{s,t} = J^{L,H^{+}}, s \gg 1,$ to a Floer equation \[(du - X^t_{K}(u) \otimes dt)^{(0,1)} = 0\] with boundary values on $L,$ and $K(s,t,x) = H_{-}^{\cD_{-}},$ $s \ll -1,$ and $K(s,t,x) = H_{+}^{\cD_{+}},$ $s \gg 1,$  we obtain a canonical graded isomorphism \[\Phi_{(H_{-};\cD_{-}),(H_{+};\cD_{+})}: HF_*(L,H_{-};\cD_{-}) \to HF_*(L,H_{+};\cD_{+})\] called the Floer continuation maps. These isomorphisms satisfy \[\Phi_{(H_{1};\cD_{1}),(H_{2};\cD_{2})} \circ \Phi_{(H_{0};\cD_{0}),(H_{1};\cD_{1})} = \Phi_{(H_{0};\cD_{0}),(H_{2};\cD_{2})},\] for each three pairs $(H_{0};\cD_{0}),(H_{1};\cD_{1}),(H_{2};\cD_{2})$ of Hamiltonians with perturbation data, as above. The colimit of the resulting indiscrete groupoid of graded $\Lambda_{\tmin}$-modules, is called the abstract Floer homology $HF_*(L)$ of $L.$ The same observation holds for coefficients extended to $\Lambda_{\tmon}.$

In our second definition, which is good to keep in mind, even if it is not used per se in our arguments, fixing $\eta,$ and suppressing it from further notation, we consider the set $\cO_{\eta}(L,H;\cD),$ and the free $\Lambda_{0,\tuniv}$-module $CF_{\text{univ}}(L,H;\cD)$ generated by it. We define the differential by \[d_{\text{univ},L,H;\cD} (x_{-}) = \displaystyle \sum_{\substack{[u] \in \cM(x_{-},x_{+}),\\ \ind(u,x_{-},x_{+}) = 1}} T^{E(u)} x_{+},\] where $E(u) > 0$ denotes the energy of the Floer trajectory $u,$ given by   \[E(u) = \int_{Z} |\del_s u|^2 ds\wedge dt, \] where the norm of $\del_s u(s,t)$ is taken with respect to the Riemannian metric given by $\omega$ and $J^{L,H}_t.$ An alternative, more general, description is that $J^{L,H}$ and $\omega$ endow $u^*(TM)$ with the structure of a Hermitian vector bundle $(\mathcal{E},h)$ over $Z,$ and $(du - X_{H^{\cD}} \otimes dt)^{(0,1)}$ gives a section $\phi$ of $\Omega^{(0,1)}(Z) \otimes_{\C} \mathcal{E}.$ Taking an area form $\sigma$ on $Z,$ we obtain by way of the complex structure $j$ on $Z$ a Hermitian metric on $TZ,$ and thus a norm $|\cdot|_{\sigma}$ on $\Omega^{(0,1)}(Z) \otimes_{\C} \mathcal{E}.$ Then \[E(u) = \frac{1}{2} \int_{Z} |du - X_{H^{\cD}}|^2_{\sigma} \,\sigma.\] Finally, the integrand is independent of $\sigma,$ since it can be written as the anti-symmetrization of the tensor on $Z$ obtained by evaluating $h$ on $\phi \otimes \phi.$

The homology $HF_{\text{univ}}(L,H;\cD)$ of $(CF_{\text{univ}}(L,H;\cD),d_{\text{univ},L,H;\cD})$ is a finitely generated $\Lambda_{0,\tuniv}$-module. Moreover, \[HF_{\text{univ}}(L,H;\cD) \otimes_{\Lambda_{0,\tuniv}} \Lambda_{\tuniv} \cong HF(L,H;\cD,\Lambda_{\tuniv}),\] the latter being the homology of $(CF_*(L,H; \cD, \Lambda_{\tuniv}),d_{L,H;\cD})$ from the first version. However, the torsion component of $HF_{\text{univ}}(L,H;\cD)$, as a $\Lambda_{0,\tuniv}$-module, depends on the choice of $(H,\cD).$

\bs
{\em Hamiltonian Floer homology:}

The absolute case of Floer theory with Hamiltonian term is defined similarly to the relative one, with various simplifications. As above, we let $(M,\omega)$ be a closed weakly monotone symplectic manifold, and $H \in \cH.$ We denote by $\cO(H;\cD)$ the contractible $1$-periodic orbits of the Hamiltonian flow $\{\phi^t_{H^{\cD}}\}$ of a perturbation $H^{\cD} = H+K^{H}$ of $H$ that is non-degenerate, in the sense that $D(\phi^1_{H^{\cD}})(x):T_x M \to T_x M$ does not have $1$ as an eigenvalue, for each starting point $x = z(0) \in M$ of $z \in \cO(H; \cD).$ If $H$ itself is non-degenerate in this sense, we take $K^H \equiv 0.$ 

We consider suitable coefficient rings. Let $s$ be a formal graded variable of degree $\deg(s) = 2$ and set \[\Lambda_{M,\text{mon},\bK} =\left \{\displaystyle\sum_{j \geq 0} a_j \cdot s^{-l_j}\;:\;\;a_j \in \bK,\, l_j \in \Z, l_j \nearrow +\infty \right\}.\] We note that $t \mapsto T^{-2\kappa}$ defines an ungraded embedding of fields $\iota_{2,0}: \Lambda_{M,\text{mon},\bK} \to \Lambda_{\text{univ},\bK}.$ The other coefficient ring, \[\Lambda_{M,\text{min},\bK} =\left \{\displaystyle\sum_{j \geq 0} a_j \cdot p^{-l_j}\;:\;\;a_j \in \bK,\, l_j \in \Z, l_j \nearrow +\infty \right\},\] where now $\deg(p) = 2 N_M,$ and $p \mapsto s^{N_M}$ gives a finite graded extension of fields $\iota_{N_M,2}: \Lambda_{M,\text{min},\bK} \to \Lambda_{M,\text{mon},\bK}.$

The Floer complex $CF_*(H;\cD)$ has, as above, as basis over $\Lambda_{M,\tmin}$ the set $\cO(H;\cD),$ and the differential counts isolated, up to $\R$-translation, solutions $u:S \to M,$ $S = \R \times S^1,$ to the Floer equation \[(du - X^t_{H^{\cD}}(u) \otimes dt)^{(0,1)} = 0, \] or in other words \[\del_s u + J_t(u)(\del_t u - X_{H^{\cD}}(u)) = 0,\] with  $u(s,-) \xrightarrow{s \to \pm \infty} x_{\pm}$  uniformly in $C^1$-topology, for $x_{-},x_{+} \in \cO(H; \cD).$ Again, for generic choice of almost complex structure $J=J^H,$ the spaces of solutions of index $k \geq 1$ are smooth manifolds of dimension $k,$ admitting a free action of $\R$ and for index $1,$ the quotient is a compact manifold of dimension $0,$ hence a finite number of points. The differential $d_{H;\cD}:CF_*(H;\cD) \to CF_*(H;\cD)[-1]$ is given by the counts in $\bK$ of the number of points in these moduli spaces: \[d_{H;\cD}(x_{-},\bar{x}_{-}) = \displaystyle \sum_{\substack{[u] \in \cM(x_{-},x_{+}),\\ \ind(u,x_{-},x_{+}) = 1}} (x_{+},\overline{x}_- \# u).\] By considering the Gromov-Floer compactification of the space of solutions of index $2,$ modulo the $\R$-action, one shows that $d_{H;\cD}^2 = 0.$ 

Over $\Lambda_{0,\tuniv},$ the differential $d_{\text{univ},H;\cD}$ on the free $\Lambda_{0,\tuniv}$-module $CF_{\text{univ}}(H;\cD)$  generated by $\cO(H;\cD)$ counts trajectories with weight given by their energy:  \[d_{\text{univ},H;\cD} (x_{-}) = \displaystyle \sum_{\substack{[u] \in \cM(x_{-},x_{+}),\\ \ind(u,x_{-},x_{+}) = 1}} T^{E(u)} x_{+},\] where $E(u) > 0$ denotes the energy of the Floer trajectory $u,$ given by \[E(u) = \int_{Z} |\del_s u|^2 ds\wedge dt, \]  the norm of $\del_s u(s,t)$ being taken with respect to the Riemannian metric given by $\omega$ and $J^{H}_t.$

We obtain $(CF_*(H;\cD),d_{H;\cD}),$ its extension of coefficients $(CF_*(H;\cD,\Lambda_{M,\tmon}),d_{H;\cD}),$ to $\Lambda_{M,\tmon},$ and $(CF_{\tuniv}(H;\cD),d_{\text{univ},H;\cD}),$ whose homologies we denote by \[HF_*(H;\cD),\; HF_*(H;\cD,\Lambda_{M,\tmon}),\; HF_{\tuniv}(H;\cD).\]

We note that when discussing absolute Floer homology in the presence of a Lagrangian submanifold, we shall consider, by default, further extensions of coefficients \[(CF_*(H;\cD,\Lambda_{\tmin}),d_{H;\cD}),\; (CF_*(H;\cD,\Lambda_{\tmon}),d_{H;\cD}),\] given by $\Lambda_{M,\tmin} \to \Lambda_{\tmin},$ $s \mapsto t^2,$ $\Lambda_{M,\tmon} \to \Lambda_{\tmon},$ $q \mapsto p^2$ respectively, with homologies \[HF_*(H;\cD,\Lambda_{\tmin}),\; HF_*(H;\cD,\Lambda_{\tmon}).\]

Finally there exist continuation maps, as above, between various pairs $(H,\cD)$ of Hamiltonians and perturbation data, the colimit along which is the abstract Floer homology of $(M,\omega),$ that is isomorphic to the quantum homology of $(M,\omega)$ by \cite{PSS}.

\subsection{Absolute Floer homology as relative}\label{sect: abs as rel}

We briefly explain how to see the latter case of Hamiltonian Floer homology as a particular case of the former, Lagrangian setting. We refer to \cite[Section 2]{LeclercqZapolsky} for details. Let $H \in \cH$ be a non-degenerate Hamiltonian on weakly monotone symplectic manifold $(M,\omega),$ and $\cD = (K^H, J^H)$ be a regular Floer datum, with $K^H = 0.$ Assume that $J_t=J_t^H$ coincides with a fixed almost complex structure for $t$ near $0,$ near $1/2,$ and for $t \in [1/2,1].$ By reparametrization, assume that $H_t$ vanishes for the same values of $t.$ Consider the weakly monotone symplectic manifold $M \times M^{-} = (M \times M, \om \oplus - \om),$ and its weakly monotone diagonal Lagrangian submanifold $\Delta_M \subset M \times M^{-}.$ Thenfor $(x,y) \in M \times M^{-},$ $\hat{H}(t,x,y) = \frac{1}{2} H_{t/2}(x) \in \cH_{M \times M^{-}}$ with perturbation $\hat{\cD} = (\hat{J},\hat{K}),$ for $\hat{J}_t = J_{t/2} \oplus - J_{0},$ and $\hat{K} \equiv 0,$ is non-degenerate, and there is a canonical chain-isomorphism of filtered complexes over $\Lambda_{M,\tmin} = \Lambda_{\Delta,\tmin},$ \[(CF(H;\cD),d_{H;\cD}) \to (CF(\Delta,\hat{H}; \cD),d_{\Delta,\hat{H}; \hat{\cD}}).\]

Similarly, we observe that the Lagrangian quantum homology $QH(L)$ for $L=\Delta_M$ is isomorphic to the quantum homology $QH(M).$ Moreover, an element $\eta_M \in \pi_1(\Ham(M,\om))$ and section class $\sigma_M$ define, via the map $\Ham(M,\om) \to \Ham(M \times M^{-}), \phi \mapsto \phi \times \id,$ an element $\eta_L \in \cP_L$ and a relative section class $\sigma_L,$ such that their Seidel elements $S_{\eta_M,\sigma_M} \in QH(M)$ and $S_{\eta_L,\sigma_L} \in QH(L)$ (see Section \ref{Sect:Seidel}) are identified by the above isomorphism.

\subsection{Persistence modules and boundary depth}\label{Sect:pmod}

We recall briefly the category $\pemod$ of persistence modules that we work with, together with their relevant properties. For detailed treatment of these topics see \cite{BauLes,Carlsson,CZCG,CrawBo,Ghrist,CarlZom}.

Let $\mathbb{K}$ be a field. A persistence module over $\mathbb{K}$ is a pair $(V,\pi)$ where, $\{V^t\}_{t\in \mathbb{R}}$ is a family of finite dimensional vector spaces over $\mathbb{K}$ and $\pi_{s,t}:V^s \rightarrow V^t$ for $s\leq t,~s,t\in \mathbb{R}$ is a family of linear maps, called {\it structure maps}, which satisfy:
\begin{enumerate}

  \item[1)] $V^t=0$ for $t \ll 0$ and $\pi_{s,t}$ are isomorphisms for all $s,t$ sufficiently large;
  \item[2)] $\pi_{t,r}\circ \pi_{s,t}=\pi_{s,r}$ for all $s\leq t \leq r$; $\pi_{s,s} = \id$ for all $s;$
  \item[3)] For every $r\in \mathbb{R}$ there exists $\varepsilon >0$ such that $\pi_{s,t}$ are isomorphisms for all $r-\varepsilon <s \leq t \leq r$;
  \item[4)] For all but a finite number of points $r\in \mathbb{R}$, there is a neighbourhood $U\ni r$ such that $\pi_{s,t}$ are isomorphisms for all $s\leq t$ with $s,t \in U$.
\end{enumerate}
The set of the exceptional points in 4), i.e. the set of all points $r\in \mathbb{R}$ for which there does not exist a neighbourhood $U\ni r$ such that $\pi_{s,t}$ are isomorphisms for all $s,t \in U$, is called {\it the spectrum} of the persistence module $(V,\pi)$ and is denoted by $\mathcal{S}(V)$. One easily checks that for two consecutive points $a<b$ of the spectrum and $a<s<t \leq b$, $\pi_{s,t}$ is an isomorphism. This means that $V^t$ only changes when $t$ "passes through points in the spectrum".

We define a morphism between two persistence modules $A:(V,\pi) \rightarrow (V',\pi')$ as a family of linear maps $A_t:V^t \rightarrow (V')^t$ for every $t\in \mathbb{R}$ which satisfies $$A_t \pi_{s,t}=\pi_{s,t}'A_s ~ ~ \text{for} ~ s<t.$$ Note that the kernel $\ker A$ and the image $\ima A$ are naturally persistence modules whose families of vector spaces are $\{\ker A_t \subset V^t\}_{t \in \R},$ $\{\ima A_t \subset (V')^t\}_{t \in \R},$ since the structure maps $\pi_{s,t}$ restrict to these systems of subspaces. In fact, it is not difficult to prove that $\pemod$ forms an abelian category, with the direct sum of two persistence modules $(V,\pi)$ and $(V',\pi')$ given by
 $$(V,\pi)\oplus(V',\pi')=(V\oplus V', \pi \oplus \pi').$$

{\exa Let $X$ be a closed manifold and $f$ a Morse function on $X$. For $t\in \mathbb{R}$ define $V^t(f)=H_*(\{ f<t \},\mathbb{K})$ to be homology of sublevel sets of $f$ with coefficients in a field $\mathbb{K},$ and let $\pi_{s,t}:V^s(f) \rightarrow V^t(f)$ be the maps induced by inclusions of sublevel sets. One readily checks that $(V(f),\pi)$ is a persistence module. The spectrum of $V(f)$ consists of critical values of $f$. Similarly fixing a degree $r \in \Z,$ one obtains a persistence module $V^t_r(f)=H_r(\{ f<t \},\mathbb{K}).$ It is easy to see that the spectrum of $V_r(f)$ is contained in the set of critical values of $f$ of critical points of index $r$ or $r+1.$ Finally $V(f) = \oplus V_r(f).$ Hence $V(f)$ has the structure of a persistence module of $\Z$-graded vector spaces.\label{BasicExample}}

\medskip

An important object in our story is the barcode associated to a persistence module. It arises from the structure theorem for persistence modules, which we now recall. Let $I$ be an interval of the form $(a,b]$ or $(a,+\infty)$, $a,b\in \mathbb{R}$ and denote by $Q(I)=(Q(I),\pi)$ the persistence module which satisfies $Q^t(I)=\mathbb{K}$ for $t\in I$ and $Q^t(I)=0$ otherwise and $\pi_{s,t}=id$ for $s,t\in I$ and $\pi_{s,t}=0$ otherwise.

\medskip

 {\thm[The structure theorem for persistence modules] For every persistence module $V$ there is a unique collection of pairwise distinct intervals $I_1,\ldots , I_N$ of the form $(a_i,b_i]$ or $(a_i,+\infty)$ for $a_i,b_i \in \mathcal{S}(V)$ along with the multiplicities $m_1,\ldots,m_N$ such that
 $$V\cong \bigoplus\limits_{i=1}^N(Q(I_i))^{m_i}.$$
 \label{StructureTheorem}}
The multi-set which contains $m_i$ copies of each $I_i$ appearing in the structure theorem is called {\it the barcode} associated to $V$ and is denoted by $\mathcal{B}(V)$. Intervals $I_i$ are called {\it bars}.

\medskip

\begin{rmk}
One feature of Example~\ref{BasicExample} is the existence of additional structure that comes from identifying  $V^\infty := \displaystyle\varinjlim V^t$ with $H_*(X,\mathbb{K}).$ Put $\Psi: V^\infty \to H_*(X,\mathbb{K})$ for the natural isomorphism. Given $a \in H_*(X,\mathbb{K})$ with $a \neq 0,$ we can produce the number $c(a,f) := \inf\{t \in \R \,|\, \Psi^{-1}(a) \in \ima (V^t \to V^\infty)\}.$ This number is called a {\em spectral invariant}, and has many remarkable properties. One can prove that for each $a \neq 0,$ $c(a,f)$ is a starting point of an infinite bar in the barcode of $V(f),$ and each such starting point can be obtained in this way.
\end{rmk}

\medskip

For an interval $I=(a,b]$ or $I=(a,+\infty)$, let $I^{-c}=(a-c,b+c]$ or $I^{-c}=(a-c,+\infty)$, and similarly $I^c=(a+c,b-c]$ or $I^c=(a+c,+\infty),$ when $b-a>2c$. We say that barcodes $\cB_1$ and $\cB_2$ admit a $\delta$-matching if it is possible to delete some of the bars of length $\leq 2\delta$ from $\cB_1$ and $\cB_2$ (and thus obtain $\bar{\cB}_1$ and $\bar{\cB}_2$) such that there exists a bijection $\mu:\bar{\cB}_1 \rightarrow \bar{\cB}_2$ which satisfies
$$\mu(I)=J \Rightarrow I\subset J^{-\delta},~J\subset I^{-\delta}.$$
We define {\it the bottleneck distance} $d_{\mrm{bottle}}(\cB_1,\cB_2)$ between barcodes $\cB_1,\cB_2$ as the infimum over $\delta >0$ such that there exists a $\delta$-matching between them.

For a persistence module $V=(V,\pi)$ denote by $V[\delta]=(V[\delta],\pi[\delta])$ the shifted persistence module given by $V[\delta]^t=V^{t+\delta},\pi_{s,t}=\pi_{s+\delta,t+\delta}$ and by $sh(\delta)_{V}:V\rightarrow V[\delta]$ the canonical shift morphism given by $(sh(\delta)_{V})_t=\pi_{t,t+\delta}:V^t \rightarrow V^{t+\delta}$. Note also that a morphism $f:V\rightarrow W$ induces a morphism of $f[\delta]:V[\delta] \rightarrow W[\delta]$. We say that a pair of morphisms $f:V\rightarrow W[\delta]$ and $g:W\rightarrow V[\delta]$ is a $\delta$-interleaving between $V$ and $W$ if
$$g[\delta]\circ f = sh(2\delta)_{V}\text{ and } f[\delta]\circ g=sh(2\delta)_{W}.$$
Now we can define {\it the interleaving distance } $d_{\mrm{inter}}(V,W)$ between $V$ and $W$ as infimum over all $\delta>0$ such that $V$ and $W$ admit a $\delta$-interleaving. The isometry theorem for persistence modules states that $d_{\mrm{inter}}(V,W)=d_{\mrm{bottle}}(\mathcal{B}(V),\mathcal{B}(W))$ (see \cite{BauLes}).

\subsection{Variants of Floer persistence}\label{Sec:Floer-persistence}

In this section we describe a persistence module that arises from the Floer complex $CF(L,H;\cD)$ (recall that it has coefficients in $\Lambda_{\tmin}$) of a monotone Lagrangian submanifold, and describe two alternative ways of computing the set of its finite bar-lengths.

Fix a degree $m \in \Z,$ and a real number $t \in \R.$ Let $CF(L,H; \cD)^{<t}$ be the subcomplex of $CF(L,H\;\cD)$ spanned by all generators $(z,\bar{z}) \in \til{\cO}_{\eta} (L,H;\cD)$ with $\cA_{H;\cD}(z,\bar{z}) < t.$ (In this paper we work with the contractible orbit class $\eta = pt.$) Now we set $V_m(L,H; \cD)^t = HF_m(L,H; \cD)^{<t} = H_m (CF(L,H; \cD)^{<t}).$ Since there are only a finite number of generators (in class $\eta$) of index $m-1,m,m+1,$ this homology is of finite rank over $\bK$ for all $t \in \R.$ The structure maps are given (in homology) by the inclusions $CF(L,H; \cD)^{<s} \subset CF(L,H; \cD)^{<t}$ of complexes, for $s \leq t.$ Hence we get a persistence module $V_m(L,H\; \cD) \in \pemod_{\bK}.$ We denote its barcode by $\cB_m(L,H\;\cD).$ We note that multiplication by $q$ induces an isomorphism \[V_m(L,H; \cD) \to V_{m + N_L}(L,H; \cD)[-A_L],\] hence, if we wish, for example, to compute the possibly distinct bar-lengths, we can restrict to \[V(L,H;\cD) = \displaystyle\oplus_{0 \leq m < N_L } V_m(L,H; \cD).\]

Finally, as in \cite{PolShe,PolSheSto}, continuation maps corresponding to changing the $J^{\cD}$ component of the perturbation datum induce isomorphisms of persistence modules, hence when $(L,H)$ is non-degenerate, the Hamiltonian term in $\cD$ can be taken to be identically zero, and we have a well-defined persistence modules that we denote $V_m(L,H)$ and $V(L,H)$ in this case. Moreover, Hamiltonian continuation maps induce interleavings between the persistence modules showing that \[d_{\mrm{inter}} (V_m(L,F),V_m(L,G)) \leq d_{\mrm{Hofer}}([F], [G] ).\]

Furthermore, we remark that if we used coefficients in $\Lambda_{\tmon},$ and defined analogously the persistence module $V_m(L,H; \cD,\Lambda_{\tmon}),$ all the distinct bar-lengths would be captured by $V_0(L,H; \cD,\Lambda_{\tmon}),$ since as above, multiplication by $t$ would induce an isomorphism \begin{equation}\label{eq:t mult iso}V_m(L,H; \cD,\Lambda_{\tmon}) \to V_{m + 1}(L,H; \cD,\Lambda_{\tmon})[-\kappa],\end{equation} recalling that $\kappa = A_L/N_L.$ In fact, it follows by an immediate chain-level inspection that \[V_r(L,H; \cD,,\Lambda_{\tmon}) = \displaystyle\oplus_{0 \leq m < N_L } V_{r+m}(L,H; \cD)[-m\kappa].\] We define the {\it bar-length spectrum} of $(L,H;\cD)$ as the tuple $(\beta_1,  \ldots , \beta_K)$ of all lengths of finite bars in $V_0(L,H; \cD,\Lambda_{\tmon}),$ arranged in increasing order: $\beta_1 \leq  \ldots \leq \beta_K.$ The boundary depth is then given by $\beta = \beta_K.$ It is sometimes beneficial to also formally add $\beta_{K+1} = +\infty,\ldots,\beta_{K+B} = +\infty,$ for $B$ equal to the number of infinite bars in the bar-code.

It is useful to keep in mind the following alternative way of describing the bar-length spectrum, due to Fukaya-Oh-Ohta-Ono \cite[Chapter 6]{FO3:book-vol12}, \cite{FOOO-polydiscs}. Via a normal form theorem, which holds for {\em finitely generated} modules over $\Lambda_{\tuniv,0},$ there is a direct sum decomposition \[HF(L,H;\cD,\Lambda_{\tuniv,0}) \cong F \oplus T,\] where $F$ is a free module over $\Lambda_{\tuniv,0}$ with $F \otimes_{\Lambda_{\tuniv,0}} \Lambda_{\tuniv} \cong HF(L,H;\cD,\Lambda_{\tuniv})$ as vector spaces over $\Lambda_{\tuniv},$ and $T$ is a canonical torsion submodule of $HF(L,H;\cD,\Lambda_{\tuniv,0})$ over $\Lambda_{\tuniv,0}.$ In turn $T \cong \displaystyle\oplus_{1\leq j\leq K} \Lambda_{\tuniv,0}/T^{\beta_j}\Lambda_{\tuniv,0}.$ The entries in the bar-length spectrum are called the torsion exponents in \cite{FO3:book-vol12,FOOO-polydiscs}.

Finally, the above two descriptions are equivalent by the work of Usher-Zhang \cite{UsherZhang}. Indeed, it follows from their arguments that there exists a non-Archimedean orthogonal homogeneous (that is each element lies in $CF_m(L,H; \cD)$ for some degree $0 \leq m < N_L$) basis \[\{x_1,...,x_B,y_1,...,y_K,z_1,...,z_K\}\] of $CF(L,H;\cD)$ as a Floer type complex over $\Lambda_{\tmin},$ satisfying \[dx_j = 0,\; dy_k =z_k\] for all $1 \leq j \leq B,\; 1 \leq k \leq K.$ Moreover, they prove that for each such basis, ordered suitably, $\beta_k = \cA(y_k) - \cA(z_k)$ for all $1 \leq k \leq K.$ Note that the number $K=K(C,d)$ is characterized by $K = \dim(\ima(d)).$ Note that the normal form over $\Lambda_{\tuniv,0}$ is obtained by a similar basis \[\{\bar{x}_1,...,\bar{x}_B,\bar{y}_1,...,\bar{y}_K,\bar{z}_1,...,\bar{z}_K\}\] over $\Lambda_{\tuniv,0}$ with the property that $d\bar{x}_1 = 0,...,d\bar{x}_B = 0$ and $d\bar{y}_k = T^{\beta_k} \bar{z}_k.$ This basis immediately gives a basis of the previous kind over $\Lambda_{\tuniv,0},$ whence the two defitions of the bar-length spectrum agree.


\subsection{Product structures on Lagrangian Floer homology}\label{sec:Products}

Above we have described Lagrangian Floer theory as a (filtered) complex. In this section we describe product structures on these complexes, and discuss their associativity. We follow \cite{SeidelBook} and \cite{LeclercqZapolsky,Zap:Orient}, as well as \cite{BiranCorneaS-Fukaya}, for these constructions.

{\em Punctured Riemann surfaces:} we consider the standard closed disk $\D \subset \C,$ or $\CP^1,$ as a Riemann surface $\overline{\Sigma}$ with boundary $\partial \overline{\Sigma} = S^1,$ respectively $\partial \overline{\Sigma} = \emptyset.$ Let $\Gamma:P \to \overline{\Sigma}$ be an embedding, where $P = I_{+} \sqcup I_{-} \sqcup B_{+} \sqcup B_{-},$ for finite sets $I_{\pm}, B_{\pm}.$ The image $\Gamma(P)$ is called the set of punctures, with interior punctures being $\Gamma(I_{\pm}) \subset \inte(\overline{\Sigma}) = \overline{\Sigma} \setminus \partial \overline{\Sigma},$ and boundary punctures being $\Gamma(B_{\pm}) \subset \partial\overline{\Sigma}.$ We call $\Gamma(P_{-}),$ for $P_{-} = I_{-} \sqcup B_{-},$ the input punctures, and $\Gamma(P_{+}),$ for $P_{+} = B_{+}\sqcup I_{+},$ the output punctures. Our object of interest shall be the punctured Riemann surface $\Sigma = \overline{\Sigma} \setminus \Gamma(P).$ In this paper we shall be interested in the case when $|I_{-}| \in \{0,1,2\},$ $|I_{+}| = 0,$ $|B_{-}| \in \{0,1,2,3\}$ and $|B_{+}| = 1.$ 

{\em Cylindrical ends:} we endow each puncture $z \in \Gamma(B_{-})$ with a cylindrical end \[\epsilon_{z}: (-\infty,0) \times [0,1] \to U_{z} \setminus \{z \},\] and similarly for $z \in \Gamma(B_{+})$ and \[\epsilon_{z}: (0, \infty) \times [0,1] \to U_{z} \setminus \{z \},\] $z \in \Gamma(I_{-})$ and \[ \epsilon_{z}: (-\infty, 0 ) \times S^{1} \to U_{z} \setminus \{z\}, \] $z \in \Gamma(I_{+})$ and \[ \epsilon_{z}: (0, \infty) \times S^1 \to U_{z} \setminus \{z\},\] where each $\epsilon_{z}$ is a biholomorphism to a punctured open neighborhood $U_{z} \setminus \{z\}$ of $z$ in $\overline{\Sigma},$ and $S^1 \cong \R/\Z.$  Finally, in the case $z \in \Gamma(I_{\pm})$ we require that $\epsilon_{z}(s,0),$ be asymptotically, as $s \to \pm \infty,$ tangent to a fixed direction $\vartheta \in \mathbb{P}_{+} T_z \overline{\Sigma}$ in the positive projectivization of $T_z \overline{\Sigma},$ and call it an {\em asymptotic marker} at $z.$ We note that the space of boudary cylindrical ends at a given boundary puncture $z \in \Gamma(B)$ or at a given interior puncture $z \in \Gamma(I),$ with a fixed asymptoic marker $\vartheta,$ are contractible (\cite{SeidelBook},\cite[Addendum 2.3]{AbouzaidSeidel}). In the case of interior punctures, we shall specify the asymptotic markers in our definitions.

{\em Cylindrical strips:}  we make use of the following notion to introduce curvature-zero Hamiltonian perturbations to Floer equations on Riemann surfaces. Consider a directed graph $G,$ without multiple edges and loops, and fix a homotopy class $\cG$ of embeddings of $G$ in $\overline{\Sigma},$ with vertices $V(G)$ corresponding to a subset of $\Gamma(P),$ and edges $E(G)$ corresponding to simple curves between points in this subset. For a vertex $z \in V(G),$ let $v(z)$ denote its valency in $G.$ For each edge $e = (z_-,z_+) \in E(G),$ where $z_{\pm} \in \Gamma(P),$ a cylindrical strip corresponding to it is a smooth and proper embedding $\varepsilon_e: \R \times (0,1) \to \Sigma$ of a strip, that satisfies for all $\pm\rho \geq  \rho_{\pm},$ \begin{equation} \label{equation: cyl ends strips}\varepsilon_{z_{-},z_{+}}(\rho,\theta) = \epsilon_{z_{\pm}} \circ \tau_{z_{\pm}}(\rho \mp \rho_{\pm},  \frac{1}{2 v(z_{\pm})}\theta + \theta_{\pm}),\end{equation} for some sufficiently large constants $\rho_{\pm} \in \R_{>0},$ $\theta_{\pm} \in \frac{1}{2v(z_{\pm})} \Z \cap [0,\frac{1}{2}),$ where $(\rho,\theta)$ are the natural coordinates on $\R \times (0,1),$  and $\tau_{z_-}(\rho, \theta) = (-\rho, \frac{1}{2}-\theta)$ if $z_-$ is an output, $\tau_{z_-} = \id$ otherwise, while $\tau_{z_+}(\rho, \theta) = (-\rho, \frac{1}{2}-\theta)$ if $z_+$ is an input, $\tau_{z_+} = \id$ otherwise. Finally, we require that different cylindrical strips have disjoint images, and their longitudinal lines $\{\varepsilon_e(\R \times \{0\})\}_{e \in E(G)},$ suitably compactified by their endpoints, represent the chosen class $\cG$ of embeddings of $G.$

\begin{rmk}
Note that for interior punctures, condition \eqref{equation: cyl ends strips} requires the cylindrial strips to enter into the cylindrical end at $z$ in the positive half-space determined by the asymptotic marker at $z$ and the complex structure on $\Sigma.$
\end{rmk}

{\em Hamiltonian terms, and perturbation data:} to describe the Floer equations on punctured Riemann surfaces with cylindrical ends, we first recall that in this paper we distinguish between fixed Hamiltonian terms in the Floer equation, which can be large, and a small Hamiltonian part of the perturbation data, that is used to render the Floer equations regular. In general, the Floer equation takes the form \[ (du - X_K(u))^{(0,1)} = 0 \] for a Hamiltonian 1-form $K \in \Omega^1(\Sigma, \sm{M}),$ and a domain-dependent almost complex structure $J$ on $M$ (that depends on points of a suitable universal curve $\cS$ with fiber diffeomorphic to $\Sigma$).

We require that $K$ restricted to each component of $\partial \Sigma := \partial {\overline{\Sigma}} \cap \Sigma$ take values in functions that vanish on $L.$ Moreover, we ask that $K$ be a $C^1$-small perturbation of the Hamiltonian 1-form $K_0$ defined as follows: for each edge $e=(z_{-},z_{+}) \in E(G),$ \[\varepsilon^*_{z_-,z_+} K_0 = (H_{z_-,z_+})_t \otimes dt\] for a Hamiltonian $H_{z_-,z_+} \in \cH,$ while ouside the cylindrical strips, $K_0$ is set to be $0.$ Note that this means that for each $z \in \Gamma(P),$ \[\eps_{z}^* K_0 = (H_z)_t \otimes dt\] for a Hamiltonian $H_z \in \cH,$ which is completely determined via \eqref{equation: cyl ends strips} and extension by $0,$ by the collection of Hamiltonians $\{H_{z_-,z_+}\}_{(z_-,z_+) \in E(G)}.$ By a $C^1$-small perturbation we mean that for each cylindrical end,\[\eps_{z}^* K = (F_z)_t \otimes dt,\] where $F_{z} \in \cH$  is a $C^1$-small regular perturbation $F_z = H^{\cD}_z$ of $H_z \in \cH.$  We require that $\eps_z^* J \equiv J_z = J^{L,H_z}$ be $s$-independent on the cylindrical ends, and such that $(L,F_z,J_z)$ yields a well defined Floer complex $CF(L,H_z; \cD).$ If $(L,H_z)$ is regular, then we shall take $H^{\cD}_z = H_z.$

Finally, we require that $\eps_z^* J \equiv J_z = J^{L,H_z}$ be $s$-independent on the cylindrical ends, and such that $(L,F_z,J_z)$ yields a well-defined Floer complex $CF(L,H_z; \cD).$

\medskip

\begin{rmk}
The curvature term $R(K) \in \Om^2(\Sigma, \cH)$ of a Hamiltonian 1-form $K$ is calculated in local coordinates $(s,t)$ on $\Sigma$ as \[R(K)(\del_s,\del_t) = \del_s K(\del_t) - \del_t K(\del_s) + \{K(\del_t), K(\del_s)\},\] where $\{F,G\} = -\omega(X_F,X_G)$ denotes the Poisson bracket of $F,G \in \sm{M,\R}.$ We note that by construction, as the Hamiltonians $H_{z_-,z_+}, H_{z_-}, H_{z_+} \in \cH$ do not depend on the $s$-variable, the curvature term of $K_0$ vanishes \[R(K_0) = 0.\] We could of course consider also the case when these Hamiltonians do depend on $s,$ but then $R(K_0)$ would not vanish in general. Moreover, for a similar reason \[\eps_z^* R(K) = 0\] for all $z \in \Gamma(P).$
\end{rmk}

\medskip

In the list below, we describe the key moduli spaces of punctured Riemann surfaces that we use in this paper, along with their cylindrical data. Our moduli spaces $\cR_{ij:kl}$ will consist of marked Riemann surfaces, where $i = |B_-|, j = |B_+|, k = |I_-|, l = |I_+|,$ and the marked points will be suitably constrained. The underlying Riemann surface is $\overline{\Sigma} = \D,$ unless $i=0$ and $j=0,$ in which case it is $\overline{\Sigma} = \C P^1.$ The embedding classes $\cG$ that we consider are determined by the graph $G,$ hyperbolic geodesic (for $\overline{\Sigma} = \D$) or spherical geodesic (for $\overline{\Sigma} = \C P^1$) segments between the endpoints, and the additional assumption that in the case of $\overline{\Sigma} = \C P^1$ and $k+l \geq 3$ all edges are embedded on the positive side of the circle through the marked points, determined by its direction and the complex structure. For our moduli spaces of curves, this can be made compatible with the compactification, for instance by a suitable adaptation of the uniformization theorem.

\begin{itemize}\label{list:moduli}
\item[$\cR_{11:00}:$] \label{R:1100} here $B_{-} = \{z_{-}\},$ $B_{+}= \{z_+\},$ and the cylindrical strip graph is $E(G) = \{(z_-,z_+)\}.$

\item[$\cR_{01:00}:$] \label{R:0100} here $B_{-} = \{z_{-}\},$ and $E(G) = \emptyset.$

\item[$\cR_{20:00}:$] \label{R:2000} here $B_{-} = \{z_{1, -},z_{2,-}\},$ and $E(G)= \{(z_{1,-},z_{2,-})\}.$

\item[$\cR_{21:00}:$] \label{R:2100} here $B_{-} = \{z_{1, -},z_{2,-}\},$ $B_{+}= \{z_+\},$ with $z_{1, -},z_{2,-},z_{+}$ ordered in clockwise order, and ${E(G)= \{(z_{1,-},z_{+}), \; (z_{2,-},z_{+})\}.}$

\item[$\cR_{31:00}:$]\label{R:3100}  here
$B_{-} = \{z_{1, -},z_{2,-},z_{3,-}\},$ $B_{+}= \{z_+\},$ with $z_{1, -},z_{2,-},z_{3,-}, z_{+}$ ordered in clockwise order, and ${E(G) = \{(z_{1,-},z_{+}), \; (z_{2,-},z_{+}), (z_{3,-},z_+)\}.}$

\item[$\cR_{11:10}:$]\label{R:1110} here $B_{-} = \{z_{-}\},$ $B_{+}= \{z_+\},$ $I_{-} = \{w_{-}\},$ with $w_-$ on the hyperbolic geodesic from $z_-$ to $z_+$ equipped with asymptotic marker $\vartheta$ tangent to this geodesic and agreeing with its direction, ${E(G) = \{(z_{-},w_{-}), \; (w_{-},z_{+})\}.}$

\item[$\cR_{11:20}:$]\label{R:1120} here $B_{-} = \{z_{-}\},$ $B_{+}= \{z_+\},$ $I_{-} = \{w_{1,-}, w_{2,-}\},$ with $w_-, w_+$ on the hyperbolic geodesic from $z_-$ to $z_+,$ {\em in this order}, equipped with asymptotic markers $\vartheta_1, \vartheta_2$ tangent to this geodesic and agreeing with its direction, ${E(G) = \{(z_{-},w_{1,-}), \; (w_{1,-},w_{2,-}),\; (w_{2,-},z_{+})\}.}$

\item[$\cR_{21:10}:$]\label{R:2110} here $B_{-} = \{z_{-},z_{*}\},$ $B_{+}= \{z_+\},$ $I_{-} = \{w_{-},\},$ with $w_-$ on the hyperbolic geodesic from $z_-$ to $z_+,$ {\em in this order}, equipped with asymptotic markers $\vartheta_1, \vartheta_2$ tangent to this geodesic and agreeing with its direction, ${E(G) = \{(z_{-},w_{-}), (w_{-},z_{+})\}.}$ In this case we keep $z_{*}$ as an extra marked point, and do not endow it with cylindrical ends.

\item[$\cR_{00:11}:$] \label{R:0011}  here $I_{-} = \{w_{-}\},$ $I_{+}= \{w_+\},$ with asymptotic markers at $w_{1,-},w_{2,-}$ pointing towards each other along a circle in $\C P^1,$ and the cylindrical strip graph is $E(G) = \{(w_-,w_+)\}.$

\item[$\cR_{00:01}:$] \label{R:0001} here $I_{-} = \{w_{-}\},$ and $E(G) = \emptyset.$

\item[$\cR_{00:20}:$] \label{R:0020} here $I_{-} = \{w_{1, -},w_{2,-}\},$ with asymptotic markers at $w_{1,-},w_{2,-}$ pointing towards each other along a circle in $\C P^1,$ and $E(G)= \{(w_{1,-},w_{2,-})\}.$

\item[$\cR_{00:21}:$] \label{R:0021} here $I_{-} = \{w_{1,-},w_{2,-}\},$ $I_{+} = \{w_+\},$ with $(w_{1,-},w_{2,-},w_+)$ on a circle in $\C P^1,$ {\em oriented in this order,} with asymptotic markers at $w_{1,-},w_{2,-}$ agreeing with the direction of the circle, the one at $w_+$ disagreeing with it, and ${E(G) = \{ (w_{+},w_{1,-}), (w_{1,-},w_{2,-}), (w_{2,-},w_{+}) \}.}$

\item[$\cR_{00:31}:$] \label{R:0031} here $I_{-} = \{w_{1,-},w_{2,-},w_{3,-}\},$ $I_{+} = \{w_+\},$ with $(w_{1,-},w_{2,-},w_{3,-},w_+)$ on a circle in $\C P^1,$ {\em oriented in this order,} with asymptotic markers at $w_{1,-},w_{2,-}, w_{3,-}$ agreeing with the direction of the circle, the one at $w_+$ disagreeing with it, and ${E(G) = \{ (w_{+},w_{1,-}), (w_{1,-},w_{2,-}), (w_{2,-},w_{+}) \}.}$

\end{itemize}

Note that $\cR_{21:00}, \cR_{31:00}, \cR_{11:10}, \cR_{11:20}, \cR_{00:21}$ consist of stable marked punctured Riemann surfaces; $\cR_{21:00},$  $\cR_{11:10}$, $\cR_{00:21}$ each consist of a single point $\ast$; $\cR_{31:00},$ $\cR_{11:20},$ $\cR_{21:10},$ $\cR_{00:31},$ can be equipped with a canonical smooth structure, diffeomorphic to the open interval $(0,1)$. Moreover, by the usual theory of Deligne-Mumford compactifications (cf. \cite[Chapter 9]{SeidelBook}), the latter moduli spaces compactify to smooth manifolds with corners (corresponding to stable nodal surfaces) in this case each diffeomorphic to the closed interval $[0,1].$ These compactifications can be described set-wise as: \[\overline{\cR}_{31:00} = {\cR}_{31:00} \sqcup \cR_{21:00} \times \cR_{21:00} \sqcup \cR_{21:00} \times \cR_{21:00},\] \[\overline{\cR}_{11:20} = {\cR}_{11:20} \sqcup {\cR}_{11:10} \times {\cR}_{11:10} \sqcup {\cR}_{11:10} \times {\cR}_{00:21},\]
\[\overline{\cR}_{21:10} = {\cR}_{21:10} \sqcup {\cR}_{21:00} \times {\cR}_{11:10} \sqcup {\cR}_{11:10} \times {\cR}_{21:00},\]
\[\overline{\cR}_{00:31} = {\cR}_{00:31} \sqcup \cR_{00:21} \times \cR_{00:21} \sqcup \cR_{00:21} \times \cR_{00:21}.\]

Let $\cR = \cR_{ij:kl}$ be one of the moduli spaces above, and $\overline{\cR} = \overline{\cR}_{ij:kl}$ be its compactification. We denote by $\cS = \cS_{ij:kl}$ the universal marked punctured curve over the moduli space $\cR$ as above, and on it we fix a choice of cylindrical ends and cylindrical strips fibered over $\cR$. The space $\cS$ admits a partial compactification $\overline{\cS} = \overline{\cS}_{ij:kl} \to \overline{\cR}.$ 

Following \cite[Chapter 9]{SeidelBook} it is rather easy to see that there exists a choice of cylindrical ends and cylindrical strips for the given fixed graph $G,$ that extends to this partial compactification. The choice is done essentially by a gluing construction for punctured surfaces with cylindrical ends, at a puncture. It is made recursively from the lower-dimensional boundary strata in $\overline{\cR}$ to higher ones: in our case we must simply extend from the zero-dimensional strata to the one-dimensional stratum. Consider a strip-like end $\epsilon_+$ of a boundary (resp. interior) output in one Riemann surface $\Sigma_+$, and a strip-like end $\epsilon_-$ of a boundary (resp. interior) input of another one $\Sigma_-$. 
Choosing a gluing length $l > 0,$ one then glues $\Sigma_{+} \setminus \epsilon_{+}(Z_{l,+}) \sqcup \Sigma_{-} \setminus \epsilon_{-}(Z_{l,-})$ where $Z_{l,+} = (l,\infty) \times [0,1],$ $Z_{l,-} = (-\infty,-l) \times [0,1]$ (resp. $Z_{l,+} = (l,\infty) \times S^1,$ $Z_{l,-} = (-\infty,-l) \times S^1$), along the maps \begin{align} [0,l] \times [0,1] \to \Sigma_{+},\; &(s,t) \mapsto (s,t), \\ [0,l] \times [0,1] \to \Sigma_{-},\; &(s,t) \mapsto (s-l,t) \end{align} (resp. $[0,l] \times S^1 \to \Sigma_{+},\; (s,t) \mapsto (s,t),$ and $[0,l] \times S^1 \to \Sigma_{-},\; (s,t) \mapsto (s-l,t)$). 

We remark the gluing length goes to $0,$ the glued curve converges to the boundary of the moduli space (more precisely, to the disjoint union of the curves $\Sigma_+, \Sigma_-$ which represents a point in the compactification). Assume that we are given graphs $G_+,G_-$ and embedding homotopy classes $\cG_+,\cG_-$ that glue to the class $\cG$ under the above procedure. Then the glued curve, that depends on $l,$ inherits cylindrical ends from the two surfaces (other than $\epsilon_{\pm}$), and this can be done consistently with the initial choice of cylindrical ends on $\cS$ (\cite[Chapter 9]{SeidelBook}) and cylindrical strips. This gives a choice of cylindrical data in the neighborhood of a particular boundary stratum. If we can make the choices of homotopy classes of graph embeddings consistent among the different boundary strata, extending the cylindrical data from the boundary to the interior of the moduli space is elementary (cf. \cite[Section 12.2]{Steenrod-book},\cite[Section 6.4]{GriffithsMorgan-book}). In our particular setting, our choice of graphs and their embedding classes guarantees the required topological consistency.

\begin{rmk}

While in this paper we consider only zero and one-dimensional moduli spaces $\overline{\cR}$ of decorated curves, a more general setting is worth mentioning. Consider the space $\cR_{\ast}$ of stable marked curves with boundary $S^1,$ or without boundary, decorated with asymptotic markers at the interior punctures. A compactification $\overline{\cR}_{\ast}$ of $\cR_{\ast}$ due to Kimura-Stasheff-Voronov \cite{KimuraStasheffVoronov}, and Liu \cite{MelissaLiu}, of this space (see also \cite[Section 5]{SeidelDisjoinable}) combines features of the Deligne-Mumford compactification for dealing with boundary punctures, and involves codimension one boundary strata to deal with interior punctures, by adding an angular parameter at the interior node. We note that as the various strata of $\overline{\cR}_{\ast}$ are not simply-connected, a consistent choice of cylindrical strips is not in general possible. It is enough, for instance, to consider the moduli space of spheres with two marked points with asymptotic markers. This space is identified with $S^1,$ and the monodromy of the universal curve $\cS_{\ast}$ over a generator of $\pi_1(S^1) \cong \Z$ is a non-trivial Dehn twist in the suitable mapping class group of the sphere with asymptotic markers, which, moreover, acts non-trivially on homotopy classes of cylindrical strips (incidentally, this is related to why, in general, the $BV$-operator cannot be defined with zero-curvature perturbation data). However, if we consider a compact submanifold with corners $\overline{\cR}$ of $\overline{\cR}_{\ast}$, with all its strata simply connected, then a consistent choice of cylindrical strips should be possible to achieve along the above lines.

\end{rmk}

{\em Consistent perturbations, Floer moduli spaces:} Consider the universal curve $\cS \to \cR ,$ with partial compactification $\overline{\cS} \to \overline{\cR} ,$ endowed with a fixed choice of strip-like ends and cylindrical strips associated to a fixed embedding $\cG$ of a graph $G,$ extended consistently to the partial compactification.

Floer data for this family of curves consist of a family $\{K_r\}_{r\in \overline{\cR}}$ of Hamiltonian 1-forms on $S_r$ for each $r \in \overline{\cR},$ and of a family $\{J_r\}_{r \in \overline{\cR}}$ of $\omega$-compatible almost complex structures on $M$ depending on $z \in S_r.$ These data are required to be consistent with the compactification, essentially in the sense that they are smooth over the compactified space (see \cite[Chapter 9]{SeidelBook}). Moreover, they are given as above, with a fixed choice of Hamiltonians $\{H_z\}_{z \in V(G)},$ and $\{H_{z_-,z_+}\}_{(z_-,z_+) \in E(G)},$ over $\cR.$ We also fix the choice of perturbations $\{F_z = H^{\cD}_z\}_{z \in \Gamma(P)}$ and $\{J_z\}_{z \in \Gamma(P)}$ on the cylindrical ends.

Finally, considering orbits $\{(x_z,\overline{x}_z)\}_{z \in \Gamma(P_-)}$ and $\{x_z\}$ for ${\{z\} = \Gamma(P_+)}$ (note that in all our cases $|P_+| = 1$), of $F_z$ with boundary on $L,$ if $z \in \Gamma(B),$ or of $F_z$ that are closed (1-periodic) if $z \in \Gamma(I),$ we consider the parametric moduli space \[\cM(\{x_z\}_{z \in \Gamma(P)}; \cR, G,\{K_r\},\{J_r\})\] of all the pairs \[\{(u,r)\;| \; r \in {\cR},\; u:S_r \to M,\; (\ast) \}\] satisfying the following asymptotic-boundary value problem $(\ast)$:
\begin{equation}\label{eq:param-Floer} \begin{aligned}& u(\partial S_r) \subset L, \\ & u\circ \eps_z (s,-) \xrightarrow{s \to -\infty} x_z, z \in \Gamma(P_-) \\ & u \circ \eps_z (s,-) \xrightarrow{s \to +\infty} x_z, z \in \Gamma(P_+) \\ & (du - X_{K_r}(u))^{(0,1)} = 0 \\ & E(u) = \int_{S_r} \lvert du - X_{K_r} \rvert^2_J \,\sigma < \infty,\end{aligned}.\end{equation}

For a generic choice of Floer data, $\cM(\{x_z\}_{z \in \Gamma(P)}; \cR,G, \{K_r\},\{J_r\})$ for $\cR = {\cR}_{ij:kl}$ is a smooth manifold, and its zero dimensional component $\cM_0 = \cM_0(\{x_z\}_{z \in \Gamma(P)}; {\cR}_{ij:kl},G, \{K_r\},\{J_r\})$ is compact, and hence consists of a finite number of points. Counting its elements as matrix coefficients yields an operation \[\mu_{\Gamma,\cR}: \bigotimes_{z \in \Gamma(B_-)} CF(L,H_z; \cD) \otimes \bigotimes_{z \in \Gamma(I_-)} CF(H_z; \cD) \to CF(L,H_{z_+}; \cD),\] if $P_+ = B_+ = \{z_+\}$ or \[\mu_{\Gamma,\cR}: \bigotimes_{z \in \Gamma(B_-)} CF(L,H_z; \cD) \otimes \bigotimes_{z \in \Gamma(I_-)} CF(H_z; \cD) \to CF(H_{z_+}; \cD),\] if $P_+ = I_+ = \{z_+\}.$

More precisely, this operation in Floer homology with coefficients in the Novikov field, $\Lambda_{\tmin},$ or $\Lambda_{\tmon},$ is defined by the sum \[ \mu_{\Gamma,\cR}(\otimes_{{z \in \Gamma(P_-)}} (x_z,\overline{x}_z))  = \sum_{x_{z_+}, (u,r) \in \cM_0} (x_{z_+},\overline{x}_{z_+}),\] where the sum runs over all possible values of $x_{z_+} \in \cO(L,H_{z_+};\cD)$ or $x_{z_+} \in \cO(H_{z_+};\cD)$ depending on the case, as above, and the output capping $\overline{x}_{z_+}$ is chosen in such a way that the connect sum of $\{\overline{x_z}\}_{z \in \Gamma(P)}$ and $u,$ has trivial symplectic area, and Maslov class (or Chern class in the absolute case). Standard index formulas, give, in our normalization, that for $i_z = 2n - |(x_{z},\overline{x_{z}})|,$ $z \in \Gamma(I),$ and $i_z = n - |(x_{z},\overline{x_{z}})|,$ $z \in \Gamma(B),$ we have the relation \[ i_{z_+} = \sum_{z \in \Gamma(P_{-})} i_{z} - \dim(\cR).\] 

This operation in Floer homology with coefficients in the Novikov ring $\Lambda_{0,\tuniv}$ is defined similarly by the sum \[ \mu_{\Gamma,\cR}(\otimes_{{z \in \Gamma(P_-)}}\, x_z)  = \sum_{x_{z_+}, (u,r) \in \cM_0} T^{E(u)} \cdot x_{z_+}.\]

Finally, for each given $\{H_z\}_{z \in \Gamma(P)},$ it is easy to see by the fact that $R(K_{r,0}) = 0$ for all $r \in \cR,$ that for each $\epsilon' >0,$ there exists perturbation data  $\{H^{\cD}_z\},$ and $\{K_r\}$ agreeing with this data at the punctures, with the property that $|K_r|_{C^1} < \epsilon_1$ uniformly for all $r \in \cR,$ where the $C^1$-norm is taken in all directions (see \cite{BiranCorneaS-Fukaya}).

Moreover, by the energy formula \[E(u) = \sum_{z_- \in \Gamma(P_-)} \cA_{H,\cD}(x_{z_-},\overline{x}_{z_-}) - \cA_{H,\cD}(x_{z_+},\overline{x}_{z_+}) + \int_{S_r} R(K)(u),\] for a solution of \eqref{eq:param-Floer}, and the fact that $R(K)$ is supported compactly in $\overline{\cS},$ one deduces that for each $\epsilon > 0,$ one can pick $|K_r|_{C^1} < \epsilon_1$ sufficiently small, so that \begin{equation}\label{eq:energy estimate}E(u) \leq \sum_{z_- \in \Gamma(P_-)} \cA_{H,\cD}(x_{z_-},\overline{x}_{z_-}) - \cA_{H,\cD}(x_{z_+},\overline{x}_{z_+}) + \epsilon,\end{equation}
\[\cA_{H,\cD}(x_{z_+},\overline{x}_{z_+}) \leq \sum_{z_- \in \Gamma(P_-)} \cA_{H,\cD}(x_{z_-},\overline{x}_{z_-})  + \epsilon. \]

We note that the moduli spaces $\cR_{11:00}$ and $\cR_{00:11}$ yield by the above construction the relative and absolute Floer continuation maps (after relaxing, in general, the condition $R(K_0) = 0$), while if we take the perturbation datum invariant with respect to a natural $\R$-action on our Riemann surface given by dilations preserving the marked points and their asymptotic markers, we obtain the absolute and relative Floer differential. Both these maps were described above. Let $\Lambda$ denote Novikov coefficients as discussed above. We summarize the new operations obtained in this way from the moduli spaces listed above:

\begin{itemize}\label{list:operations}

\item[$\cR_{01:00}$] \label{mu:0100} yields a map $U:\Lambda \to CF(L,0;\cD),$ with the property that $u_L = U(1) \in CF(L,0;\cD)$ represents the unit in $H_*(L,0;\cD);$

\item[$\cR_{20:00}$] \label{mu:2000} yields a non-degenerate pairing $CF(L,H;\cD) \otimes CF(L,\overline{H};\cD) \to \Lambda;$
\item[$\cR_{21:00}$] \label{mu:2100} yields the product \[\mu_2 = \mu_{2:}: CF(L,H_1;\cD) \otimes CF(L,H_2;\cD) \to  CF(L,H_1 \# H_2;\cD);\] where $H_1 = H_{z_{1,-}}, H_2 = H_{z_{2,-}},$ and $H_{z_+} = H_1 \# H_2.$ 

\item[$\cR_{31:00}$]\label{mu:3100} yields the associator \[\mu_3 = \mu_{3:}: CF(L,H_1;\cD) \otimes CF(L,H_2;\cD) \otimes CF(L,H_3;\cD) \to  CF(L,H_1 \# H_2 \# H_3;\cD),\] where $H_1 = H_{z_{1,-}}, H_2 = H_{z_{2,-}},  H_3 = H_{z_{3,-}}$ and $H_{z_+} = H_1 \# H_2 \# H_3.$ 

\item[$\cR_{11:10}$]\label{mu:1110} yields the module action \[\mu_{1:1}: CF(L,H_1;\cD) \otimes CF(H_2;\cD) \to  CF(L,H_1 \# H_2;\cD);\]

\item[$\cR_{11:20}$]\label{mu:1120} yields the right associator for the module action \[\mu_{1:2}: CF(L,H_1;\cD) \otimes CF(H_2;\cD) \otimes CF(H_3;\cD) \to  CF(L,H_1 \# H_2 \# H_3;\cD);\]

\item[$\cR_{21:10}$]\label{mu:2110}  yields the left associator for the module action \[\mu_{2:1}: CF(L,H_1;\cD) \otimes CF(L,H_2;\cD) \otimes CF(H_3;\cD) \to  CF(L,H_1 \# H_2 \# H_3;\cD);\] below we will use the version where the marked point $z_{\ast}$ is not considered as an edge of the graph, work in a restricted setting preventing bubbling, and look at configurations of trajectories determined by the Floer equation and negative gradient flow trajectories of a Morse function $f$ on $L,$ suitably incident at the image of $z_{\ast};$ this will also give a left associator \[\mu'_{2:1}: C(L,f) \otimes CF^{I_1}(L,H_2;\cD) \otimes CF(H_3;\cD) \to  CF^{I_2}(L, H_2 \# H_3;\cD)\] for suitable action windows $I_1, I_2,$ where $C(L,f)$ is the Morse complex of $f:L \to \R,$ with respect to a suitable auxiliary Riemannian metric;

\item[$\cR_{00:01}$] \label{mu:0001} yields a map $U:\Lambda_{\tmin} \to CF(0;\cD),$ with the property that $u_L = U(1) \in CF(0;\cD)$ represents the unit in $H_*(0;\cD) \cong H_*(M;\Lambda);$

\item[$\cR_{00:20}$] \label{mu:0020}  yields a non-degenerate pairing $CF(L,H;\cD) \otimes CF(L,\overline{H};\cD) \to \Lambda;$

\item[$\cR_{00:21}$] \label{mu:0021} yields the product \[\mu_{:2}: CF(H_1;\cD) \otimes CF(H_2;\cD) \to  CF(H_1 \# H_2;\cD);\] 

\item[$\cR_{00:31}$] \label{mu:0031} yields the associator \[\mu_{:3}: CF(H_1;\cD) \otimes CF(H_2;\cD) \otimes CF(H_3;\cD) \to  CF(H_1 \# H_2 \# H_3;\cD).\]

\end{itemize}

Considering compactness and gluing of Floer trajectories of the above parametric equations, one obtains, as usual, the following few identities (we work over $\mathbb{F}_2,$ and refer to \cite{SeidelBook,Zap:Orient} for a treatment of signs). The compactification of $\cR_{31:00}$ gives us
\begin{equation}\label{eq:mu_2 assoc}\begin{aligned}&\mu_{1:} \mu_{3:} (x,y,z) + \mu_{3:}(\mu_{1:}(x),y,z) + \mu_{3:}(x,\mu_{1:}(y),z) + \mu_{3:}(x,y,\mu_{1:}(z)) = \\ &= \mu_{2:} (\mu_{2:}(x,y),z) + \mu_{2:} (x,\mu_{2:}(y,z)),\end{aligned}\end{equation} where $\mu_{1:}$ is the differential in the Lagrangian Floer complex, and $x,y,z$ are arbitrary chains in $CF(L,H_1; \cD), CF(L,H_2; \cD), CF(L,H_3; \cD).$ This is one term of the usual $A_{\infty}$-relations in Lagrangian Floer homology. Meanwhile, the compactification of $\cR_{11:20}$ gives us
\begin{equation}\label{eq:mu11:1 assoc}\begin{aligned}&\mu_{1:} \mu_{1:2}(x:y,z) +  \mu_{1:2}(\mu_{1:}(x):y,z) + \mu_{1:2}(x:\mu_{:1}(y),z) + \mu_{1:} \mu_{1:2}(x:y,\mu_{:1}(z)) =\\ & =  \mu_{1:1}(\mu_{1:1}(x:y),z) + \mu_{1:1}(x:\mu_{:2}(y,z)),\end{aligned}\end{equation}
where now, in addition, $\mu_{:1}$ is the differential in the absolute Floer homology, and $x,y,z$ are arbitrary chains in $CF(L,H_1; \cD), CF(H_2; \cD), CF(H_3; \cD).$

\subsection{Quantum homology of a Lagrangian submanifold}\label{subsecQuantum}

We briefly recall the construction of the Lagrangian quantum homology (see \cite{BiranCorneaRigidityUniruling,BiranCorneaLagrangianQuantumHomology,Bi-Co:qrel-long} and \cite{Zap:Orient}). Fix pearl data $\cD = (f,\rho,J)$, where $f : L \to \R$ is a Morse function, $\rho$ is a Riemannian metric on $L$, and $J \in \cJ'_M$ is an almost complex structure compatible with $\omega$. Set $\bK = \bF_2.$ Let $C(L; \cD) := \bK \left< \text{Crit}(f)\right> \otimes \Lambda$, be the complex generated by critical points of $f$, where $\Lambda = \Lambda_{\tmin}$. We then grade $C(L;\cD)$ by the Morse indices of $f$ and the grading of $\Lambda$. Given two points $x,y \in \text{Crit}(f)$, and a class $A \in \pi_2(M,L)$, consider the space of sequences $(u_1,\ldots,u_l)$, where:

\begin{itemize}
\item $u_i : (D,\partial D) \to (M,L)$ is a non-constant $J$-holomorphic disc.
\item $u_1(-1) \in W^u(x)$.
\item For each $i \in \{1,\ldots, l-1\}$, $u_{i+1}(-1) \in W^u(u_i(1))$.
\item $y \in W^u(u_l(1))$.
\item $[u_1] + \ldots + [u_l] = A.$
\end{itemize}

Two sequences $(u_1,\ldots,u_l)$ and $(u'_1,\ldots,u'_{l'})$ are considered equivalent if $l=l'$ and for each $i=1,\ldots,l$ there exists $\sigma_i \in Aut(D)$ with $\sigma_i(-1) = -1, \sigma_i(1) = 1$, and $u_i = u'_i \circ \sigma_i$.
Let $\cM(x,y;A;\cD)$ be the quotient space with respect to this equivalence relation. We call the elements in $\cM(x,y;A;\cD)$ pearly trajectories. The virtual dimension of $\cM(x,y;A;\cD)$ is $ind(x) - ind(y) + \mu(A) - 1$. For generic $J$, $\cM(x,y;A;\cD)$ is a smooth compact $0$-dimensional manifold. Denote \[d(x) = \sum_{y,A} \#_2 \cM(x,y;A;\cD) y q^{-\mu(A)},\] and extend $d$ by linearity to $C(L;\cD)$. Denote $QH(L)$ to be the homology of $(C(L;\cD),d)$. We remark that $QH(L)$ has an associative ring structure as well and refer to \cite{BiranCorneaRigidityUniruling,BiranCorneaLagrangianQuantumHomology,Bi-Co:qrel-long, Zap:Orient} for the details.

Let us define the natural action filtration $\cA : C(L;\cD) \to \R$. Take the valuation on $\Lambda_{\text{univ}}$ and consider its pullback $\nu$ under the map $\Lambda_{\text{min}} \to \Lambda_{\text{univ}}$ given by $q \mapsto T^{\kappa N_L}$. Recall that $A_L = \kappa N_L$ is the minimal positive area of a disc with boundary on $L$. For each critical point $x$ of the Morse function $f$, set $\cA(x) = 0$, and extend this to $C(L;\cD)$ by
\[ \cA( \lambda_1 x_1 + \ldots + \lambda_k x_k) = \max_{1 \leq j \leq k}\{\cA(x_j) - \nu(\lambda_j)\}.\]

We consider the quantum homology $QH_m(L)$ of a monotone Lagrangian submanifold $L$ in $M,$ with $N_L \geq 2,$ in a fixed degree $m$ as a persistence module. This is done by the action filtration. It is then easy to see, by the usual proof of the independence of $QH(L)$ on the pearl data, that the persistence module $V_m(L)^t = QH_m(\{\cA \leq t\})$ does not depend on the pearl data.

We require the following statement. Recall that $L$ is called {\em wide} if $QH(L) \cong H(L) \otimes \Lambda.$ Moreover, we call a persistence module $V$ free if it contains no finite bars; equivalently its boundary depth satisfies $\beta(V) = 0.$

\medskip

\begin{prop}\label{prop: beta 0 for wide}
Let $L$ be a wide Lagrangian submanifold of $M$. For each $m \in \Z,$ the persistence module $V_m(L)$ is free.
\end{prop}

This proposition is due to Usher \cite[Theorem 1.7 (ii)]{UsherBD2}. It can alternatively be rather quickly deduced from the Oh spectral sequence \cite{Oh:SpectralSequence} (cf. \cite{Biran:Nonintersections,BiranCorneaRigidityUniruling,Zap:Orient}, and \cite[Appendix A]{BiranMembrez-Cubic}.

\subsection{Piunikhin-Salamon-Schwarz isomorphism}

Fix a regular Floer datum $(H,\cD)$ and a regular quantum datum $\cD^{\text{pearl}}$ for $L$.
There are well-defined ring isomorphisms ${PSS} : QH(L;\cD^{\text{pearl}}) \to HF(L,H; \cD)$ which also commute with the Floer continuation maps. 

Let us briefly recall their construction (for the details see e.g. \cite{PSS,albers08,BiranCorneaLagrangianQuantumHomology,BiranCorneaRigidityUniruling,Kat-Mil:PSS,Bi-Co:qrel-long,Oh-Zhu:PSS}).
Given $x \in \text{Crit}(f)$, $(\gamma,[\overline{\gamma}]) \in \text{Crit}(\cA_H)$, and $A \in \pi_2(M,L)$.
Consider the sequence $(u_2,\ldots,u_l)$ of pseudo-holomorphic discs like in the definition of the differential (see Section \ref{subsecQuantum}), where $x \in W^u(u_l(1))$, and define $u_1 : \R \times [0,1] \to M$ where $u_1(\R,\{0,1\}) \in L$, $u_1(-\infty,t) = \gamma(t)$, and $u_1(\infty,t) = q$ with $u_2(-1) \in W^u(q)$. In addition $u_1$ satisfies the equation
\[ \partial_s u_1 + J(u_1) (\partial_t u_1 - \beta(s) X^t_{H}(u_1)) = 0,\]
where $\beta$ is a cutoff function which decreases and vanishes for $s \geq \frac{1}{2}$, and equals $1$ for $s \leq 0$. We require that $[u_1 \# \overline{\gamma}] + [u_2] + \ldots + [u_l] = A.$ The virtual dimension of the relevant moduli space is $\mu(\overline{\gamma}) + \text{ind}(x) + \mu(A) - n$, where $\text{ind}(x)$ is the Morse index of $x$.

Let $n(x,(\gamma,\overline{\gamma}))$ be the number modulo 2 of the elements in this moduli space if the dimension is zero, and zero otherwise.
Then
\[ \Phi_{PSS}(x) = \sum n(x,(\gamma,\overline{\gamma})) (\gamma,\overline{\gamma}). \]
Extend this map by linearity.

\subsection{Lagrangian spectral invariants}

Spectral invariants were introduced into symplectic topology by Viterbo \cite{Viterbo-specGF} in the context of generating functions, and by Schwarz \cite{Schwarz:action-spectrum} and Oh \cite{Oh-construction} in Hamiltonian Floer theory (see also \cite{Usher-spec}). Here we use the version of spectral invariants for Lagrangian Floer homology, in the monotone case, as developed in \cite{LeclercqZapolsky}; we refer there for a further review of the literature. In this section we work with Novikov coefficients $\Lambda = \Lambda_{\min,R},$ $\Lambda_{\mrm{mon},R}$ or $\Lambda_{\text{univ},R},$ over an arbitrary commutative ground ring $R,$ for the Lagrangian quantum homology, and Floer homology. 

Let $i^{t,\infty}_m :  HF_m(L;H,J)^{<t} \to HF_m(L;H,J)$ be the structure maps of the Floer persistence modules as described in Section \ref{Sec:Floer-persistence}.
Assuming that $(H,L)$ is regular, we define the spectral invariant $c(L;-,H) : QH(L)\setminus \{0\}  \to \R$ to be
\[ c(L;a,H) = c(a,H) = \inf \{ t : {PSS}(a) \in \text{Im}(i^{t,\infty}_*)\} .\] Note that $c(a,H)$ is a left end-point of an infinite interval in the barcode associated to the Floer persistence module. This invariant satisfies a continuity property with respect to the $H$ variable, that allows one to extend it to a map \[c: QH(L) \setminus \{0\} \times C^\infty \big(M\times [0,1] \big) \to \R.\] We formally set $c(0,H) = - \infty,$ and denote \[c_+(H) := c([L],H).\] The following proposition summarizing the properties of the resulting function is proved in \cite{LeclercqZapolsky}.

\medskip

\begin{prop}\label{prop:main_properties_Lagr_sp_invts}
  Let $L$ be a closed monotone Lagrangian of $(M,\omega)$ with minimal Maslov number $N_L \geq 2$. The function
$$c : QH_*(L) \setminus \{0\} \times C^\infty \big(M\times [0,1] \big) \rightarrow \R$$
is well-defined, and satisfies the following properties. 
\begin{Properties} 	
   \item[Spectrality] For $H \in C^\infty \big(M \times [0,1]\big)$, $c(a;H) \in  \Spec(H,L)$.
   \item[Ring action] For $r \in R$, $c(r \cdot a;H) \leq c(a;H)$. In particular, if $r$ is invertible, then $c(r \cdot a;H) = c(a;H)$. If $R = \bK$ is a field, $c(\lambda \cdot a; H) = c(a;H) - \nu(\lambda),$ for all $\lambda \in \Lambda.$
   \item[Symplectic invariance] Let $\psi \in \Symp(M,\omega)$ and $L'=\psi(L)$. Let
   $$c' : QH_*(L') \times C^0\big(M\times [0,1] \big) \to \R$$
   be the corresponding spectral invariant. Then $c(a;H) = c'(\psi_*(a);H \circ \psi^{-1})$. 
   \item[Normalization] If $\alpha$ is a function of time then $$c(a;H+\alpha)=c(a;H) + \int_0^1 \alpha(t) \,dt\,.$$ We have $c(a;0) = \cA(a)$ and $c_+(0)=0$.
   \item[Continuity] For any $H$ and $K$, and $a \neq 0$:
   $$\int_0^1 \min_M (K_t - H_t) \,dt \leq c(a;K) - c(a;H) \leq \int_0^1 \max_M (K_t - H_t) \,dt \,.$$ 
   \item[Monotonicity] If $H \leq K$, then $c(a;H) \leq c(a;K)$.
   \item[Triangle inequality] For all $a$ and $b$, $c(a \ast b; H \# K) \leq c(b;H) + c(a;K)$.
   
   \item[Lagrangian control] If for all $t$, $H_t|_L = c(t) \in \mathbb{R}$ (respectively $\leq$, $\geq$), then
   $$c_{+}(H)=\int_0^1 c(t) \,dt\quad\text{(respectively } \leq, \geq)\,.$$
   Thus for all $H \in \cH$:
   $$\int_0^1 \min_L H_t \,dt \leq c_{+}(H) \leq \int_0^1 \max_L H_t \,dt \,.$$
   
   \item[Duality] Let $ a^\vee \in QH^{n-k}(L)$ be the element corresponding to $a \in QH_k(L)$  under the natural duality isomorphism\footnote{This works for coefficients in $\mathbb{F}_2$ in the relative case, and for arbitrary coefficients in the absolute case. For arbitrary coefficients in the relative case one should take $a^\vee \in QH^{n-*}(L;\cL)$ in the quantum cohomology of $L$ twisted by a suitable orientation local system $\cL$ - see \cite{LeclercqZapolsky}. We refer {\em ibid.} for the definition of the duality isomorphism.}. Then we have
   $$c(a;\overline{H}) = - \inf \big\{c(b;H)\,|\, b \in QH_{n-k}(L;\cL) : \langle a^\vee, b\rangle \neq 0 \big\}\,.$$
   
   \item[Non-negativity] $c_+(H) + c_+(\overline{H}) \geq 0$.
   \item[Maximum] $c(a;H) \leq c_+(H)+\cA(a)$. 
\end{Properties}
\end{prop}

In view of non-negativity we define the spectral norm $\gamma(L;H) := c_+(H) + c_+(\overline{H}) $.
If $N_L > n$, one can use the duality property to get $c_+(\overline{H}) = -c([pt],H) $, and hence $\gamma = c_+(H) - c_-(H)$ where $c_-(H) := c([pt],H)$.

\subsection{Lagrangian Seidel representation} \label{Sect:Seidel}

Here we give a brief description of the Seidel automorphsim $S_{\eta,\sigma} : QH(L) \to QH(L)$ (see \cite{huLalondeLeclercq,huLalonde}, and \cite{charetteCornea} for alternative descriptions).

Denote by $\Ham_L$ the group of Hamiltonian diffeomorphisms $f$ of $M$ that satisfy $f(L) = L$.
Let $\Gamma_L$ be the space of paths $\eta : [0,1] \to \Ham(M)$ such that $\eta^0 = \Id$ and $\eta^1 \in \Ham_L$. Let $\mathcal{P}_L = \pi_0(\Gamma_L)$.

An element $[\eta] \in \cP_L$ defines a canonical up to isomorphism Hamiltonian fibration over $\D$ as follows.
Let $\D_\pm := \D \cap (\HH_{\pm}),$ where $\HH_{+} = \HH \subset \C$ is the upper half plane, and $\HH_{-} = \overline{\HH} \subset \C$ is the lower half plane.
The fibration induced from $\eta$ is given by
\[ P_\eta := M \times \D_+ \cup M \times \D_- / \sim,\]
where
\[ (x,(1-2t,0)) \sim (\eta_t(x),(1-2t,0)) .\]
In other words we glue the two half-discs along $\D \cap \R$ where the $M$-coordinate is glued by $\eta$.

We define a symplectic form $\tau + C \pi^* \omega_0$ on $P_\eta$, where $\tau$ is the coupling form of the Hamiltonian generating $\eta$, $C$ is a large constant, and $\omega_0$ is the area form on $\D$.  Along the $S^1$-boundary we have the Lagrangian submanifold $N := \cup_{t \in S^1} (L,t)$.
Let $\pi : (P_\eta,N) \to (\D,S^1)$ denote the projection.

Let $F : N \to \R$ be a Morse function which satisfies that $F|_{(L,-1)},F|_{(L,1)}$ are Morse functions on $L$, $\text{Crit}(F) = \text{Crit}(F|_{(L,-1)}) \cup \text{Crit}(F|_{(L,1)})$, and $\max(F|_{(L,-1)})+1 < \min(F|_{(L,1)})$.

Let $\rho$ be a generic metric on $P_\eta$. Let $J$ be a compatible almost complex structure on $P_\eta$ which is compatible with the fibration, i.e. $J$ restricted to a fiber is an $\omega$-compatible almost complex structure, and the projection $\pi$ is $(J,j)$-holomorphic with respect to the standard complex structure $j$ on $\D$.
We call a class $B \in \pi_2(P_\eta,N)$ a section class if $\pi_* B \in \pi_2(\D,S^1)$ is the positive generator. We say that $B$ is a fiber class if $B$ is in the image of the map $\pi_2(M,L) \to \pi_2(P_\eta,N)$ induced from the inclusion of a fiber.
Let $x_- \in \text{Crit}(F|_{(L,-1)})$ and $y_+ \in \text{Crit}(F|_{(L,1)})$. We consider pearly trajectories in $(P_\eta,N)$, where exactly one of the $J$-holomorphic discs in the configuration represents a section class, and all other discs represent fiber classes.
This amounts to consider the moduli space $\cM(P_\eta,N;\sigma; F,\rho,J;x_-,y_+)$, where $\sigma \in \pi_2(P_\eta,N)$ is a section class.

One can check that
\[ \dim \cM(P_\eta,N;\sigma; F,\rho,J;x_-,y_+) = ind(x_-) - ind(y_+) + \mu^v(\sigma) ,\]
where $\mu^v$ is the vertical Maslov index.

Let $\sigma \in \pi_2(P_\eta,N)$ denote a particular choice of reference section class.
Denote $f_- := F|_{(L,-1)} :L \to L$ and $f_+ := F|_{(L,1)}: L \to L$.
The chain level Seidel automorphism is
\[ S_{\eta,\sigma} : \text{Crit}(f_-) \otimes \Lambda \to \text{Crit}(f_+) \otimes \Lambda ,\]
defined by
\[ S_{\eta,\sigma}(x_-) = \sum_{B,y_+} \#_2 \cM(P_\eta,N;\sigma; F,\rho,J;x_-,y_+) q^{\nu(B)} y_+ .\]

Let $\eta \in \Gamma_L$ and $\bar{\eta}$ be a capping of $\{\eta^t(x_0)\}_{t \in [0,1]}$ for some $x_0 \in L$.
Let us consider the action of $(\eta,\bar{\eta})$ on $HF(L;H)$:
\[ P_{\eta,\bar{\eta}} : CF(L;H) \to CF(L,K \# H),\]
where $K$ is the Hamiltonian generating $\eta$. For the details of the construction see for example \cite{huLalondeLeclercq}. Figure \ref{PdefFig} shows this construction.

\begin{figure}[h]
	
	\includegraphics[width=\textwidth]{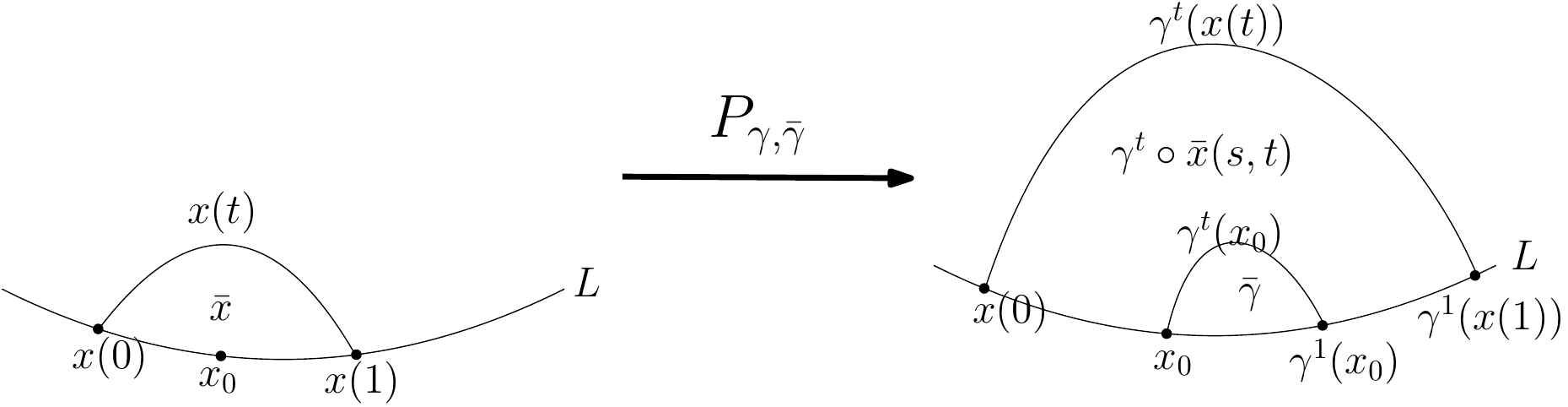}
	\caption{Definition of $P_{\eta,\bar{\eta}}$: For $x$ a path from $L$ to itself, and $\bar{x}$ a capping, we think of $\bar{x}$ as a homotopy from the constant path $x_0$ to $x$. We then get a new path and capping as described in the figure.}
	\label{PdefFig}
	
\end{figure}

In \cite{huLalondeLeclercq} it is proven that for any $\eta$ and any $\sigma$, there is a capping $\bar{\eta}$ so that the following diagram is commutative. In fact there is a bijection between the section classes $\sigma$ and cappings $\bar{\eta}$ with this property.
\[ \begin{matrix} HF(L;H) & \xrightarrow{P_{\eta,\bar{\eta}}} & HF(L;K \# H) \\
                    \downarrow \text{\tiny{PSS}} & & \downarrow \text{\tiny{PSS}} \\
                    QH(L) & \xrightarrow{S_{\eta,\sigma}} & QH(L)
                    \end{matrix}
                    \]
In other words, $\bar{\eta}$ satisfies
\[ PSS \circ P_{\eta,\bar{\eta}}  = S_{\eta,\sigma} \circ PSS.\]

In order to calculate $c(L;[L],\eta \phi)$ one needs to understand how the action changes after applying $P_{\eta,\bar{\eta}}$, and how the class in $QH(L)$ changes after applying $ S_{\eta,\sigma}$.

\medskip

\begin{prop}
	\label{propActionFormula}
	We have the following relation on the action after applying $P_{\eta,\bar{\eta}}$.
	\[(P_{\eta,\bar{\eta}}^* \cA_{K\# H})(x,\bar{x}) = \cA_H(x,\bar{x}) + \cA_K(\eta,\bar{\eta}).\]
\end{prop}
\begin{proof}
	\begin{align*}
	(P_{\eta,\bar{\eta}}^* \cA_{K\# H})(x,\bar{x}) &= \int_{0}^1 K_t(\eta_t(x(t))) + H_t (\eta_t^{-1} \circ \eta_t (x(t))) dt - \int_{D^2} (\eta_t \circ \bar{x}(s,t))^* \omega - \int_{D^2} \bar{\eta}^* \omega \\
	&= \int_0^1 H_t(x(t)) dt + \int_{0}^1 K_t(\eta_t(x(t))) - \int_0^1 \int_0^1 \omega(\eta_{t*} \frac{\partial \bar{x}}{\partial t} + \frac{d\eta_t}{dt}(\bar{x}(s,t)),\eta_{t*} \frac{\partial \bar{x}}{\partial s}) ds dt - \int_{D^2} \bar{\eta}^* \omega \\
	&= \cA_H(x,\bar{x}) + \int_{0}^1 K_t(\eta_t(x(t))) - \int_0^1 \int_0^1 d(K_t \circ \eta_t) \frac{\partial \bar{x}}{\partial s} ds dt - \int_{D^2} \bar{\eta}^* \omega \\
	&= \cA_H(x,\bar{x}) + \int_{0}^1 K_t(\eta_t(x(t))) - \int_0^1 K_t(\eta_t(x(t))) - K_t(\eta_t(x_0)) dt - \int_{D^2} \bar{\eta}^* \omega \\
	&= \cA_H(x,\bar{x}) + \int_0^1 K_t(\eta_t(x_0)) dt - \int_{D^2} \bar{\eta}^* \omega \\
	&= \cA_H(x,\bar{x}) + \cA_K(\eta,\bar{\eta}).
	\end{align*}
\end{proof}

\begin{cor}
	\label{corSpecInv}
	For any class $a \in QH(L)$, and $H \in \cH$ we have
	\[c(L;a,H) = c(L;S_{\eta,\sigma}(a),K \# H) - \cA_K(\bar{\eta}).\]
\end{cor}
\begin{proof}
	Let $p \in HF(L;H)$. From the fact that $PSS \circ P_{\eta,\bar{\eta}}  = S_{\eta,\sigma} \circ PSS$ we get
	\[  PSS(p) = a \iff PSS(P_{\eta,\bar{\eta}} p) = S_{\eta,\sigma} a . \]
	Let $\epsilon > 0$. There exists $p \in HF(L;H)$ with $PSS(p) = a$ and $\cA_H(p) \leq c(a;L,H) + \epsilon$.
	We get that $PSS(P_{\eta,\bar{\eta}} p) = S_{\eta,\sigma} a$ so
	\[ \cA_{K \# H}(P_{\eta,\bar{\eta}} p) \geq c(L;S_{\eta,\sigma}(a),K \# H). \]
	 However from Proposition \ref{propActionFormula} we have $\cA_{K \# H}(P_{\eta,\bar{\eta}} p) = \cA_H(p) + \cA_K(\bar{\eta})$ so we get
	 \[ c(L;S_{\eta,\sigma}(a),K \# H) \leq c(L;a,H) + \cA_K(\bar{\eta}). \]
	 The argument for $\geq$ is similar.
\end{proof}


\section{Lagrangian circle actions}

In this section we prove the upper bounds in Theorem \ref{thm:slimLagrExamples}. It consists of verifying that the pairs $(M,L),$ with $L \subset M$ a Lagrangian submanifold, that appear in Theorem \ref{thm:slimLagrExamples} satisfy the conditions of Proposition \ref{Prop:slimLagr}.

For use in this section, we recall that an $S^1$-action on a space $X$ is called semi-free if the stabilizer of each point $x \in X$ is either $0$ or $S^1.$ The main references for this section are \cite{McDuffTolman} and \cite{hyvrier}.

\subsection{$\C P^n$}

We start with Case (\ref{thmBDbounds-case:CPn}) of Theorem \ref{thm:slimLagrExamples}, and show that it satisfies the conditions of Proposition \ref{Prop:slimLagr}.

\begin{lma}
	\label{lemSeidelCPn-abs}
	Let $M = \CP^n.$ In the construction of $QH(M)$, choose the quantum variable $q$ to be with action $\frac{1}{n+1}$ and degree $2$. Then there exists a subgroup of $\pi_1(\Ham(M))$ isomorphic to $\Z / (n+1)\Z$, with generator $\eta$ such that there exists a choice of section class $\sigma$ for which we have
	\[S_{\eta,\sigma}([pt]) = q^{-n} [M].\]
\end{lma}

\begin{proof}
The suitable loop is $\eta(t): [z_0:z_1:\ldots:z_n] \mapsto [e^{-2 \pi i t}z_0:z_1:\ldots:z_n]$ for $t \in [0,1].$ This follows by the original computation of Seidel element \cite{seidelInvertibles} (see also \cite{McDuffTolman,EntovPolterovichCalabiQM}, and note that this action is semi-free).
\end{proof}

This statement implies an analogous one for $L=\Delta_{\C P^n}$ in $M=\C P^n \times (\C P^{n})^{-},$ by Section \ref{sect: abs as rel}. This shows that $L \subset M$ satisfies the conditions of Proposition \ref{Prop:slimLagr}.

\subsection{$\R P^n$}

We continue with Case (\ref{thmBDbounds-case:RPn}) of Theorem \ref{thm:slimLagrExamples}.

\begin{lma}
	\label{lemSeidelCPn}
	Let $M = \CP^n$ and $L = \RP^n$. In the construction of $QH(L)$, choose the quantum variable $t$ to be with action $\frac{1}{2n+2}$ and degree $1$. Then there exists a subgroup of $\mathcal{P}_L$ isomorphic to $\Z / (n+1)\Z$, with generator $\eta$ for which there exists a choice of $\sigma$ such that we have
	\[S_{\eta,\sigma}([pt]) = t^{- n} [L].\]
\end{lma}
\begin{proof}
    Let $\eta(t): [z_0:z_1:\ldots:z_n] \mapsto [e^{-\pi i t}z_0:z_1:\ldots:z_n]$ for $t \in [0,1]$. Note that $\eta$ concatenated with $\eta\circ \eta(1),$ is a circle action $\eta^2$ that generates a subgroup $\Z / (n+1)\Z \subset \pi_1(\Ham(\CP^n))$ (see \cite{McDuffTolman}). In \cite{hyvrier} it is shown how to calculate the Seidel element of a curve $\eta \in \mathcal{P}_L$, with $\eta^2 \in \pi_1(\Ham)$ a circle action, for a particular choice of $\sigma.$  Following the same notations as in \cite{hyvrier}, we note that the action is semi-free, the fixed point set of $\eta^2$ where the moment map takes its maximal value, is $F_{\max} = \CP^{n-1} \subset \CP^n.$ (The other fixed point set is $F_{\min} = [1:0:\ldots:0].$) Denote $F^L_{\max} := F_{\max} \cap L = \RP^{n-1}$ and observe that it is a Lagrangian submanifold of $F_{\max}.$ The sum of weights of $\eta^2$ in $F_{\max}$ is $w_{\max} = -1$, so we get $S_{\eta,\sigma}([L]) = [\RP^{n-1}] t,$ and the section class we consider is determined by the constant sections at points of $F^L_{\max}$ Hence, by \cite{hyvrier}, we obtain
    \[ S_{\eta,\sigma}([pt]) = [\mathbb{R}\mathbb{P}^{n-1}] \ast [pt] t = t^{-n} [L].\]
\end{proof}

\subsection{$S^n \subset Q^n$ - the sphere in the quadric}

Now we prove that Case (\ref{thmBDbounds-case:Sn}) of Theorem \ref{thm:slimLagrExamples} satisfies the requirements of Proposition \ref{Prop:slimLagr}. We first describe the geometric and topological setting following \cite{BiranCorneaRigidityUniruling},\cite{SmithPencils}, and \cite{SeidelGraded}.

Let $n \geq 2.$ Consider the quadric $Q^n \subset \C P^{n+1}$ defined by \[Q^n = Q^n(q) = \{ [z_0,\ldots,z_{n+1}]\,|\, q(z) = 0 \},\] where $q: \C^{n+1} \to \C$ is a non-degenerate quadratic form. This is a symplectic manifold, endowed with the Kahler form $\om_{Q^n} = \om_{FS}|_{Q^n}.$ By Moser's argument, up to symplectomorphism $(Q^n, \om_{Q^n})$ does not depend on the choice of $q.$ Therefore we will freely choose suitable $q$ at different points. For $n=2,$ by a suitable choice of $q,$ $Q^n$ takes the form of the image of the Segre embedding $\C P^1 \times \C P^1 \hookrightarrow \C P^2,$ and is hence symplectomorphic to $(\C P^1 \times \C P^1, \om_{FS} \oplus \om_{FS}).$ The symplectic manifold $Q_n$ is monotone, with monotonicity constant $2\kappa = \frac{1}{n},$ and minimal Chern number $N_{Q^n} = n.$

The choice $q = \sum_{j=0}^n z^2_j - z^2_{n+1}$ allows us to observe that $Q^n$ admits a Lagrangian sphere $L \cong S^n$, given for this choice by $L = Q^n \cap \R P^{n+1}$ (we refer to \cite[Section 4.1]{SmithPencils} for an alternative interpretation via Lefschetz pencils). This Lagrangian submanifold is monotone, with monotonicity constant $\kappa = \frac{1}{2n},$ and minimal Maslov number $N_L = 2n.$ From now on, we separate the following two cases: $n = 2k$ - even, and $n = 2k+1$ - odd.

In the case when {\em $n = 2k$ is even,} considering $\C P^{n+1}$ with homogeneous coordinates \[ [z_0,\ldots,z_k,w_0,\ldots,w_k],\]  we can describe the quadric as the zero locus $Q^n =\{\sum_{j=0}^{k} z_j w_j = 0\}.$ Our Lagrangian sphere will correspond to \[L = \{w_j = \overline{z_j},\;  1 \leq j \leq k-1,\; z_k \in \R, w_k \in \R\}.\] 

In these new coordinates, $Q^n$ inherits the Hamiltonian $S^1$-action \[\lambda \cdot [z_0,\ldots,z_k,w_0,\ldots,w_k] = [\lambda^{-1}\cdot z_0,\ldots,\lambda^{-1}\cdot z_k,w_0,\ldots,w_k]\] for $\lambda \in S^1,$ considered as a subgroup of $\C^{\times}.$ This action is semi-free. Consider the Hamiltonian loop $\eta_n = \{(e^{2\pi i t}\cdot - )\}_{t \in [0,1]}.$ It tautologically defines a Lagrangian circle action for $L.$ 

In the case when {\em $n = 2k+1$ is odd,} write the homogeneous coordinates on $\C P^{n+1}$ as \[ [u,z_0,\ldots,z_{k},w_0,\ldots,w_{k}],\] and write $Q^n =\{\sum_{j=0}^k z_j w_j + u^2 = 0\}.$ The Lagrangian sphere corresponds to \[L = \{w_j = \overline{z_j},\;  1 \leq j \leq k-1,\; z_k \in \R, w_k \in \R, u \in \R\}.\] 

The quadric $Q^n$ inherits the Hamiltonian $S^1$-action \[\lambda \cdot [u,z_0,\ldots,z_{k},w_0,\ldots,w_{k}] = [\lambda^{-1} \cdot u,\lambda^{-2}\cdot z_0,\ldots,\lambda^{-2}\cdot z_{k},w_0,\ldots,w_{k}]\] for $\lambda \in S^1.$ This action is not semi-free, since all points with $u=0,$ outside the maximum and minimum components, have two-fold isotropy. Nevertheless, it is easy to see that the path $\eta_n = \{(e^{2\pi i t}\cdot - )\}_{t \in [0,1/2]}$ defines a Lagrangian circle action for $L,$ and we shall compute its Lagrangian Seidel element by a direct argument.

\begin{lma}
	\label{lemSeidelQn}
	Let $M = Q^n$ and $L \cong S^n$ as above. Take coefficients in $\bK = \mathbb{F}_2.$ Then 
	
	\begin{enumerate}
	
\item\label{quadric2k:absolute} If $n = 2k$ is even, the absolute Seidel invariant of $\eta_n,$ with respect to a suitable section class $\sigma,$ satisfies \[S^{abs}_{\eta_n,\sigma} = t^n [F],\] for the homology class $[F]$ of the $k$-plane \[F = \{[z_0,\ldots,z_k,w_0,\ldots,w_k]\,|\, z_j =0,\; \forall j \}.\]

\item\label{quadric2k:relative} If $n$ is even, the relative Seidel invariant of $\eta_n,$ with respect to a suitable relative section class $\sigma,$ satisfies \[S^{rel}_{\eta_n,\sigma} = t^{n} [pt], \; S_{\eta_n,\sigma}^2 = [L].\]

\item\label{quadric2k+1:relative} If $n$ is odd, the relative Seidel invariant of $\eta_n,$ with respect to a suitable relative section class $\sigma,$ satisfies \[S^{rel}_{\eta_n,\sigma} = t^{n} [pt], \; S_{\eta_n,\sigma}^2 = [L].\]

	\end{enumerate}
\end{lma}

\medskip

\begin{rmk}
We note that Case \eqref{quadric2k+1:relative} implies that the quantum homology $QH(L,\Lambda_{\tmin})$ with $\bK = \mathbb{F}_2$ is isomorphic as a ring to \begin{equation}\label{eq:sphere QH}\Lambda_{\tmin} [X]/ \langle X^2 = q \cdot 1\rangle,\end{equation} with $1$ corresponding to $[L],$ and $X$ corresponding to $[pt].$ This is known by \cite{BiranCorneaRigidityUniruling} in the case $n$-even. For $\bK = \C,$ the identity \begin{equation}
QH(L,\Lambda_{\tmin}) \cong \Lambda_{\tmin} [X]/ \langle X^2 = c_n\, q \cdot 1\rangle\end{equation} for $c_n \in \C \setminus \{0\}$ holds by \cite{SmithPencils}. We expect that suitably extending Cases \eqref{quadric2k:relative} and \eqref{quadric2k+1:relative} to Novikov coefficients $\Lambda_{\Z,\mathrm{min}}$ with ground ring $\Z,$ for example along the lines of \cite{charetteCornea}, would extend \eqref{eq:sphere QH} to these coefficients, and hence to coefficients in every commutative ground ring. Since this would require a digression on orientations in Lagrangian Floer theory in general, and in the Lagrangian Seidel invariant in particular, this shall appear elsewhere.
\end{rmk}

\medskip

\begin{proof}[Proof of Lemma \ref{lemSeidelQn}]
Case \eqref{quadric2k:absolute} is again a direct consequence of \cite[Theorem 1.10]{McDuffTolman}, since \[\codim F = 2k+2 < 2 N_M = 2n =4k.\] 

Case \eqref{quadric2k:relative} can either be computed directly, along the lines of \cite{hyvrier}, similarly to Case \eqref{quadric2k+1:relative} below, or one can argue as follows. By \cite[Corollary 3.16]{huLalonde}, the relative and absolute Seidel elements of $\eta_n$ are related by \[S^{rel}_{\eta_n,\sigma^{rel}} = S^{abs}_{\eta_n,\sigma^{abs}} \ast [L],\] where $\ast$ denote the structure of $QH(L,\Lambda_{\tmin})$ as a module over $QH(M,\Lambda_{\tmin})$ (see \cite{BiranCorneaLagrangianQuantumHomology,Bi-Co:qrel-long}), and $\sigma^{rel}$ is the relative section class obtained from the absolute section class $\sigma^{abs}$ (see \cite[Lemma 3.13]{huLalonde}). From the latter references, and Case \eqref{quadric2k:absolute}, it is now direct to see that \[S^{abs}_{\eta_n,\sigma^{abs}} \ast [L] = t^n [F] \ast [L] = t^n [F] \circ [L] = t^n [pt].\] Indeed, as $N_L = 2n > n = \dim(L) + 2n - \dim(L) - \dim(F),$ there are no quantum correction terms in this module product.

We turn to Case \eqref{quadric2k+1:relative}. While the Hamiltonian $S^1$-action resulting from doubling $\eta_n$ is not semi-free, and $F^L_{\max} = F_{\max} \cap L = \{p =[0,0,\ldots,0,1]\}$ is not a Lagrangian submanifold of \[F_{\max} = \{[u,z_0,\ldots,z_k,w_0,\ldots,w_k]\,|\, u=0,\, z_j =0\; \forall j \}\] (these are assumptions (A1),(A2) in \cite[Theorem 1.3]{hyvrier}), we will still compute the relative Seidel by modifying the argument in \cite{hyvrier}. Indeed, following the same steps verbatim, with the addition that only terms appearing from the maximal section class may contribute to $S(\eta_n,\sigma_{\max}),$ since $N_L = 2n > n = \codim (p),$ where $p$ is considered as a $0$-dimensional submanifold of $L$ (compare \cite[Proposition 4.3]{hyvrier}), the answer readily follows, given that we prove the regularity of the constant $J$-holomorphic section at the fixed point $p \in L$ of the action. As in \cite{hyvrier}, this regularity follows from the Fredholm regularity of the following Riemann-Hilbert problem in a trivial complex vector bundle over the standard unit disk $\D \subset \C,$ with boundary conditions in a loop of linear Lagrangian subspaces: the trivial complex vector bundle is given by $T_p Q,$ with the standard complex structure, and the linear loop of Lagrangian boundary conditions is given by $L_{e^{i \theta}} = D_p\eta_n(\theta/4 \pi) T_p L,$ where $\theta \in [0,2\pi],$ and $e^{i \theta} \in \partial \D.$

Since $p$ lies in the affine chart given by $w_k = 1,$ we work therein, and see that \[(u,z_0,\ldots,z_{k-1},w_0,\ldots,w_{k-1})\] provide holomorphic coordinates on $Q$ near $p.$ Using these coordinates, we identify $T_p Q$ with $\C^n = \C \times \C^{k} \times \C^{k}$ (recalling that $n = 2k +1$), and $T_p L$ with the Lagrangian subspace $A \cdot \R^n,$ where $A \in GL(n,\C)$ is given by \[A = \id_{\C} \times A',\] \[A' = \begin{pmatrix} \id_{\C^k} & i \cdot \id_{\C^k} \\ \id_{\C^k} & -i \cdot \id_{\C^k} \end{pmatrix} \] In the above coordinates, the Lagrangian loop $\{L_{e^{i\theta}}\} _{\theta \in [0,2\pi]}$ corresponds to \[\tau(\theta) = A(\theta) \cdot \R^n,\] \[A(\theta) = e^{-i\theta/2} \cdot \id_{\C} \times e^{-i\theta} \cdot A'.\] Now we calculate \[A(\theta) \cdot \overline{A}(\theta)^{-1} = e^{-i\theta} \cdot \id_{\C^n},\] whence according to \cite[Theorem 2]{Oh-RiemannHilbert}, the partial indices of our Riemann-Hilbert problem on the trivial bundle $\C^n \times \D$ over the standard disk $\D,$ with boundary conditions on $\tau \subset \C^n \times \partial \D,$ $\tau_{e^{i\theta}} = \tau(\theta) \subset \C^n \times \{e^{i\theta}\},$ are all equal to $-1,$ and this Fredholm problem is regular.

\end{proof}

\subsection{$\bH P^n \subset Gr(2,2n+2)$} Finally, we prove that Case (\ref{thmBDbounds-case:HPn}) of Theorem \ref{thm:slimLagrExamples} satisfies the requirements of Proposition \ref{Prop:slimLagr}.

We recall that $M = Gr(2,2n+2)$ has $N_M = 2n+2,$ and $L = \bH P^n$ is simply connected, whence $N_L = 2 N_M = 4n+4.$ Since $N_L > n_L +1,$ where $n_L = \dim L = 4n,$ $L$ is wide by degree reasons. It is useful to observe that $L$ is the fixed point set of the anti-symplectic involution on $M$ given by right multiplication by the quaternion $j,$ viewing $\C^{2n+2}$ as $\bH^{n+1}.$ Now consider the Hamiltonian $S^1$-action $\eta_n$ on $M$ induced by the following Hamiltonian $S^1$-action on $\C^{2n+2}:$ \[\lambda\cdot (z_1,z_2, \ldots, z_{2n+2}) = (\lambda^{-1}z_1,z_2,\ldots, z_{2n+2}).\] This action is semi-free. Denote by $e \in \C^{2n+2}$ the standard basis vector $(1,0,\ldots,0).$ The maximum set of the action $\eta_n$ is \[F = F_{\max} = \{ E \in Gr(2,2n+2) \,: \, E \perp e \} \cong Gr(2, 2n+1).\] Its minimum set is $F_{\min} = \{E \in Gr(2,2n+2) \,: \, E \ni e \},$ and it has no other fixed points. Applying the McDuff-Tolman theorem \cite{McDuffTolman}, we obtain \[S^{abs}_{\eta_n,\sigma} = [F] t^2,\] where $|t|=1,$ and $\sigma$ is the constant section class corresponding to fixed points in $F.$ Now we compute \[S^{rel}_{\eta_n,\sigma} = S^{abs}_{\eta_n,\sigma} \ast [L] = t^2 [F] \ast [L] = t^2 [H],\] where $H \subset L$ is the subspace $H = \{[h_0,\ldots,h_n] \in \bH P^n\,:\, h_0 = 0\} \cong \bH P^{n-1}.$ Here $\ast$ denotes the module action of $QH(M)$ on $QH(L),$ and $[F] \ast [L] = [F] \circ [L] = [H]$ since by degree reasons there are no quantum corrections to the module action.

Finally, we observe that by degree reasons again, since $N_L = 4n+4 > n_L = 4n,$ we have for all $1 \leq k \leq n,$ $[H]^{\ast k} = [H] ^{\circ k}= [H_k],$ for $H = \{[h_0,\ldots,h_n] \in \bH P^n\,:\, h_j = 0,\, \forall 1 \leq j \leq k\} \cong \bH P^{n-k}.$ Note that $H_1 = H$ and \[[H]^{\ast n} = [pt].\] Moreover, \[[H]^{\ast (n+1)} = q^{-1} [L] = t^{-N_L} [L],\] where we used a degree calculation, and the fact that $[H] = t^{-2} S^{rel}_{\eta_n,\sigma}$ is invertible.

\section{Interleavings and continuation elements} \label{sec:filteredYoneda}

In this section we prove Proposition \ref{prop: Lipschitz general coarse}, by applying the philosophy of the filtered Yoneda lemma introduced in \cite{BiranCorneaS-Fukaya} for somewhat different purposes (see also \cite{FukayaImmersed}). The basic idea is that each continuation map in Floer homology is a part of an isomorphism of suitable $A_{\infty}$-modules, and as such, whenever the modules are homologically unital, it can be given, in homology, by the operator of multiplication by a suitable homology class. This class is called a {\em continuation element} in the literature (cf. \cite{AK-simplehomotopy,AS-KhovanovFloer,LekiliPascaleff-equivariant,GPS-covariantly}). Finally, considering the chain level and filtrations, it turns out beneficial to replace the continuation map by the operator of multiplication with the cycle of least action that represents the continuation element. This yields sharper bounds on the distance between Floer persistence modules, and in particular the following statement, whose proof occupies this section.

\begin{prop}\label{prop: Lipschitz general coarse}
	Let $L$ in $M$ be weakly monotone with $QH(L) \neq 0.$ Then for each $r \in \Z,$ $F\in \cH, \;G \in \cH$ there exists $c \in \R,$ such that
	\[d_{\mrm{inter}}(V_r(L,F),V(L,G)[c]) \leq \frac{1}{2}(\gamma(L,G \# \overline{F}) + \beta(L,0)),\] where $\beta(L,0)$ is the boundary depth of the pearl complex $(C(L;\cD),d_{(L;\cD)})$ computing $QH(L).$ In fact $c = - \frac{1}{2} (\beta(L,0) - c([L],G \# \overline{F}) + c([L], F \# \overline{G})).$
\end{prop}

\begin{rmk}
	In fact, a closer look at the proof shows that $\beta(L,0)$ in Proposition \ref{prop: Lipschitz general coarse} can be replaced by $\beta_n(L,0),$ the maximal length of a finite bar obtained from the degree $n$ persistence module. 
\end{rmk}

\begin{rmk}
Note that in particular, Proposition \ref{prop: Lipschitz general coarse} answers Usher's question \cite{Usher-private} for all weakly monotone $L \subset M,$ with $QH(L) \neq 0,$ with, perhaps, the weaker bound $\beta(L,H) \leq A + \beta(L,0),$ for all $H \in \cH.$
\end{rmk}

Proposition \ref{prop: Lipschitz general coarse} combined with Proposition \ref{prop: beta 0 for wide}, immediately gives Theorem \ref{theorem: Lipschitz wide}.

\subsection{Filtered multiplication operators}

For $F,G \in \cH,$ set $\Delta_{F,G} := G \# \overline{F}.$ Fix $\epsilon > 0.$ 
Now take three functions $F,G,H \in \cH.$ It is clear that there is Floer perturbation data $\cD,$ such that $F^{\cD},J^{F,\cD},$  $G^{\cD},J^{G,\cD},$  $H^{\cD},J^{H,\cD},$ $\Delta_{F,G}^{\cD},J^{\Delta_{F,G},\cD},$ $\Delta_{G,H}^{\cD},J^{\Delta_{G,H},\cD}$ are regular, and the Hamiltonian terms $K^{F,\cD},K^{G,\cD},K^{H,\cD},K^{\Delta_{F,G},\cD},K^{\Delta_{G,H},\cD}$ are of $C^2$-norm $\ll \epsilon.$ We observe that if $\cD$ is suitably chosen on the moduli space of discs with $3$ and $4$ boundary punctures, with cylindrical ends marked by the above Floer data, and boundary conditions on $L,$ then, for homogeneous cycles $x \in CF_k(L,\Delta_{F,G}; \cD),$ $y \in CF_l(L,\Delta_{G,H}; \cD),$ with $a =\cA(x), \, b = \cA(y)$ and for $\eps' \ll \eps$ the maps \[\mu_2(x,-): CF_*(L,F; \cD) \to CF_{*+k -n}(L,G; \cD)[a+\eps']\]
\[\mu_2(y,-): CF_*(L,G; \cD) \to CF_{*+l - n}(L,H; \cD)[b+\eps']\] \[\mu_2(y,-): CF_{*}(L,\Delta_{F,G}; \cD) \to CF_{*+l - n}(L,\Delta_{F,H}; \cD)[b + \eps']\]
are well-defined and filtered. Moreover, by \eqref{eq:mu_2 assoc} and the discussion in Section \ref{sec:Products}, the following {\em associativity} relation holds.

\begin{lma}\label{lem: assoc} The following chain maps, obtained by post-composing with suitable inclusions,
\[\mu_2(y,\mu_2(x,-)): CF_{*}(L,F; \cD) \to CF_{*+k+l - 2n}(L,H; \cD)[a+b + 3\eps']\]
\[\mu_2(\mu_2(y,x),-): CF_{*}(L,F; \cD) \to CF_{*+k+l - 2n}(L,H; \cD)[a+b + 3\eps']\]
are well-defined, filtered, and filtered-chain-homotopic.
\end{lma}

\subsection{Filtered continuation elements}

First, since $QH(L)$ is a unital algebra (over the Novikov field $\Lambda_{\mathrm{mon},\bK}$), the condition $QH(L) \neq 0$ holds if and only if its unit $[L]$ does not vanish. 

Consider two homogeneous cycles $x \in CF_n(L,\Delta_{F,G}; \cD),$ $y \in CF_n(L,\Delta_{G,F}; \cD),$ with \[[x] = PSS_{\Delta_{F,G},\cD}([L]), [y] = PSS_{\Delta_{G,F},\cD}([L]),\] and action $a =\cA(x), \, b = \cA(y)$ satisfying \[c([L],\Delta_{F,G}; \cD) \leq a < c([L],\Delta_{F,G}; \cD) + \eps',\] \[c([L],\Delta_{G,F}; \cD) \leq b < c([L],\Delta_{G,F}; \cD) + \eps',\] and $\eps' \ll \eps.$ All these numbers are finite, since $[L] \neq 0.$

Now we note that $[\mu_2(y,x)] = PSS_{0,\cD}([L])$ and $[\mu_2(x,y)] = PSS_{0,\cD}([L]).$ It is easy to see that for a choice of perturbation that is sufficiently small in $C^1$-norm, $PSS_{0,\cD}([L])$ is represented by a cycle $z \in CF(L,0; \cD)$ with $\cA(z) < \epsilon'/2 \ll \epsilon.$ Indeed, $z$ can be taken to be the image under the chain-level PSS map $\Phi_{PSS}: C(L;\cD^{\text{pearl}}) \to CF(L,0; \cD)$ of a chain $w$ in the pearl complex of $L$ of filtration level $\cA(w) = 0,$ that represents $[L] \in QH_n(L).$ Moreover the persistence module $V_*(L,0; \cD)$ is $\epsilon'/2$-interleaved, by the PSS map and the PSS map in the reverse direction, with the persistence module $V_*(L;\cD^{\text{pearl}})$ determined by the pearl complex for Lagrangian quantum homology, whence \[\beta(L,0; \cD) \leq \beta(L,0) + \epsilon'.\] The interleaving bounds, as well as the bound on $\cA(z),$ follow from standard action estimates (cf. \cite[Lemma 3.11]{CorneaS} and references therein).

Finally, by the same argument as for associativity, together with standard action estimates for the PSS map, we note that the multiplication maps \[m_2(z,-): CF_*(L,F; \cD) \to CF_*(L,F; \cD)[\eps'], \; m_2(z,-): CF_*(L,G; \cD) \to CF_*(L,G; \cD)[\eps'],\] are filtered chain-homotopic to the standard inclusions, and hence induce the $\eps'$-shift maps on the respective persistence modules.
Now, by Lemma \ref{lem: assoc}, the compositions (post-composed with suitable inclusion maps) \[\mu_2(y,\mu_2(x,-)): CF_{*}(L,F; \cD) \to CF_{*}(L,F; \cD)[a+b + 3\eps'],\] \[\mu_2(y,\mu_2(x,-)): CF_{*}(L,G; \cD) \to CF_{*}(L,G; \cD)[a+b + 3\eps']\] of the maps \[\mu_2(x,-): CF_*(L,F; \cD) \to CF_{*}(L,G; \cD)[a+\eps'],\] \[\mu_2(y,-): CF_*(L,G; \cD) \to CF_{*}(L,F; \cD)[b+\eps']\] are filtered chain-homotopic to the multiplication operators \[\mu_2(\mu_2(y,x),-): CF_{*}(L,F; \cD) \to CF_{*}(L,F; \cD)[a+b + 3\eps'],\] \[\mu_2(\mu_2(y,x),-): CF_{*}(L,G; \cD) \to CF_{*}(L,G; \cD)[a+b + 3\eps'].\]

However $\mu_2(y,x) \in CF_{n}(L,0; \cD)$ satisfies $\mu_2(y,x) = z + d b_{y,x},$ for a chain $b_{y,x}$ with $\cA(b_{x,y}) \leq a+b+ \beta(L,0) + 2\epsilon',$ and similarly $\mu_2(x,y) \in CF_{n}(L,0; \cD)$ satisfies $\mu_2(x,y) = z + d b_{x,y},$ for a chain $b_{x,y}$ with $A(b_{x,y}) \leq a+b+ \beta(L,0) + 2\epsilon'.$ Therefore, composed with suitable inclusions, we obtain the fact that the compositions \[\mu_2(y,\mu_2(x,-)): CF_{*}(L,F; \cD) \to CF_{*}(L,F; \cD)[a+b + \beta(L,0) + 5\eps'],\] \[\mu_2(y,\mu_2(x,-)): CF_{*}(L,G; \cD) \to CF_{*}(L,G; \cD)[a+b + \beta(L,0) + 5\eps']\] are filtered chain-homotopic to the multiplication operators \[\mu_2(z,-): CF_{*}(L,F; \cD) \to CF_{*}(L,F; \cD)[a+b + \beta(L,0) + 5\eps'],\] \[\mu_2(z,-): CF_{*}(L,G; \cD) \to CF_{*}(L,G; \cD)[a+b + \beta(L,0) + 5\eps'],\] with the last two chain-homotopies being given by $\mu_2(b_{y,x},-), \; \mu_2(b_{x,y},-).$ This implies immediately that on the level of filtered homology, we obtain the identity of persistence module morphisms:
\[[\mu_2(y,-)] \circ [\mu_2(x,-)] = [\mu_2(z,-)] = sh_{\rho_{a,b,L} + 6\eps'} : V_{*}(L,F; \cD) \to V_{*}(L,F; \cD)[\rho_{a,b,L} + 6\eps'], \]
\[[\mu_2(x,-)] \circ [\mu_2(y,-)] = [\mu_2(z,-)] = sh_{\rho_{a,b,L} + 6\eps'} : V_{*}(L,G; \cD) \to V_{*}(L,G; \cD)[\rho_{a,b,L} + 6\eps']\] for shift \[\rho_{a,b,L} = a+b+ \beta(L,0).\]

Finally, we recall that if $(L,F)$ and $(L,G)$ are regular, then the Hamiltonian terms $K^F, K^G$ in the perturbation data can be chosen to be identically zero. By continuity of spectral invariants, we can assume that \[a < c([L],\Delta_{F,G}) + 2\eps',\] \[b < c([L],\Delta_{G,F}) + 2\eps'.\] Hence, as \[c([L],\Delta_{F,G}) + c([L],\Delta_{G,F}) = \gamma([G] [F]^{-1}),\] denoting \[\rho_{F,G,L} = \frac{1}{2} {\left( \beta(L,0) + \gamma([G] [F]^{-1}) \right)} \geq 0,\] we finally obtain the identity of persistence module morphisms:
\[[\mu_2(y,-)] \circ [\mu_2(x,-)]  = sh_{2\rho_{F,G} + 8\eps'} : V_{*}(L,F; \cD) \to V_{*}(L,F; \cD)[2\rho_{F,G,L} + 8\eps'], \]
\[[\mu_2(x,-)] \circ [\mu_2(y,-)]  = sh_{2\rho_{F,G} + 8\eps'} : V_{*}(L,G; \cD) \to V_{*}(L,G; \cD)[2\rho_{F,G,L} + 8\eps'].\]  

The multiplication operators can now be defined as maps (composed with suitable inclusions): \[\mu_2(x,-): V_*(L,F; \cD) \to V_{*}(L,G; \cD)[c([L],\Delta_{F,G}) + 4\eps']\] \[\mu_2(y,-): V_*(L,G; \cD) \to V_{*}(L,F; \cD)[c([L],\Delta_{G,F}) + 4\eps']\] 

This means that $V_*(L,F; \cD)$ and $V_*(L,G; \cD)[\sigma_{F,G,L}]$ are $(\rho_{F,G,L} + 4\eps')$-interleaved, where \[\sigma_{F,G,L} = -\frac{1}{2} \left( \beta(L,0) - c([L],\Delta_{F,G}) + c([L],\Delta_{G,F}) \right).\]

Now, when $(L,F),(L,G)$ are non-degenerate, choosing $\cD$ with $K^{(L,F)} \equiv 0, K^{(L,G)} \equiv 0,$ the persistence modules $V_*(L,F; \cD), V_*(L,G; \cD) [\sigma_{F,G,L}]$ depend only of $F,G$ (hence we omit $\cD$ from the notation), and are $(\rho_{F,G,L} + 4\eps')$-interleaved for all $\eps' > 0$ sufficiently small. We conclude that $V_*(L,F)$ and $V_*(L,G) [\sigma_{F,G,L}]$ are $\rho_{F,G,L}$-interleaved, and remind the reader that \[\rho_{F,G,L} = \frac{1}{2} {\left( \beta(L,0) + \gamma([G] [F]^{-1}) \right)}.\]  This finishes the proof.

\section{A Chekanov type theorem for the spectral norm} \label{Sect:Chekanov}

\subsection{Non-displacement}\label{subsec:non-displ}

In this section we prove Theorem \ref{thm:Chekanov}, using notions from Section \ref{sec:filteredYoneda} and the approach in \cite{Chekanov,Char,BarraudCorneaSerre,CorneaS}. Since these methods are by now quite standard, we outline the main points, leaving the details to the interested reader. We start with a slight generalization of $\hbar(J,L),$ which we defined for $J \in \cJ'_M$ so far.

\begin{df}\label{def: hbart}
Let $(M,\omega)$ be a closed symplectic manifold and $L \subset M$ a closed connected Lagrangian submanifold. For $J = \{J_t\}_{t \in [0,1]} \in \cJ_M$ we set $\hbar(J,L) > 0$ to be the minimal area of a non-constant $J_0$ or $J_1$ holomorphic disk on $L,$ or a non-constant $J_t$-holomorphic sphere in $M$ for some $t \in [0,1].$ If no such disks or curves exist, set $\hbar(J,L) = +\infty.$
\end{df}

\begin{rmk}\label{rmk: hbar and hbart}
For $J' \in \cJ'_M$ and $F \in \cH,$ consider $J = \{(\phi^t_F)_* J'\} \in \cJ_M$ It is immediate that in this case $\hbar(J,L) = \min\{\hbar(J',L),\hbar(J',(\phi^1_F)^{-1}(L))\}.$
\end{rmk}

\begin{rmk}\label{rmk: hbar is continuous}
By a simple application of Gromov compactness, $\hbar(J,L)$ is lower semi-continuous in the $C^{\infty}$-topology in the variable $J \in \cJ_M.$
\end{rmk}

We continue by observing that given $F \in \cH$ and $J' \in \cJ'_M,$ for each interval $[v,w)$ of length $0 < w-v < \hbar = \hbar(J',L),$ with $v,w \notin \Spec(L,F) \cup \Spec(L,0)$ the Floer complexes $CF(L,0;\cD)^{[v,w)},$ $CF(L,F;\cD)^{[v,w)}$ generated by capped orbits in $\til{\cO}_{\eta}(L,F; \cD)$ with actions in $[v,w)$ are well-defined, for perturbation data $\cD$ with Hamiltonian term sufficiently $C^2$-small, and almost complex structure $J^{\cD}$ being a small perturbation of $J'.$ This follows by Gromov compactness, and associated standard bubbling analysis (see \cite[Section 3.5]{CorneaS}). We denote their homologies by $HF(L,0;\cD)^{[v,w)},$ $HF(L,F;\cD)^{[v,w)}.$ We focus on the contractible class $\eta = pt$ of chords in this section. We remark that $CF(L,F;\cD)^{[v,w)}$ and $HF(L,F;\cD)^{[v,w)}$ are naturally modules over $\Lambda_{L,\Gamma_{\om},\bK}.$

Let $\phi \in \Ham(M,\om)$ be a Hamiltonian diffeomorphism with $\gamma(\phi) < \hbar(J',L)$ for $J' \in \cJ'_M.$ For the rest of this section, fix a constant $\delta > 0,$ such that $\delta \ll \hbar(J',L) - \gamma(\phi).$

There exists $H \in \cH$ with $\phi = \phi^1_H,$ and perturbation data $\cD,$ with complex structure $J^{\cD}$ being a small perturbation of $J',$  such that \[c([M], H; \cD) + c([M], \overline{H}; \cD) < \hbar(J^{\cD},L).\] Note that in view of Remark \ref{rmk: hbar is continuous}, $\hbar(J^{\cD},L)$ can be made arbitrarily close to $\hbar(J,L),$ by suitably choosing $\cD,$ and hence we may assume that $\hbar(J^{\cD},L) > \hbar(J,L) - \delta,$ and use the continuity property of spectral invariants.

We pick representatives $x \in CF(H; \cD), \; y \in CF(\overline{H}; \cD)$ with $[x] = PSS([M]),\; [y] = PSS([M])$ and $\cA_{H;\cD}(x) = c([M], H; \cD) = a,\; \cA_{H;\cD}(y) = c([M], \overline{H}; \cD) = b.$ It is convenient for the next section to choose $F = \overline{H}$ to compute Floer homology in action windows. Note that $\gamma([H]) = \gamma([H]^{-1}) = \gamma([\overline{H}]),$ and the same identity holds for the relative spectral pseudo-norm $\gamma(L,-).$

Now, applying Section \ref{sec:Products}, for suitable consistent choices of perturbation data $\cD$ for the various functions involved (see Section \ref{sec:filteredYoneda}) from the moduli space $\cR_{11:10}$ we obtain multiplication operators \[\mu_{1:1}(-:x): CF(L,0;\cD)^{[-a-b-4\delta,4\delta)} \to CF(L,\overline{H};\cD)^{[-b-3\delta,a+5\delta)},\] \[\mu_{1:1}(-:y): CF(L,\overline{H};\cD)^{[-b-3\delta,a+ 5\delta)} \to CF(L,0;\cD)^{[-2\delta,a+b+6\delta)},\] where $\delta > 0$ is sufficiently small so that \begin{equation}\label{eq:threshold} a+b + 10\delta < \hbar(J^{\cD},L).\end{equation} 
Note that these operators induce chain maps, because bubbling is prohibited by \eqref{eq:threshold}, and hence induce multiplication maps
\begin{equation}\label{eq:ext mult operators}
\begin{aligned}
 & [\mu_{1:1}(-:x)]: HF(L,0;\cD)^{[-a-b-4\delta,4\delta)} \to HF(L,\overline{H};\cD)^{[-b-3\delta,a+5\delta)},\\
 & [\mu_{1:1}(-:y)]: HF(L,\overline{H};\cD)^{[-b-3\delta,a+ 5\delta)} \to HF(L,0;\cD)^{[-2\delta,a+b+6\delta)}.
\end{aligned}
\end{equation}

Moreover, choosing pertrubation data $\cD$ with Hamiltonian terms sufficiently small relative to $\delta,$ we obtain from \eqref{eq:mu11:1 assoc}, and the bubbling threshold condition \eqref{eq:threshold}, the associativity relation: \begin{equation}\label{eq:assoc for mu_11}[\mu_{1:1}(-:y)] \circ [\mu_{1:1}(-:x)] = [\mu_{1:1}(-: \mu_{:2}(y,x))]: HF(L,0;\cD)^{[-a-b-4\delta,4\delta)} \to   HF(L,0;\cD)^{[-2\delta,a+b+6\delta)}.\end{equation}

In turn, $\mu_{:2}(y,x) \in CF(0;\cD)$ represents the class $[M]$ of the unit. Hence, if $\cD$ is chosen sufficiently small, $\mu_{:2}(y,x) = z + db_{y,x},$ where $\A_{0;\cD}(z) \ll \delta,\; \cA_{0; \cD}(b_{y,x}) \ll \delta,$ and $z = PSS_{0;\cD}(m),$ where $m$ is the unique maximum of a Morse function $f$ on $M.$ It is now clear that $[\mu_{1:1}(-: \mu_{:2}(y,x))]$ induces the interval shift map \[[\mu_{1:1}(-:z)]: HF(L,0;\cD)^{[-a-b-4\delta,4\delta)} \to   HF(L,0;\cD)^{[-\delta,a+b+7\delta)}.\]

This map is non-zero, and in fact has rank at least $\dim_{\bK} H_*(L;\bK),$ since it commutes with the PSS-type injections \begin{equation}\label{eq:pss}\begin{aligned} & pss_{L,0,\cD, [-a-b-4\delta,4\delta)}: H_*(L;\bK) \hookrightarrow HF(L,0;\cD)^{[-a-b-4\delta,4\delta)},\\ & pss_{L,0,\cD,[-\delta,a+b+7\delta)}: H_*(L;\bK) \hookrightarrow HF(L,0;\cD)^{[-\delta,a+b+7\delta)} \end{aligned}\end{equation} that are well defined because of \eqref{eq:threshold}. Moreover, these injections admit left inverses  \begin{equation}\label{eq:pss overline}\begin{aligned} & \overline{pss}_{L,0,\cD, [-a-b-4\delta,4\delta)}: HF(L,0;\cD)^{[-a-b-4\delta,4\delta)} \to  H_*(L;\bK),\\ & \overline{pss}_{L,0,\cD,[-\delta,a+b+7\delta)}: HF(L,0;\cD)^{[-\delta,a+b+7\delta)} \to  H_*(L;\bK).\end{aligned}\end{equation}

Finally by \eqref{eq:assoc for mu_11} we obtain that $\dim_{\bK} HF(L,\overline{H};\cD)^{[-b-3\delta,a+5\delta)} \geq \dim_{\bK} H_*(L;\bK).$ If $\phi^1_H(L) \cap L = \emptyset,$ then $L \cap \phi^1_{\overline{H}} L = L \cap (\phi^{1}_H)^{-1} (L) = \emptyset,$ and for sufficiently small perturbation data $\cD,$ we have $L \cap \phi^1_{{\overline{H}}^{\cD}}(L)  = \emptyset,$ in contradiction with this estimate. 

If $L \pitchfork \phi(L) = L \pitchfork \phi^1_H(L),$ and $k:=\# L \cap \phi(L) = \# L \cap \phi^{-1} (L)$ satisfies $k < l := \dim_{\bK} H_*(L,\bK),$ then one can choose perturbation data $\cD = (J^{L,\overline{H}},K^{L,\overline{H}})$ with $K^{L,\overline{H}} \equiv 0.$ Then $CF(L,\overline{H};\cD)^{[v,w)}$ will be generated over $\bK$ by the capped orbits $(x,\overline{x})\in \til{\cO}_{pt}(L,\overline{H})$ with $x \in \cO_{pt}(L,\overline{H})$ and $v \leq \A_{L:\overline{H}}(x,\overline{x}) < w.$ We set $v = - b - 3\delta,$ $w = a+ 5\delta.$

Now consider a basis $(v_1,\ldots,v_l)$ of $H_*(L,\bK),$ and let $(v'_1,\ldots,v'_l),$ respectively $(y_1,\ldots,y_l),$ be its image under $pss_{L,0,\cD, [-a-b-4\delta,4\delta)},$ respectively $[\mu_{1:1}(-: x)] \circ pss_{L,0,\cD, [-a-b-4\delta,4\delta)}.$ Consider chain representatives $\til{Y} = (\til{y}_1,\ldots,\til{y}_l)$ of $(y_1,\ldots,y_l)$ in $CF(L,\overline{H};\cD)^{[-b-3\delta,a+5\delta)}.$  

It is straightforward to deduce from $k < l$ that the chains $\til{Y}$ are contained in a $\Lambda_{0,L,\Gamma_{\om},\bK}$-submodule of $CF(L,\overline{H};\cD)^{[-b-3\delta,a+5\delta)}$ with $k$ generators over $\Lambda_{0,L,\Gamma_{\om},\bK}.$ Representing $\til{Y}$ in terms of the $k$ generators, and applying the Smith normal form\footnote{While, similarly to $\Lambda_{0,\tuniv},$ the ring $\Lambda_{0,L,\Gamma_{\om},\bK}$ is not in general a principal ideal domain, it is easy to see that each of their {\em finitely generated} ideals is principal, and that the theory of the Smith normal form of matrices applies in this setting.} over $\Lambda_{0,L,\Gamma_{\om},\bK},$ shows that there are elements $c_1,\ldots, c_l \in \Lambda_{0,L,\Gamma_{\om},\bK}$ such that \begin{equation} \label{eq:val0} \sum_{j} c_j \cdot y_j = 0,\end{equation} and moreover, at least one of $c_1,\ldots,c_l$ has valuation $0,$ so that the vector $(c_{1,0},\ldots,c_{l,0}) \in \bK^l,$ of valuation $0$ parts of $c_{1},\ldots,c_l,$ is non-zero in $\bK^l.$ Applying $\overline{pss}_{L,0,\cD,[-\delta,a+b+7\delta)} \circ [\mu_{1:1}(-:y)]$ to \eqref{eq:val0}, we obtain the relation $ \sum_{j} c_{j,0} \cdot v_j = 0,$ in contradiction to $(v_1,\ldots,v_l)$ being a basis of $H_*(L,\bK).$ Indeed it is easy to see that for each $y_j$ and each element $c_{+} \in \Lambda_{0,L,\Gamma_{\om},\bK}$ with $\val(c_{+}) > 0,$ \[\overline{pss}_{L,0,\cD,[-\delta,a+b+7\delta)} \circ [\mu_{1:1}(-:y)] (c_{+} \cdot y_j) =\overline{pss}_{L,0,\cD,[-\delta,a+b+7\delta)} (c_{+} \cdot [\mu_{1:1}(-:z)](v'_j))= 0.\]

\begin{rmk}
Another proof, using Floer homology with coefficients in $\Lambda_{0,\tuniv},$ is also possible.
\end{rmk}

\subsection{Non-degeneracy}\label{subsec:non-deg}

We continue with the setup of Section \ref{subsec:non-displ} to prove Theorem \ref{thm:spectral_norm non-deg}. As in \cite[Sections 3.8, 3.9]{CorneaS}, this statement relies on $HF(L,0; \cD)^{[v,w)}$ and $HF(L,\overline{H};\cD)^{[v,w)}$ where $v < 0 < w,$ and $w - v < \hbar(J,L),$ and $\cD$ is a perturbation datum with sufficiently small Hamiltonian part, being modules over $HF_*(L; \bK).$ This module structure is given by requiring in the Floer equation that the image one boundary marked point in $\cS_{21:00}$ over $\cR_{21:00}$ be incident with a negative gradient trajectory defined for times $t \in (-\infty,0]$  of a given Morse function $f$ on $L,$ at $t=0.$ Similarly $HF_*(L; \bK)$ is a module over itself, given by the intersection product, which can be described in Morse-theoretical terms by requiring the negative gradient flow-lines of another given Morse function $g$ to be incident at a given interior point with negative gradient half-flow-lines of $f.$  Moreover, the maps $[\mu_{1:1}(-: x)],$ $ [\mu_{1:1}(-: y)],$ $ pss_{L,0,\cD, [-a-b-4\delta,4\delta)},$ $ pss_{L,0,\cD,[-\delta,a+b+7\delta)},$ $ \overline{pss}_{L,0,\cD, [-a-b-4\delta,4\delta)},$ $ \overline{pss}_{L,0,\cD,[-\delta,a+b+7\delta)}$ from \eqref{eq:ext mult operators},\eqref{eq:pss}, and \eqref{eq:pss overline}, are module morphisms with respect to these module structures. This is shown analogously to the associativity properties of the various maps we have considered, notably using the compactified moduli space $\overline{\cR}_{21:10}$ for the case of $[\mu_{1:1}(-: x)],$ $ [\mu_{1:1}(-: y)].$

Now given an embedding $e:B_r \to M,$ as in Definition \ref{def: relative Gromov width} of $w(L;L')$ for $L' = \phi^1_H(L) \neq L,$ with $\frac{\pi r^2}{2} > w(L;L') - \epsilon',$ for $\epsilon' \ll \epsilon,$ we choose an almost complex structure $J_\epsilon \in \cJ'_M$  that coincides with $e_*(J_{st})$ on $e(B_{r'})$ where \[\frac{\pi (r')^2}{2} > w(L;L') - 2\epsilon'> w(L;L') - \epsilon.\] Note that by the formulation of the statement, we can assume that \[\gamma(H) < \min\{\hbar(J_{\epsilon},L), \hbar(J_{\epsilon}, \phi^1_H(L))\} = \hbar(J,L)\] for $J = \{J_t\},\; J_t = (\phi^t_{\overline{H}})_* J_\epsilon$ and prove that in this case $\gamma(H) \geq w(L;L') - \epsilon.$

Now we choose $f$ to have a unique minimum at $R = e(0),$ and note that by the module-morphism property above, and the obvious relation $[pt] \circ [L] = [pt] \neq 0$ in $H_*(L;\bK),$ we obtain that there exists a solution $u:\R \times [0,1] \to M$ with boundary on $L$ of energy $E(u) < a+b + 8\delta < \hbar(J^{\cD},L)$ of the Floer equation $\del_s u + J^{\cD}(u)(\del_t u - X_{H^{\cD}})(u) = 0,$ with the property that $u(0,0) = R,$ where $J^{\cD}$ can be chosen arbitrarily close to $\{(\phi^t_{\overline{H}^{\cD}})_* J_{\eps}\}.$  By the standard naturality transformation $v(s,t) = (\phi^t_{\overline{H}^{\cD}})^{-1}(u(s,t)),$ we get a $(\phi^t_{\overline{H}^{\cD}})^{-1}_* J^{\cD}$-holomorphic curve $v:\R \times [0,1] \to M$ with $u(\R \times \{0\}) \subset L,$ $u(\R \times \{1\}) \subset L'',$ with energy $E(v) = E(u) < a+b + 8 \delta < \gamma(\overline{H}) + 10 \delta,$ and constraint $v(0,0) = R,$ where $L'' = (\phi^1_{\overline{H}^{\cD}})^{-1}(L)$ can be chosen sufficiently $C^1$-close to $L'$ so that $L'' \cap e(B_{r'}) = \emptyset.$  Note that $J_{\epsilon,\cD} = \{(\phi^t_{\overline{H}^{\cD}})^{-1}_* J^{\cD}\}$ can be made arbitrarily close to $J_{\epsilon}.$ 

Choosing a sequence $\cD_k$ of perturbation data with $J_{\epsilon,\cD_k}$ converging to $J_{\epsilon},$ we obtain a sequence $v_k$ of $J_{\epsilon,\cD_k}$-holomorphic curves of energy $E(v_k) < \gamma(\overline{H}) + 10 \delta,$ with constraint $v_k(0,0) = R,$ and boundary conditions as before. Applying Gromov compactness to the sequence $\{v_k\}$ we obtain a $J_{\epsilon}$-holomorphic curve $v_{\infty}$ with boundary conditions as above, and energy \[E(v_{\infty}) \leq \gamma(\overline{H}) + 10 \delta.\]

Now, by a standard Lelong inequality argument we obtain that \[E(v_{\infty}) \geq \frac{\pi (r')^2}{2} > w(L;L') - \epsilon.\] Combining this with the bound on the energy of $v_{\infty},$ we obtain \[\gamma(\overline{H}) + 10 \delta \geq \gamma({\overline{H}}^{\cD}) + 8 \delta > w(L;L') - \epsilon.\] In the limit $\delta \to 0,$ we obtain the required inequality \[\gamma(H) = \gamma(\overline{H}) \geq w(L;L') -\epsilon.\]

The second part of Theorem \ref{thm:spectral_norm non-deg} on the non-degeneracy of $d_{L,\gamma,int}$ is proven essentially identically, with the only difference being replacing the maps $[\mu_{1:1}(-: x)],$ $ [\mu_{1:1}(-: y)]$ by the maps $[\mu_{2}(x,-)],$ $ [\mu_{2}(y,-)],$ as in Section \ref{sec:filteredYoneda}, and we leave its proof to the interested reader.

\begin{rmk}
We note, that fixing an $\om$-compatible almost complex structure $J$ on $M,$ and fixing $\hbar = \hbar(J,L)$ we could, following Chekanov \cite{ChekanovFinsler}, prove the non-degeneracy of $d_{L,\gamma,ext}$ that is implied by the first part of that theorem by considering the split product $L' = L \times L_{\hbar} \subset (P,\Om) = (M \times T^2, \om \oplus \sigma)$ of $L \subset (M,\omega)$ with a displaceable circle $L_{\hbar} \subset T^2$ in $T^2$ bounding a disk of area $\hbar,$ where $T^2$ is endowed with a symplectic form $\sigma$ of large area $A > 2\hbar.$ In this case $\gamma(H \oplus 0) \leq \gamma(H),$ for all $H \in \cH,$ by the K\"{u}nneth theorem for spectral invariants (see \cite[Theorem 5.1]{EntovPolterovich-rigid}), whence $\gamma(\phi \times \id_{T^2}) \leq \gamma(\phi)$ and consequently $d_{L',\gamma,ext}(L', (\phi \times \id_{T^2}) L') \leq d_{L,\gamma,ext}(L,\phi(L))$ for all $\phi \in \Ham(M,\om).$ Now since by definition $d_{L',\gamma,ext}(\phi(L'_1),\phi(L'_2)) = d_{L',\gamma,ext}(L'_1,L'_2)$ for all $\phi \in \Ham(P,\Om)$ and $L'_1,L'_2 \in \cO_{L'}$ the argumentation in \cite[Theorem 2, Lemma 9]{ChekanovFinsler} shows that $d_{L',\gamma,ext}$ is either non-degenerate, in which case so is $d_{L,\gamma,ext},$ or identically zero. However, the latter option does not hold, as by Theorem \ref{thm:Chekanov}, if $h \in \Ham(T^2,\sigma)$ displaces $L_{\hbar},$ then $d_{L',\gamma,ext}(L',(\id_M \times h)L') \geq \hbar = \hbar(J\times j_{T^2}, L'),$ for $j_{T^2}$ the standard complex structure on $T^2.$
\end{rmk}


\section{Sharpness in Theorem \ref{thm:slimLagrExamples}}

\subsection{Example using radially symmetric Hamiltonian functions}

The goal of this section is to give examples of Lagrangian submanifolds $L \subset M$ with large boundary depths $\overline{\beta}(M,L).$ 
The main reference for this section is \cite{stevenson}, where examples for radially symmetric Hamiltonian diffeomorphisms with large boundary depths are constructed. Here we work out variants of these examples in the relative case. By way of terminology, it will be more convenient to think of Hamiltonian chords of $H \in \cH$ from $L$ to $L$ as intersections points $L \cap (\phi^1_H)^{-1}(L);$ indeed the two relevant sets are in a bijective correspondence.

Suppose that $L \subset M$ is a weakly monotone closed Lagrangian submanifold. Let $B(2 \pi R) \subset M$ be a symplectically embedded ball of radius $\sqrt{2 R}$, so that $B \cap L$ coincides with the real part $\R^n \cap B \subset B \subset \R^{2n}$ of $B$.
Let $f : [0,R] \to \R$ be a smooth function so that its derivative of all orders vanish at $r = R$, and $f'(0)$ is not an integer multiple of $\pi$. In addition, suppose that the points $r_i$ in which $f'$ is an integer multiple of $\pi$ are isolated, and moreover $f''(r_i) \neq 0$.
Let $F: M \to \R$ be defined as follows. For $x \in \partial B(2 \pi r) \subset B$, $F(x) = f(r)$.
For $x \not\in B$, $F(x) = G(x) + f(R)$, where $G$ is a smooth extension of a $C^2$-small Morse function defined on $L$ to a small Weinstein neighborhood of $L$.
The flow of $F$ inside $B$ is $ x \mapsto e^{\sqrt{-1}f'(\frac{|x|^2}{2}) t} x$.
Hence for $x \in L \cap B$, $\phi^1_F(x) \in L$ if $f'(\frac{|x|^2}{2}) = l \pi$, where $l \in \Z$.
In every sphere $\frac{|x|^2}{2} = r_i$ there are an $S^{n-1}$ worth of intersection points $\phi^1_F(L) \cap L$. We therefore perturb $F$ there by adding an extension of a perfect Morse function on $S^{n-1}$ to a small tubular neighborhood of this sphere, for all $r_i$.
The perturbation splits each $S^{n-1}$ into two intersection points, whose Maslov indices, if one takes cappings inside $B$, are
$-l n + n, -ln + 1$ if $f''(r_i) < 0$ and $-l n + n - 1, -ln$ if $f''(r_i) > 0$.
One way to calculate the indices is by calculating the Robbin--Salamon index of the Lagrangian path $(\phi^t_F)_* T_x L$ (see a similar computation for the absolute case in \cite[Section 3.3]{oancea}) and then add $\frac{n}{2}$ which is the Robbin--Salamon index of the canonical short path from $  (\phi^t_F)_* T_x L$ to $T_{\phi^1_F(x)}L$ (see \cite{auroux} for the definition of the canonical short path).
The actions of the intersection points are approximately $f(r_i) - l \pi r_i $.
There is an additional intersection point at the origin, with action $f(0)$ and index $-ln$ with the trivial capping, if the slope $s$ of $f$ coming out of $r=0$ satisfies $l \pi < s < (l+1) \pi $. 
For every critical point of $G$ outside of $B$ with Morse index $j$, there is an intersection point of index $j$ with the trivial capping and its action is approximately $f(R)$.

Let us consider a piecewise linear function $f:[0,R] \to \R$, where $f'$ is not a multiple of $\pi,$ when defined. Let $\{r_i\}$ be the points where $f$ is not differentiable.
We follow \cite{stevenson} and consider a sequence of \begin{it} standard perturbations \end{it} $f_j$ of $f$, where we take small neighborhoods around the $r_i$'s, and pick smoothings in these neighborhoods which have strictly monotonic first derivatives.
Let $F_j$ be the Hamiltonian function induced from $f_j$ for any $j \in \N$. 
One can check (or see the proof in \cite[Claim 4.1]{stevenson}) that the barcodes in any degree $d$ of $(L,F_j)$ converge in the bottleneck distance to a unique barcode which we denote by $\cB_d(f)$.
We call numbers that are either left endpoints of bars in ${\cB}_{d}(f)$ or right endpoints in ${\cB}_{d-1}(f)$ the {\em degree $d$ actions of $f$}.
Note that by construction, continuity of barcodes (see e.g. \cite{UsherZhang}) also applies for the case of piecewise linear radially symmetric functions. More precisely,
\[ d_{\mrm{bottle}} (\cB_d(f_1), \cB_d(f_2)) \leq C \max_{r \in [0,R]} |f_1(r) - f_2(r) | ,\] where $C$ is a constant greater than $1$. Recall from Section \ref{sec:intro-continuation} that $w(L)$ is the relative Gromov width of $L.$

\begin{lma}\label{lem:lower bound on beta}
Let $M^{2n}$ be a weakly monotone symplectic manifold that is either closed, or compact with convex boundary, or open tame at infinity, and $L \subset M$ a weakly monotone closed Lagrangian submanifold. Let $A_L$ be the smallest positive symplectic area of a disc with boundary on $L$, and $N_L$ be the minimal Maslov number. Set $B_L = \frac{n}{N_L} A_L$ if $L$ is monotone, and $B_L = +\infty$ if $L$ is weakly exact. If $L$ is monotone, and $n>1,$ suppose that $N_L > n,$ and that the greatest common divisor of $N_L$ and $n$ satisfies $\gcd(N_L,n) > 1.$ Then for each $\varepsilon > 0$, there exists a Hamiltonian diffeomorphism $\phi$ with boundary depth \[\beta(L,\phi) > \min\{w(L), B_L\} - \varepsilon.\] 

\end{lma}

\begin{proof}
	
We first assume that $L$ is monotone, that $n>1,$ and that $M$ is closed. Then we explain how to modify this lemma to the other situations. 

From the definition of $w(L)$, there exists an embedded symplectic ball $B(2 \pi R) \subset M$ of radius $\sqrt{2R}$ with $\pi R$ arbitrarily close to $w(L)$, so that $B \cap L$ coincides with the real part $\R^n \cap B \subset B \subset \R^{2n}$ of $B$. In case $w(L) > A_L \frac{n}{N_L}$ we consider a smaller ball satisfying that $B \cap L$ coincides with $\R^n \cap B$ and so that $\pi R = A_L \frac{n}{N_L}$.

Let $f : [0,R] \to \R$ be a piecewise linear function so that $f'$ is not a multiple of $\pi$. Let $\{r_i\}$ be the points where $f$ is not differentiable. We call $r_i$ a concave up kink if $f'$ increases right after $r_i$, and a concave down kink if $f'$ decreases right after $r_i$.
After small standard perturbations in neighborhoods of the $r_i$'s, there are two intersection points in $\phi^1_F(L) \cap L$ for every $l$ so that $\pi l$ is between the slope of $f$ before $r_i$ and the slope of $f$ after $r_i$. The actions of these intersection points converge to actions in the barcode of $f$.

Let $R-\varepsilon < r_1 < r_2 < r_3 < R$.
Let $f_0$ be the piecewise linear function connecting the points $(0,-m_0 r_2),(r_2,0),(r_3,\pi R), (R, \pi R + m_1 (R - r_3)) $,
where $-1 << m_0 < 0 < m_1 << 1$.

Let $r(t) = (1-t) r_2 + t r_1$.

Let $f_t$ (see Figure \ref{HamiltonianFig1}) be the piecewise linear function connecting
\[(0,-m_0 r(t) + \pi R t) , (r(t),\pi R t), (r_2,0), (r_3,\pi R), (R, \pi R + m_1 (R - r_3)).\]
This is a homotopy between $f_0$ and $f_1$ which is a piecewise linear function connecting the points
\[(0, -m_0 r_1 + \pi R), (r_1,\pi R), (r_2,0), (r_3,\pi R), (R, \pi R + m_1 (R - r_3)).\]

\begin{figure}
	
	\includegraphics[width=\textwidth]{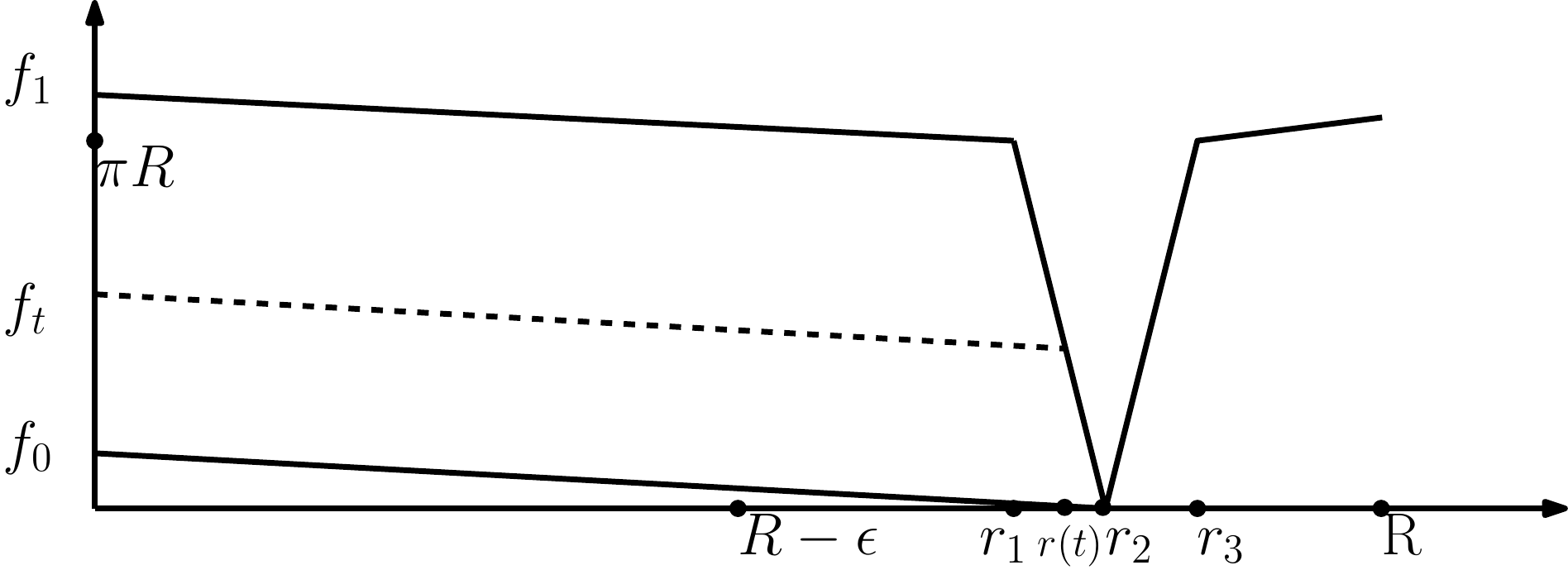}
	\caption{The homotopy $\{f_t\}$}
	\label{HamiltonianFig1}
	
\end{figure}

Let $r = r_2$ be the concave up kink of $f = f_t$. The action $A = f(r) - \pi r l + k A_L$, appears with the degrees \\
$d_1 = -l n + N_L k$,\\
$d_2 = -l n + n - 1 + N_L k$,\\
for any $l \in \Z$ with $f'(r-\epsilon) < \pi l < f'(r+\epsilon)$ (for $\epsilon$ small enough), where $N_L k$ is the Maslov index of the relative disc in the recapping.
Note that a concave up kink cannot have a degree $n+1$ action appearing as a right endpoint of a bar in the barcode.
Indeed, the first possibility is 
\[ -ln + N_L k = n + 1,\]
i.e.,
\[ n (-1 - l) + N_L k = 1,\]
which is impossible since $gcd(N_L,n) > 1$.
The second possibility is
\[ -ln + n -1 + N_L k = n + 1,\]
i.e.,
\[ -ln + N_L k = 2.\]
Denote this degree $n+1$ action by $c^{n+1}$.
Our next goal is to show that for each $t \in [0,1]$, $c^{n+1}$ appears only as a left endpoint of a bar in the barcode. 
For this we follow the proof of Lemma 4.9 in \cite{stevenson}, and we consider the homotopy $h_\tau$, where $h_0$ is a piecewise linear function with only one kink at $r=r_3$, where $h_0(r_3) = f_t(r_3), h_0(r_2) = f_t(r_2),$ and $h_0(R) = f_t(R)$, and the homotopy can be described by taking $h_0$'s graph and folding it along the kinks at $r=r_2$ and at $r=r(t)$ until $f_t$'s graph is created (see Figure \ref{HamiltonianFig2} and Figure \ref{HamiltonianFig3}).
More precisely, let $h_{\frac{1}{2}}$ be a piecewise linear function with two kinks at $r=r_3$, and at $r=r_2$, where $h_{\frac{1}{2}}(R) = f_t(R), h_{\frac{1}{2}}(r_3) = f_t(r_3), h_{\frac{1}{2}}(r_2) = f_t(r_2), h_{\frac{1}{2}}(r(t)) = f_t(r(t))$,
and let $h_1 = f_t$, and define the homotopy
\[ h_\tau = \left\{ \begin{matrix} (1-2\tau) h_0 + 2\tau h_{\frac{1}{2}} &, & 0 \leq \tau \leq \frac{1}{2} \\ (2-2\tau) h_{\frac{1}{2}} + (2\tau - 1) h_1 &, & \frac{1}{2} \leq \tau \leq 1 \end{matrix} \right. \]
\begin{figure}
	
	\includegraphics[width=\textwidth]{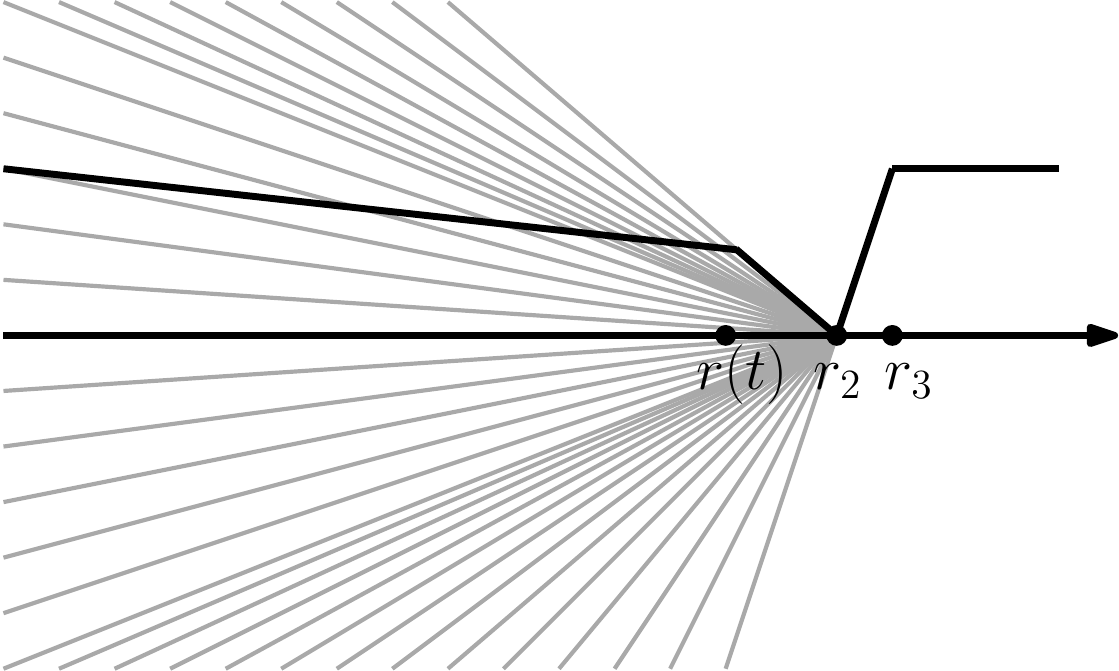}
	\caption{The first part of the homotopy $\{h_\tau\}$, where at $\tau=0$ the only kink is at $r=r_3$}
	\label{HamiltonianFig2}
	
\end{figure}

\begin{figure}
	
	\includegraphics[width=\textwidth]{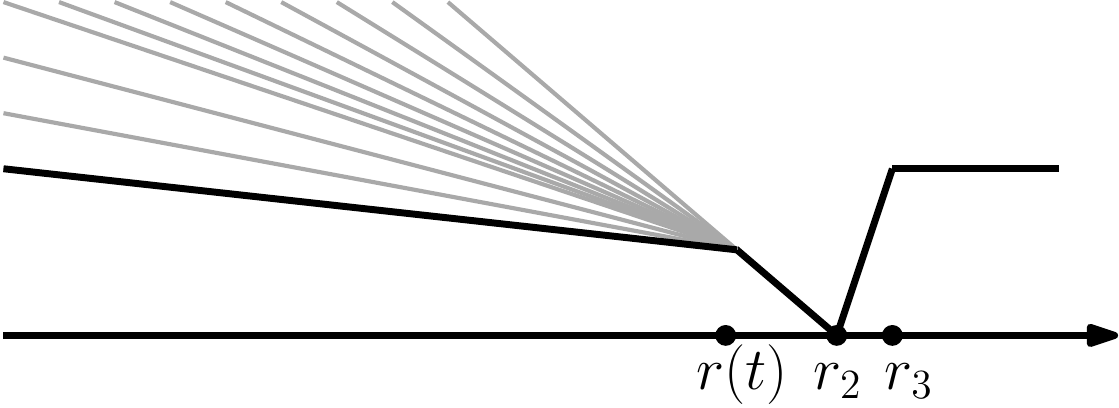}
	\caption{The second part of the homotopy $\{h_\tau\}$ whose end at $\tau=1$ is $f_t$'s graph}
	\label{HamiltonianFig3}
	
\end{figure}

Let $s(\tau)$ denote the slope of the line coming from $r=0$ in $h_\tau$. The time $\tau_0 \in [0,\frac{1}{2}]$ for which $s(\tau_0) = l \pi$, is the first time where our $n+1$ action $c^{n+1}$ exists in the spectrum of $h_\tau$. 
Hence from continuity of barcodes, one gets that right after $\tau=\tau_0$ there is a degree $n$ action or a degree $n+2$ action that converges to $c^{n+1}$ as $\tau \to \tau_0$.
Note that the only actions changing with time are the actions coming from $r=0$. Since $h_\tau(0) = -s(\tau) r_2$, the $n+2$ action coming from $r=0$ right after $\tau_0$, i.e. when $s(\tau) < l \pi$, is equal to $c^{n+1}$.
Hence for $\tau$ right after $\tau_0$, $c^{n+1}$ is a left endpoint of a bar in the barcode.
Since along the entire homotopy $h_\tau$ the only actions changing with time are the ones coming from $r=0$, there is no degree $n+1$ action changing with time. The only way that $c^{n+1}$ stops being a left endpoint, is that a degree $n+2$ action changing with time $c^{n+2}(\tau)$ is part of a bar $[c^{n+1},c^{n+2}(\tau))$ and there in some time $\tau = \tau_1$, one has $c^{n+2}(\tau_1) = c^{n+1}$. If $c^{n+2}(\tau)$ keeps decreasing right after $\tau_1$, then this is impossible because there must be a left or a right endpoint to the bar that includes $c^{n+2}(\tau)$ and hence a new bar is formed, but again from continuity of barcodes, there must be a degree $n+1$ or a degree $n+3$ action that equals $c^{n+1}$, as $\tau \to \tau_1$. 
There are no candidates for these actions and hence we get a contradiction. Thus $c^{n+1}$ remains a left endpoint in the entire homotopy, and hence it is also a left endpoint in the barcode of $f_t$. 

Our next step is to consider the degree $n+1$ actions coming from the concave down kinks.
If $r$ is a concave down kink one has the action $A = f(r) - \pi r l + k A_L$ with the degrees\\
$d_1 = -l n + 1+ N_L k$,\\
$d_2 = -l n + n + N_L k$,\\
for any $l \in \Z$ that satisfies $f'(r-\epsilon) > \pi l > f'(r+\epsilon)$.
The only possibility for a degree $n+1$ action comes from the first option. Indeed, if
\[ -ln + n + N_L = n + 1,\]
then
\[ -ln + N_L k = 1,\]
which is impossible since $gcd(N_L,n) > 1$.

Since the slope of $f_t$ coming out of the origin for each $t$ is negative and close to zero, the action $A = f_t(0) + k A_L$ is the only one coming from the origin with the degree $d = n + N_L k$.

Using the above observations, the degree $n+1$ actions that appear as right endpoints can only come from the concave down kinks at $r=r_3$, and at $r=r(t)$, and from the external points.
We wish to show that all degree $n+1$ actions not coming from $r=r(t)$ satisfy that their actions are greater than $2\pi R$.

In $r = r_3$, the slopes before and after the kink are positive, so one has $l > 0$ for each $l$ satisfying $f'(r_3-\epsilon) > \pi l > f'(r_3 + \epsilon)$. One has a degree $n+1$ action if 
\[ -l n + 1 + N_L k = n + 1,\]
i.e.,
\[ k = \frac{l n+n}{N_L} .\]
After plugging this in the formula for the action one gets
\begin{align*}
 A &=  \pi R - \pi r l +  \frac{l n + n}{ N_L} A_L \\
 &\geq \pi R - \pi r l + \frac{l n + n}{ N_L} \pi R \frac{N_L}{n} \\
 &= \pi R - \pi r l + \pi R (l + 1) \\
 &=  2 \pi R + l \pi (R - r) \\
 &> 2 \pi R .
\end{align*}

For an external point, i.e. a critical point of $G$ with Morse index $j$,
since $j \leq n$, one has $k \geq 1$, and $A > \pi R + k A_L \geq \pi R + k \pi R \frac{N_L}{n} > 2 \pi R$.

Consider the intersection point of degree $n+1$ one gets from $f_t$ in the point $r=r(t)$ with $l=-1$ and the trivial capping.
Its action is $\pi (Rt + r(t))$.

When $t$ tends to $0$, $r(t) \to r_2$, and the action tends to $\pi r_2$.
Note that there are no other degree $n+1$ actions with action close to $\pi r_2$. Since this action does not appear in the barcode of $f_0$, from continuity of barcodes one gets that the action must cancel out with a degree $n+2$ or a degree $n$ action that tends to the same action as $t \to 0$ (to make a bar of length $0$). One can eliminate the possibility that this action does not come from $r=r_2$ by doing a small perturbation.
After maybe another perturbation, one can make sure that $\pi r_2$ and $A_L$ are rationally independent, and hence the only possibility for the action $\pi r_2$ is the action coming from $r=r_2$ of degree $n$, with $l=-1$, and $k=0$. Hence, when $t$ is sufficiently small, there is a bar $(\pi r_2,\pi(Rt + r(t)))$.

Since all other degree $n+1$ actions appearing as right endpoints are greater than $2 \pi R$, the right endpoint of the bar stays $\pi (Rt + r(t))$ throughout the homotopy.

The left endpoint of the bar can change only if a degree $n$ action that changes with time, crosses the value $\pi r_2$. 
The only actions changing with time are those coming from $r=r(t)$ or from $r=0$.
The degree $n$ actions coming from $r=r(t)$ satisfy $-ln + n + N_L k = n$, i.e. $k=l \frac{n}{N_L}$, and since from looking at the slopes one has $l < 0$, their actions satisfy
\begin{align*}
A &= \pi R t - l \pi r(t) + l \frac{n}{N_L} A_L \\
& \leq \pi R t - l \pi r(t) + l \frac{n}{N_L} \pi R \frac{N_L}{n} \\
&= \pi R t + l (\pi R - \pi r(t))\\
&\leq \pi R t - (\pi R - \pi r(t)) \\
&= \pi R (t - 1) + \pi r(t) \leq \pi r(t) \\ 
&< \pi r_2 .
\end{align*} 
Hence the only possibility for a degree $n$ action changing with time that crosses the action $\pi r_2$ is the degree $n$ action coming from $r=0$. Its action is $\pi R t + \varepsilon$.

We get that in $f_1$ one must has a bar of length greater than or equal to
\[ (\pi R + \pi r_1) - (\pi R + \epsilon) = \pi r_1 - \varepsilon \geq \pi R - 2\varepsilon .\]
Hence $\beta \geq \pi R - 2 \varepsilon$.
Since one can choose $\pi R > w(L) - \varepsilon$ if $w(L) < A_L \frac{n}{N_L}$, or $\pi R = A_L \frac{n}{N_L}$ otherwise, one gets an example where $\beta \geq \min(w(L),A_L \frac{n}{N_L}) - 3 \varepsilon$ for any $\varepsilon > 0$.

In the case when $L$ is weakly exact, one proceeds precisely as above, except that the arguments are made considerably easier by the absence of recapping. When the symplectic manifold $M$ has convex boundary, or is tame at infinity, Lagrangian Floer theory of the closed Lagrangian submanifold $L$ is well-defined, and the same arguments for it go through.

Finally, in the case $n=1,$ the Lagrangian submanifold $L$ is a simple closed curve in a surface. In case the surface is symplectically aspherical, the above proof applies verbatim. Otherwise $L$ must be separating, one component of the complement being a disk of area $A_L.$ In this case its minimal Maslov number is $N_L = 2.$ If the other component is not a disk, it is easy to see that the quantum homology of $L$ vanishes, hence the spectral norm is not well-defined, however the boundary depth still is. For the constant Hamiltonian $H \equiv 0,$ it is easy to calculate in this case that $\beta(L,H) = A_L.$ 

This brings us to the case when $L = \R P^1,$ $M = \C P^1 = S^2,$ and $L$ divides $S^2$ into two disks of equal areas $A=A_L.$ We normalize the symplectic form on $S^2$ to have $A=1/2.$ Section \ref{sect:combinatorial} explains a combinatorial approach to the construction of a Hamiltonian deformation $L' = \phi^1_H(L)$ of $L$ with boundary depth $\beta(L,H)> 1/2 - \epsilon.$ However, in this section we explain a different description of this example.

Consider the radial function $f(r)$ as before, but with the profile depicted in Figure \ref{fig:new-radial}.  

\begin{figure}
	
	\includegraphics[width=\textwidth]{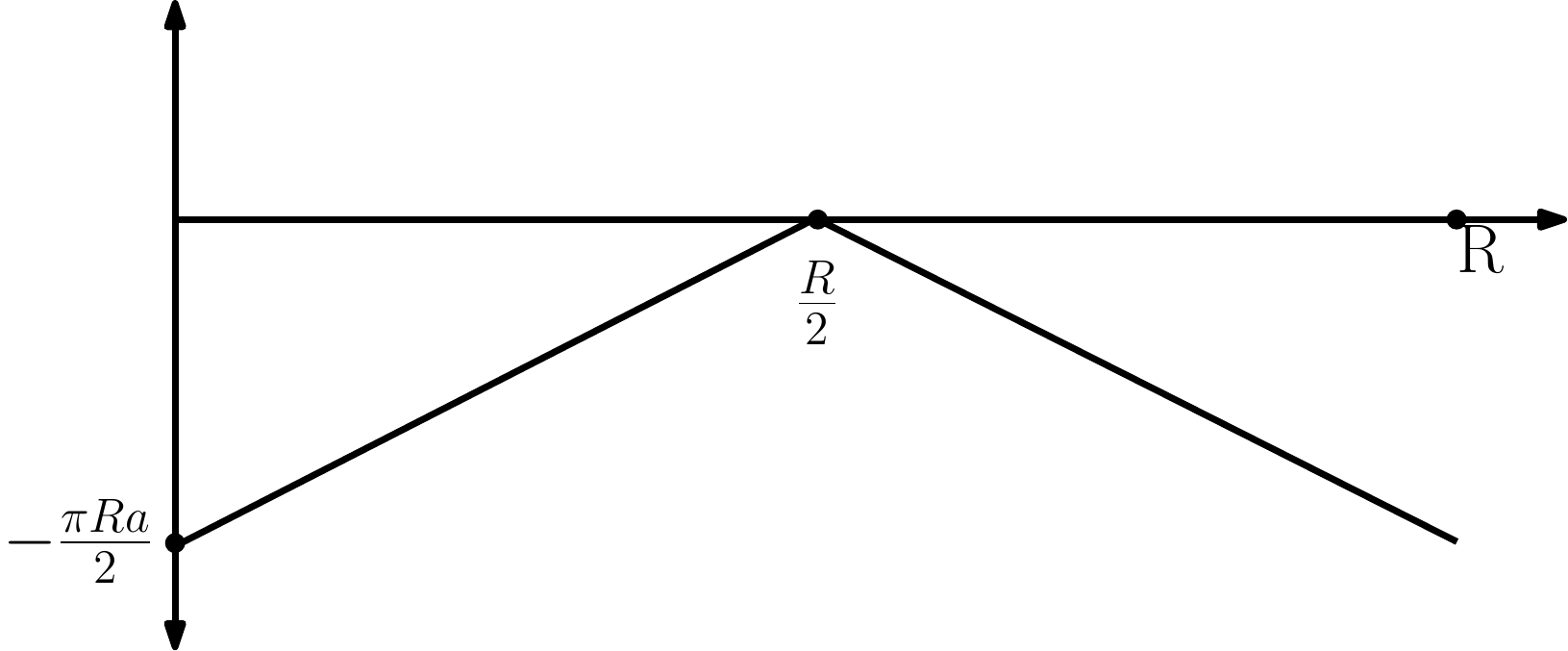}
	\caption{The new profile of the radial function $f(r)$}
	\label{fig:new-radial}
	
\end{figure}

More precisely choose $R>0$ with $\pi R = A.$ Then the graph of $f(r)$ joins the points \[(0,-\pi R a/2),(R/2,0),(R,-\pi Ra/2)\] for $0<a<1.$ We observe that this function, considered on the open ball of capacity $A,$ embedded as described in Section \ref{subsec:equator-sphere} relative to $L,$ extends to a smooth function on $S^2.$ However, a more general approach would be to slowly and monotonically increase the slope to $0$ near $r = R.$

A suitable non-degenerate small perturbation $H$ of $f(r)$ gives one generator of $CF(L,H)$ at $r=0$ with index $0$ and action $-\pi R a/2$ and consequently recapped generators with index $kN_L$ and action $-\pi R a/2 + kA,$ two generators at $r=R/2$ with index $d = 1$ and action $0$ hence recapped generators with index $1 + kN_L$ and action $k A,$ and finally one generator at $r=R,$ of index $0$ and action $-\pi R a/2,$ hence recapped generators of index $kN_L$ and action $-\pi R a/2 + kA.$ Since $N_L = 2$ we obtain that there are precisely two generators in each index $r \in \Z.$ This implies for example that either the barcode $\cB_1(L,H)$ in degree $1$ or $\cB_0(L,H)$ in degree $0$ has finite bars. In either case, one can check that the barcode that has finite bars consists of a unique finite bar, and a unique infinite bar, however this is not important for this argument. The index $1$ generators have action $0,$ the index $2$ generators have action $-\pi Ra/2 + A,$ and the index $0$ generators have action $(-\pi Ra/2).$  If $\cB_0(L,H)$ has finite bars, then its finite bars have length at least $\pi R a/2,$ and if $\cB_1(L,H)$ has finite bars, then its finite bars have length at least $-\pi Ra/2 + A.$ Hence $\beta(L,H) \geq \min \{\pi Ra/2, -\pi Ra/2+A\} = \pi Ra/2,$ by the condition that $\pi R = A.$ Fixing $\epsilon > 0,$ and taking $a$ sufficiently close to $1,$ we obtain an example with $\beta(L,H) > \frac{A}{2} - \epsilon.$

\end{proof}

\subsection{Example using combinatorial Floer homology}\label{sect:combinatorial}

In this section we give another example for a Lagrangian submanifold of $S^2$ with boundary depth arbitrarily close to $\frac{1}{4}$. For this we follow a combinatorial construction of Floer homology on surfaces described, and proven to be equivalent to Floer homology in the framework of Lagrangian intersections, in \cite{CombinatorialFloer}. For a detailed review of Floer theory in the framework of Lagrangian intersections we refer to \cite{OhBook}, and the references therein. In particular, see \cite[Section 14.4]{OhBook} where the equivalence of the Floer theory for $L$ with Hamiltonian perturbation $H \in \cH$ as described in Section \ref{Sect:prelim} and the Floer theory in the framework of Lagrangian intersections of the pair $L,$ $\phi^1_H(L),$ with suitable additional choices, is shown. The boundary depth does not depend on these choices, and hence can equivalently be computed in either framework.
While in \cite{CombinatorialFloer} the equivalence of the combinatorial construction and the Floer construction was not fully proven for the case of equators on $S^2,$ all indications are that this can be carried out in a quite straightforward way\footnote{We thank Michael Khanevsky for explaining this to us.}.
Since in this section we wish to discuss empirically an additional example, obtained directly from computer experiments carried out by AK, we leave the analytic details to the interested reader, and use the combinatorial Floer homology as it is expected to work.

Let $L_{std} \subset S^2$ be the standard equator. Let $L \subset S^2$ be the equator described in Figure \ref{equatorFig}. The order of the intersections of $L$ and $L_{std}$ are described via the permutation $(1,4,3,2)$. From the fact that each equator divide $S^2$ into two components of area $\frac{1}{2}$ we have the following.
\[| A_1 | + |A_2| + |A_3| = \frac{1}{2},\]
\[ |A_4| + |A_5| + |A_6| = \frac{1}{2},\]
\[ |A_2| + |A_5| + |A_6| = \frac{1}{2},\]
\[ |A_1| + |A_3| + |A_4| = \frac{1}{2}.\]
One can check that for every choice of $|A_1|,\ldots,|A_6| \in \R^+$ that satisfy these constraints, one can find an appropriate $L$ with the required areas.

We wish to calculate the boundary depth of $CF(L_{std},L)$. We will work with the combinatorial definition of Floer homology described in \cite{CombinatorialFloer}. Note that for the calculation of the Floer complex we only need to know the permutation of the intersections, and the areas of the components of $S^2 \setminus (L \cup L_{std})$.

\begin{figure}
	
	\includegraphics[width=0.5\textwidth]{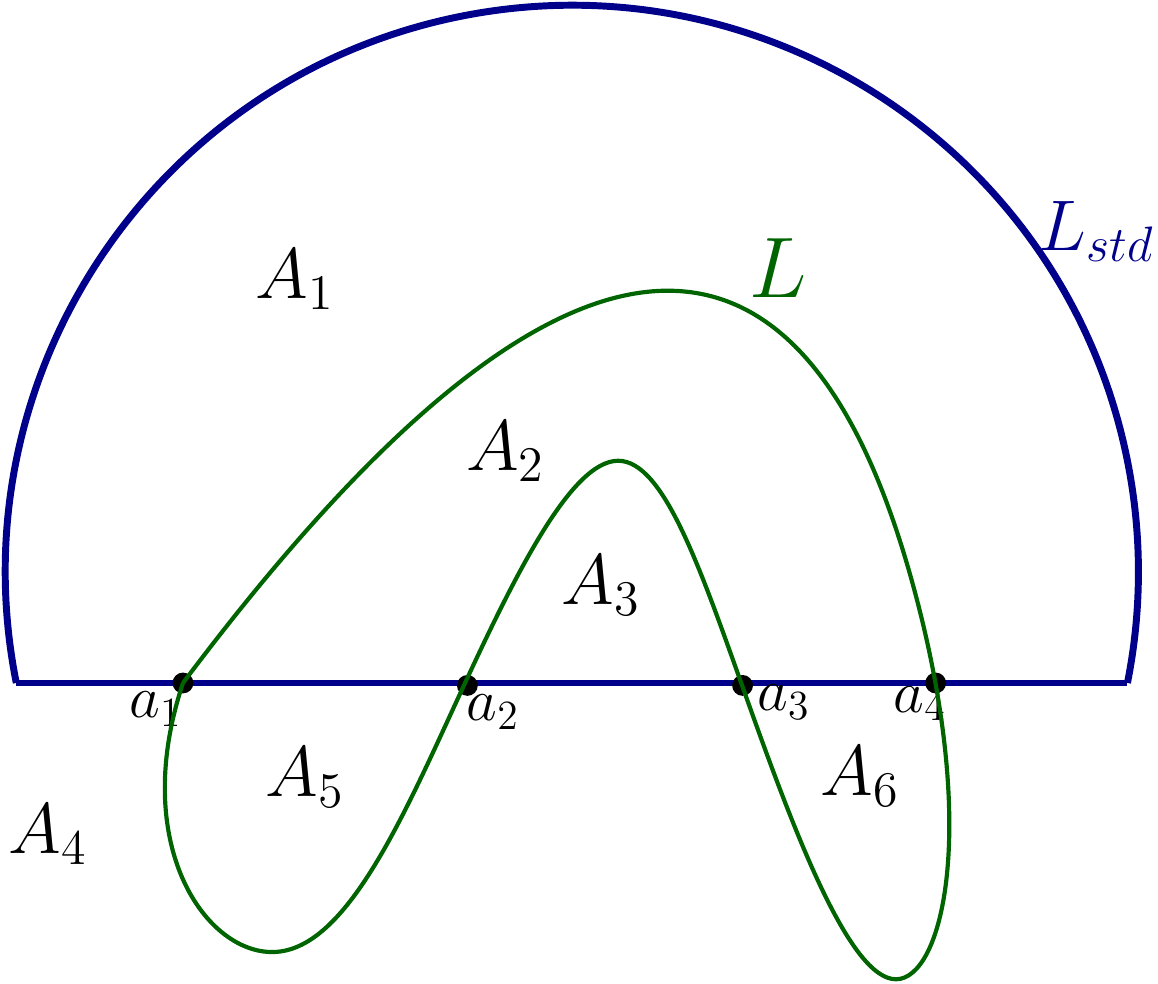}
	\caption{The blue line is the standard equator, and the green line represents the equator $L$. The intersections are the points $a_1,\ldots,a_4$, and $A_1,\ldots,A_6$ are the areas of connected components of $S^2 \setminus (L \cup L_{std})$. Note that the sum of the areas inside (or outside) each equator must be $\frac{1}{2}$.  }
	\label{equatorFig}
	
\end{figure}

The combinatorial Floer complex is generated by $a_1,\ldots,a_4$ over the Novikov ring with coefficients in $\mathbb{F}_2 = \Z / 2\Z$, where the quantum variable $q$ has index $2$ and action $\frac{1}{2}$.
Note that the definition in \cite{CombinatorialFloer} describes how to calculate the action difference, and the boundary operator $\del$, but there is freedom in the choice of cappings and indices. One can check that if there is a lune from $a_i$ to $a_j$ then the parity of $i$ and the parity of $j$ are different. This gives us the ability to choose the indices of $a_1,a_3$ to be $0$, and those of $a_2,a_4$ to be $1$.
Then for odd $i$, $\partial a_i = \sum q a_j$, where the sum runs over all even $j$'s that satsify that there is exactly one lune from $a_i$ to $a_j$ (the other possibilities are zero lunes or two lunes). Similarly for even $i$,  $\partial a_i = \sum a_j$, where the sum runs over all odd $j$'s that satisfy that there is exactly one lune from $a_i$ to $a_j$.

One could go over all the possibilities for lunes in our example and get the following.
\[ \partial a_1 = 0, \]
\[ \partial a_2 = a_1 + a_3, \]
\[ \partial a_3 = 0, \]
\[ \partial a_4 = a_1 + a_3. \]

The boundary depth is the lengh of the bar coming from the boundary $a_1+a_3$,
which is the action difference between the minimum between the actions of $a_2$ and $a_4$, and the action of $a_1+a_3$ which is the maximum between the actions of $a_1,a_3$.
We get that the boundary depth is
\[\min(a_2-a_1,a_2-a_3,a_4-a_1,a_4-a_3) .\]
Since the action difference equals the area of the lune between the intersection points, we get that the boundary depth is
\[ \min(|A_5|,|A_3|,|A_1|,|A_6|).\]
One can check that under our constraints for $A_1,\ldots,A_6$, an upper bound for the boundary depth is $\frac{1}{4}$, and if we put
\[ |A_1| = |A_3| = |A_5| = |A_6| = \frac{1}{4} - \epsilon,\]
\[ |A_2| = |A_4| = 2 \epsilon, \]
then all the constraints are satisfied, and the boundary depth is $\frac{1}{4} - \epsilon$. 

Finally, we remark that roughly the same example gives a boundary depth arbitrarily close to $\frac{1}{2}$ on the annulus $D^* S^1$, where we normalize the total area to be 1.
Indeed, one can take the sphere with the Lagrangian submanifolds $L_{std}$ and $L$, and then place two holes inside the areas $A_1$ and $A_5$ respectively.
The combinatorics of the lunes stays the same, except that one should remove lunes that pass through $A_1$ or $A_5$.
One gets that the boundaries change to
\[ \partial a_2 = \partial a_4 = a_3,\]
and still
\[\partial a_1 = \partial a_3 = 0.\]
Hence the boundary depth is
\[\min(a_2 - a_3, a_4 - a_3) = \min(|A_3|,|A_6|).\]
By taking
\[|A_3| = |A_6| = 1/2 - 2 \epsilon,\]
\[|A_1| = |A_2| = |A_4| = |A_5| = \epsilon,\]
one gets that the boundary depth is $1/2 - 2 \epsilon$, and all the constraints are satisfied.


\subsection{Relative ball embeddings}

In this section we show the existence of symplectic ball embeddings into a symplectic manifold $M$ relative to a Lagrangian submanifold $L,$ satisfying Lemma \ref{lem:lower bound on beta}, and hence providing lower bounds on $\overline{\beta}(M,L).$
In particular this proves the sharpness part of Theorem \ref{thm:slimLagrExamples}.  In this section, for $R> 0,$ we denote by $D(R)$ the standard symplectic ball of capacity $R$ (corresponding to radius $r=\sqrt{R/\pi}$).

\subsubsection{$S^1 \subset S^2$}\label{subsec:equator-sphere}

Recall that we let the total area of $S^2$ be equal to $1.$ Taking a point $p$ on $L=S^1,$ it is easy to see that $S^2\setminus \{p\}$ is symplectomorphic to the standard open disk $D$ of area $1,$ by a symplectomorphism $\varphi$ sending $S^1 \setminus \{p\}$ to the real axis intersected with $D.$ Taking the preimage by $\varphi$ of the closed disk $\overline{D}(1/2)$ yields the required embedding. This example generalizes to produce an open ball $D$ of capacity $1$ symplectically embedded in $\C P^n$ isomorphically onto $D_{st} = \C P^n \setminus \C P^{n-1}$ intersecting $\R P^n$ exactly along the real part. We call $D_{st}(1-\epsilon)$ the restriction of this embedding to $D(1-\epsilon).$ Together with Lemma \ref{lem:lower bound on beta} this shows that the bound in Case (\ref{thmBDbounds-case:RPn}) for $n=1,$ of Theorem \ref{thm:slimLagrExamples} is sharp.

\subsubsection{$S^1 \subset D^*S^1 = [-1/2,1/2] \times S^1$}

It is easy to construct the embedding of a ball of capacity $1-\eps$ relative to $L = S^1 \times \{0\}$ by considering $[-1/2,1/2] \times \{p\},$ for a point $p \in S^1.$ This, together with Lemma \ref{lem:lower bound on beta} this gives the lower bound \[\overline{\beta}(M,L) \geq 1/2\] in this case. This generalizes to the existence of a symplectic embedding of a ball of capacity $1-\eps$ inside $\C P^n \setminus Q^{n-1}$ relative to $\R P^n.$ The construction follows \cite[Section 3]{McDuffPolterovichPacking}. One considers the holomorphic automorphism $\sigma_t([z_0,\ldots,z_n]) = [z_0,\ldots,t^{-1}\cdot z_n]$ for $t > 0,$ notes that it preserves $L = \R P^n,$ and that for $t = T$ sufficiently large, $\sigma_T(Q^{n-1})$ lies in the complement of the ball $\overline{D}_{st}(1-\epsilon)$ from the above example. Considering the preimage $\overline{D}_0(1-\epsilon) = (\sigma_T)^{-1} (\overline{D}_{st}(1-\epsilon)),$ and observing that $\sigma_T^{*} \omega_{FS,n}$ is a K\"{a}hler form cohomologous to $\omega_{FS,n},$ we finish the construction by applying Moser's method relating these two forms, which can be arranged to produce diffeomorphisms preserving $L.$ In conclusion we obtain the embedding of a ball $\overline{D}(1-\epsilon)$ of capacity $1-\epsilon$ relative to $L.$ 

\begin{rmk}
Section \ref{sect:combinatorial} describes a combinatorial example showing the lower bound $\overline{\beta}(M,L) \geq 1/2.$ Moreover, it is interesting to observe that this lower bound is twice larger than the one in Case (\ref{thmBDbounds-case:RPn}) for $n=1,$ of Theorem \ref{thm:slimLagrExamples}.
\end{rmk}

\subsubsection{$\C P^n \cong \Delta \subset \C P^n \times (\C P^n)^{-}$}

Observe that $\C P^n \setminus \C P^{n-1},$ where we take $\C P^{n-1} = \{z_n = 0\},$ is symplectomorphic to the open $2n$-dimensional ball $B$ of capacity $1.$ We shall construct the required ball embedding inside $B \times B^{-}$ so that it intersect $\Delta_B$ exactly along the real part. First of all $B \times B^{-}$ contains a symplectically embedded copy of the open $4n$-dimensional disk $D$ of capacity $1.$ By a suitable unitary transformation it can be reparametrized to intersect $\Delta_B \subset B \times B^{-}$ exactly along the real part. Then taking the preimage of $\overline{D}(\frac{n}{n+1})$ yields the required embedding into $M = \C P^n \times (\C P^n)^{-}.$ We note that this provides the lower bound $\frac{n}{2(n+1)}$ on $\overline{\beta}(M,L)$, which is twice smaller than the upper bound from Theorem \ref{thm:slimLagrExamples}.

However, running the same argument as that for Lemma \ref{lem:lower bound on beta}, but in the absolute case, which can in fact be extracted from \cite{stevenson}, and applying it to the ball $\overline{D}_{st}(\frac{n}{n+1}) \subset \C P^n \setminus \C P^{n-1}$ yields, given $\eps >0,$ a Hamiltonian $H \in \cH_{\C P^n},$ with $\beta(H) > \frac{n}{n+1} - \epsilon.$ Applying Section \ref{sect: abs as rel}, we obtain a Hamiltonian $\hat{H} \in \cH_{\C P^n \times (\C P^n)^{-}}$ with $\beta(L,\hat{H}) > \frac{n}{n+1} - \epsilon.$ We note that instead of the disk $\overline{D}^{4n}(\frac{n}{n+1}),$ the Hamiltonian $\hat{H}$ is supported in the product $\overline{D}_{st}(\frac{n}{n+1}) \times \overline{D}_{st}(\frac{n}{n+1})^{-}.$ This shows that the bound in Case (\ref{thmBDbounds-case:CPn}) of Theorem \ref{thm:slimLagrExamples} is sharp.

\subsubsection{$S^n \subset Q^n$}\label{sect: embedding sphere in quadric}

As the cases $n=1,2$ were covered in previous examples, let $n\geq 3.$ The quadric $Q^n$ admits a degree $2$ branched cover $\pi: Q^n \to \C P^n$ ramified over $Q^{n-1} \subset \C P^n,$ with branch locus $Q^{n-1} \subset Q^n,$ such that $L = S^n$ is the preimage of $\R P^n \subset \C P^n$ \cite{Audin-Examples}. An elementary calculation shows that the restriction $\om_{FS,n+1}|_{Q^n}$ of the Fubini-Study form on $\C P^{n+1}$ to the quadric is cohomologous to $\pi^*(\om_{FS,n}).$ Hence fixing $\eps>0,$ and considering the preimage under $\pi$ of the  ball $\overline{D}(1-\epsilon)$ of capacity $1-\epsilon,$ embedded relative to $\R P^n$ inside $\C P^n \setminus Q^{n-1}$
we obtain, by an application of the Moser method relating $\pi^*(\om_{FS,n})$ and $\om_{FS,n+1}|_{Q^n},$ which can be arranged to produce diffeomorphisms preserving $L,$ a symplectic embedding of two disjoint balls of capacity $1-\epsilon$ each, with real part of each lying exactly on $L.$ We note that as $L$ is simply connected, $A_L = A_{Q^n} = 1$ in this case. We refer to \cite[Section 3]{McDuffPolterovichPacking} for related constructions.

Applying Lemma \ref{lem:lower bound on beta} to one of the balls of capacity $1-\epsilon,$ we obtain a Hamiltonian $H \in \cH_{Q^n}$ with $\beta(L,H) > \frac{1}{2} - \epsilon.$ Since $\epsilon > 0$ was arbitrary, this proves that the bound in Case (\ref{thmBDbounds-case:Sn}) of Theorem \ref{thm:slimLagrExamples} is sharp.

\begin{rmk}
We expect the upper bound in Theorem \ref{thm:slimLagrExamples} for $L=\R P^n,$ $n>1,$ to be sharp, however we were not able to prove this. Note that Lemma \ref{lem:lower bound on beta} does not apply in this case, because $L$ is monotone, with $n>1,$ and $\gcd(N_L,\dim L) = \gcd(n+1,n) = 1.$ The best lower bound we were able to produce in this case, by very slightly modifying the argument for case $n=1$ of Lemma \ref{lem:lower bound on beta}, is $\frac{1}{4} \leq \overline{\beta}(M,L),$ and $\frac{1}{4}$ is strictly smaller than the upper bound $\frac{n}{2n+2}$ for $n>1.$ Somewhat similarly, in the case of $\bH P^n \subset Gr(2,2n+2)$ the explicit embedding of \cite{GcLu} of a ball of capacity $A=\omega([\C P^1])$ adapted in such a way as to intersect $L = \bH P^n$ precisely along its real part, yields the lower bound $\frac{1}{2} A,$ that is strictly smaller, for $n>1,$ than the upper bound $\frac{n}{n+1} A$ that we have produced. 
\end{rmk}

\section{Discussion}

We collect a few remarks on related topics and possible future extensions of our results.

\begin{rmk}
 Considering the boundary depth $\beta$ of a Hamiltonian diffeomorphism of the sphere $S^2,$ for instance, or that of the Lagrangian equator $L_{eq}$ in $S^2$ with a Hamiltonian perturbation, one may wonder whether one can show the existence of a sequence in $\Ham(S^2)$ or in the Hamiltonian orbit of $L_{eq}$ on which $\beta$ tends to infinity. By \cite{UsherBD1,UsherBD2} this would have implied that such a sequence is not contained in a $R$-neighborhood in the Hofer metric of a given quasi-geodesic, or respectively $L_{eq},$ for each $R>0.$ This would have provided an answer to a well-known question of Polterovich and Kapovich from 2006. However, we have shown that this approach, per se, is rather naive, by providing a uniform upper bound on the boundary depth in these two cases. 
\end{rmk}

\begin{rmk}\label{rmk:Chekanov-transversality}
	We discuss the generality in Theorem \ref{thm:Chekanov}. The assumption that $(M,\omega)$ is monotone is not essential for the proof, as soon as the absolute Hamiltonian Floer theory, with associative pair-of-pants products, is well-defined with coefficients in $\mathbb{F}_2$. If in addition $L$ is orientable, one expects this statement to hold in full generality, with coefficients in a field of characteristic $0,$ by techniques of virtual fundamental cycles in Hamiltonian Floer theory (see \cite{FukayaOno1,LiuTian,PardonHam},\cite{FO3:book-vol12, HWZ-GW, HWZ-SC, HWZ-PF17}, and see also \cite{H-PF17,McDuff-K17} for a recent description of the state of the art and further references). In case when the cohomology class of symplectic form is rational, one may instead invoke classical transversality in a suitable way, following \cite{CharestWoodward-Floer17}. Finally, in case when $L$ is connected, orientable, and relatively spin,  following \cite{STV-Flux18} one may improve the bound $\hbar(J,L)$ to the invariant $\Psi(L),$ introduced {\it ibid.}, that depends only on $L \subset M$  (using virtual techniques for general closed $(M,\omega),$ or \cite{CharestWoodward-Floer17,CharestWoodward-Fukaya} when $(M,\omega)$ is rational). Moreover, one can improve the lower bound $\#(L \cap \phi(L)) \geq 1$ to $\#(L \cap \phi(L)) \geq \mathrm{cl}(L,\bK) + 1,$ via the cup-length of $L$ for the suitable coefficient field $\bK,$ in the general, possibly non-transverse case, following \cite{FloerCuplength,HoferCuplength,SchwarzCuplength}. These variations would require digressions on perturbation techniques, and on Lusternik-Schnirelman results in Floer theory, and hence shall appear elsewhere.
\end{rmk}

\begin{rmk}
	We remark that the class of examples in Theorem \ref{thm:slimLagrExamples}, $\RP^n, \C P^n, \HH P^n, S^n,$ consists of compact symmetric spaces of rank one. These are Zoll manifolds, that is they admit metrics all of whose prime geodesics are closed, and have the same period (cf. \cite{BesseClosed}). Moreover, their Lagrangian embeddings into monotone symplectic manifolds that we consider above can all be obtained by the so-called Zoll cut (see \cite{Audin-Zoll}), which is a particular case of the symplectic cut construction \cite{Lerman-cuts}, applied to the cotangent bundle of these manifolds. The remaining symmetric space with the same property is $\mathbb{O} P^2,$ the Cayley projective plane, and it should satisfy the conditions of Proposition \ref{Prop:slimLagr}, with respect to the Zoll cut embedding. However, as this case would require a digression on the octonions and exceptional Lie groups, it shall appear elsewhere.
	
\end{rmk}

\begin{rmk}
Finally, we expect it to be possible to extend Theorems \ref{thm:spectral_norm non-deg} and \ref{theorem: Lipschitz wide ext} to the case when $(M,\omega)$ is (monotone or symplectically aspherical) open tame at infinity, or compact with convex boundary, by working everywhere in the language of Floer cohomology instead of homology. 
\end{rmk}

\bibliography{bibliographyBBD}

\begin{thebibliography}{100}

\bibitem{AbouzaidBook}
M.~Abouzaid.
\newblock Symplectic cohomology and {V}iterbo's theorem.
\newblock In {\em Free loop spaces in geometry and topology}, volume~24 of {\em
  IRMA Lect. Math. Theor. Phys.}, pages 271--485. Eur. Math. Soc., Z\"urich,
  2015.

\bibitem{AK-simplehomotopy}
M.~Abouzaid and T.~Kragh.
\newblock Simple homotopy equivalence of nearby {L}agrangians.
\newblock Preprint, arXiv:1603.05431 [math.SG], 2016.

\bibitem{AbouzaidSeidel}
M.~{Abouzaid} and P.~{Seidel}.
\newblock {An open string analogue of Viterbo functoriality.}
\newblock {\em {Geom. Topol.}}, 14(2):627--718, 2010.

\bibitem{AS-KhovanovFloer}
M.~Abouzaid and I.~Smith.
\newblock Khovanov homology from {F}loer cohomology.
\newblock {\em J. Amer. Math. Soc.}, 2018.
\newblock Available at https://doi.org/10.1090/jams/902.

\bibitem{albers08}
P.~Albers.
\newblock A {L}agrangian {P}iunikhin-{S}alamon-{S}chwarz morphism and two
  comparison homomorphisms in {F}loer homology.
\newblock {\em Int. Math. Res. Not. IMRN}, (4):Art. ID rnm134, 56, 2008.

\bibitem{Team}
D.~{Alvarez-Gavela}, V.~{Kaminker}, A.~{Kislev}, K.~{Kliakhandler},
  A.~{Pavlichenko}, L.~{Rigolli}, D.~{Rosen}, O.~{Shabtai}, B.~{Stevenson}, and
  J.~{Zhang}.
\newblock Embeddings of free groups into asymptotic cones of hamiltonian
  diffeomorphisms.
\newblock {\em J. Topol. Anal.}, Online Ready, 2018.
\newblock Available at https://doi.org/10.1142/S1793525319500213.

\bibitem{Audin-Examples}
M.~Audin.
\newblock On the topology of {L}agrangian submanifolds. {E}xamples and
  counter-examples.
\newblock {\em Port. Math. (N.S.)}, 62(4):375--419, 2005.

\bibitem{Audin-Zoll}
M.~Audin.
\newblock Lagrangian skeletons, periodic geodesic flows and symplectic
  cuttings.
\newblock {\em Manuscripta Math.}, 124(4):533--550, 2007.

\bibitem{auroux}
D.~Auroux.
\newblock A beginner's introduction to {F}ukaya categories.
\newblock In {\em Contact and symplectic topology}, volume~26 of {\em Bolyai
  Soc. Math. Stud.}, pages 85--136. J\'anos Bolyai Math. Soc., Budapest, 2014.

\bibitem{Barannikov}
S.~A. Barannikov.
\newblock The framed {M}orse complex and its invariants.
\newblock In {\em Singularities and bifurcations}, volume~21 of {\em Adv.
  Soviet Math.}, pages 93--115. Amer. Math. Soc., Providence, RI, 1994.

\bibitem{BarraudCorneaSerre}
J.-F. Barraud and O.~Cornea.
\newblock Lagrangian intersections and the {S}erre spectral sequence.
\newblock {\em Ann. of Math.}, 166:657--722, 2007.

\bibitem{BauLes}
U.~Bauer and M.~Lesnick.
\newblock Induced matchings and the algebraic stability of persistence
  barcodes.
\newblock {\em J. Comput. Geom.}, 6(2):162--191, 2015.

\bibitem{BesseClosed}
A.~L. Besse.
\newblock {\em Manifolds all of whose geodesics are closed}, volume~93 of {\em
  Ergebnisse der Mathematik und ihrer Grenzgebiete [Results in Mathematics and
  Related Areas]}.
\newblock Springer-Verlag, Berlin-New York, 1978.
\newblock With appendices by D. B. A. Epstein, J.-P. Bourguignon, L.
  B\'erard-Bergery, M. Berger and J. L. Kazdan.

\bibitem{Biran:Nonintersections}
P.~Biran.
\newblock Lagrangian non-intersections.
\newblock {\em Geom. Funct. Anal.}, 16(2):279--326, 2006.

\bibitem{Bi-Co:qrel-long}
P.~Biran and O.~Cornea.
\newblock Quantum structures for {L}agrangian submanifolds.
\newblock {\em Preprint}, 2007.
\newblock Available at arXiv:0708.4221 [math.SG].

\bibitem{BiranCorneaLagrangianQuantumHomology}
P.~Biran and O.~Cornea.
\newblock A {L}agrangian quantum homology.
\newblock In {\em New perspectives and challenges in symplectic field theory},
  volume~49 of {\em CRM Proc. Lecture Notes}, pages 1--44. Amer. Math. Soc.,
  Providence, RI, 2009.

\bibitem{BiranCorneaRigidityUniruling}
P.~Biran and O.~Cornea.
\newblock Rigidity and uniruling for {L}agrangian submanifolds.
\newblock {\em Geom. Topol.}, 13(5):2881--2989, 2009.

\bibitem{BiranCorneaS-Fukaya}
P.~Biran, O.~Cornea, and E.~Shelukhin.
\newblock Lagrangian shadows and triangulated categories.
\newblock Preprint, arXiv:1806.06630 [math.SG], 2018.

\bibitem{BiranMembrez-Cubic}
P.~Biran and C.~Membrez.
\newblock The {L}agrangian cubic equation.
\newblock {\em Int. Math. Res. Not. IMRN}, (9):2569--2631, 2016.

\bibitem{BHS-spectrum}
L.~Buhovsky, V.~Huili\'{e}re, and S.~Seyfaddini.
\newblock The action spectrum and {$C^0$} symplectic topology.
\newblock Preprint, arXiv:1808.09790, 2018.

\bibitem{scl}
D.~{Calegari}.
\newblock {\em {scl.}}
\newblock Tokyo: Mathematical Society of Japan, 2009.

\bibitem{Carlsson}
G.~Carlsson.
\newblock Topology and data.
\newblock {\em Bull. Amer. Math. Soc. (N.S.)}, 46(2):255--308, 2009.

\bibitem{CZCG}
G.~Carlsson, A.~Zomorodian, A.~Collins, and L.~J. Guibas.
\newblock Persistence barcodes for shapes.
\newblock {\em International Journal of Shape Modeling}, 11(02):149--187, 2005.

\bibitem{CharestWoodward-Fukaya}
F.~Charest and C.~Woodward.
\newblock {Fukaya algebras via stabilizing divisors}.
\newblock {\em Preprint, arXiv:1505.08146 [math.SG]}.

\bibitem{CharestWoodward-Floer17}
F.~Charest and C.~Woodward.
\newblock Floer trajectories and stabilizing divisors.
\newblock {\em J. Fixed Point Theory Appl.}, 19(2):1165--1236, 2017.

\bibitem{Char}
F.~Charette.
\newblock A geometric refinement of a theorem of {C}hekanov.
\newblock {\em J. Symplectic Geom.}, 10(3):475--491, 2012.

\bibitem{charetteCornea}
F.~Charette and O.~Cornea.
\newblock Categorification of {S}eidel's representation.
\newblock {\em Israel J. Math.}, 211(1):67--104, 2016.

\bibitem{CdSGO-structure}
F.~Chazal, V.~de~Silva, M.~Glisse, and S.~Oudot.
\newblock {\em The structure and stability of persistence modules}.
\newblock SpringerBriefs in Mathematics. Springer, [Cham], 2016.

\bibitem{Chekanov}
Y.~Chekanov.
\newblock Lagrangian intersections, symplectic energy, and areas of holomorphic
  curves.
\newblock {\em Duke Math. J.}, (1):213--226, 1998.

\bibitem{ChekanovFinsler}
Y.~V. Chekanov.
\newblock Invariant {F}insler metrics on the space of {L}agrangian embeddings.
\newblock {\em Math. Z.}, 234(3):605--619, 2000.

\bibitem{CEH-stability}
D.~Cohen-Steiner, H.~Edelsbrunner, and J.~Harer.
\newblock Stability of persistence diagrams.
\newblock {\em Discrete Comput. Geom.}, 37(1):103--120, 2007.

\bibitem{CorneaRanicki}
O.~Cornea and A.~Ranicki.
\newblock Rigidity and gluing for {M}orse and {N}ovikov complexes.
\newblock {\em J. Eur. Math. Soc. (JEMS)}, 5(4):343--394, 2003.

\bibitem{CorneaS}
O.~Cornea and E.~Shelukhin.
\newblock Lagrangian cobordism and metric invariants.
\newblock {\em J. Diff. Geom.}
\newblock to appear. Available at arXiv:1511.08550 [math.SG].

\bibitem{CrawBo}
W.~Crawley-Boevey.
\newblock Decomposition of pointwise finite-dimensional persistence modules.
\newblock {\em J. Algebra Appl.}, 14(5):1550066, 8, 2015.

\bibitem{CombinatorialFloer}
V.~de~Silva, J.~W. Robbin, and D.~A. Salamon.
\newblock Combinatorial {F}loer homology.
\newblock {\em Mem. Amer. Math. Soc.}, 230(1080):v+114, 2014.

\bibitem{ELZ-simplification}
H.~Edelsbrunner, D.~Letscher, and A.~Zomorodian.
\newblock Topological persistence and simplification.
\newblock {\em Discrete Comput. Geom.}, 28(4):511--533, 2002.
\newblock Discrete and computational geometry and graph drawing (Columbia, SC,
  2001).

\bibitem{Entov-ICM}
M.~Entov.
\newblock Quasi-morphisms and quasi-states in symplectic topology.
\newblock In {\em Proceedings of the {I}nternational {C}ongress of
  {M}athematicians---{S}eoul 2014. {V}ol. {II}}, pages 1147--1171. Kyung Moon
  Sa, Seoul, 2014.

\bibitem{EntovPolterovich-semisimp}
M.~{Entov} and L.~{Polterovich}.
\newblock {Symplectic quasi-states and semi-simplicity of quantum homology.}
\newblock In {\em {Toric topology. International conference, Osaka, Japan, May
  28--June 3, 2006}}.

\bibitem{EntovPolterovichCalabiQM}
M.~Entov and L.~Polterovich.
\newblock Calabi quasimorphism and quantum homology.
\newblock {\em Int. Math. Res. Not.}, (30):1635--1676, 2003.

\bibitem{EntovPolterovich-rigid}
M.~Entov and L.~Polterovich.
\newblock Rigid subsets of symplectic manifolds.
\newblock {\em Compos. Math.}, 145(3):773--826, 2009.

\bibitem{FloerCuplength}
A.~{Floer}.
\newblock {Cuplength estimates on Lagrangian intersections.}
\newblock {\em {Comm. Pure Appl. Math.}}, 42:335--356, 1989.

\bibitem{Fraser}
M.~{Fraser}.
\newblock {Contact spectral invariants and persistence}.
\newblock Preprint arXiv:1502.05979, 2015.

\bibitem{FukayaImmersed}
K.~Fukaya.
\newblock Unobstructed immersed {L}agrangian correspondence and filtered a
  infinity functor.
\newblock Preprint, arXiv:1706.02131 [math.SG].

\bibitem{FO3-spec}
K.~Fukaya, Y.-G. Oh, H.~Ohta, and O.~K.
\newblock Spectral invariants with bulk, quasimorphisms and {L}agrangian
  {F}loer theory.
\newblock Preprint arXiv:1105.5123 [math.SG], 2011.

\bibitem{FO3:book-vol12}
K.~Fukaya, Y.-G. Oh, H.~Ohta, and K.~Ono.
\newblock {\em Lagrangian intersection {F}loer theory: anomaly and obstruction.
  {P}arts {I} and {II}}, volume 46.1 and 46.2 of {\em AMS/IP Studies in
  Advanced Mathematics}.
\newblock American Mathematical Society, Providence, RI, 2009.

\bibitem{FOOO-polydiscs}
K.~Fukaya, Y.-G. Oh, H.~Ohta, and K.~Ono.
\newblock Displacement of polydisks and {L}agrangian {F}loer theory.
\newblock {\em J. Symplectic Geom.}, 11(2):231--268, 2013.

\bibitem{FukayaOno1}
K.~Fukaya and K.~Ono.
\newblock Arnold conjecture and {G}romov-{W}itten invariant.
\newblock {\em Topology}, 38(5):933--1048, 1999.

\bibitem{GPS-covariantly}
S.~Ganatra, J.~Pardon, and V.~Shende.
\newblock Covariantly functorial wrapped {F}loer theory on {L}iouville sectors.
\newblock 2017.
\newblock Preprint arXiv:1706.03152 [math.SG].

\bibitem{Ghrist}
R.~Ghrist.
\newblock Barcodes: the persistent topology of data.
\newblock {\em Bull. Amer. Math. Soc. (N.S.)}, 45(1):61--75, 2008.

\bibitem{GG-pseudorotations}
V.~Ginzburg and B.~G\"{u}rel.
\newblock Hamiltonian pseudo-rotations of projective spaces.
\newblock {\em Invent. Math.}, Online First:1--50, 2018.
\newblock https://doi.org/10.1007/s00222-018-0818-9.

\bibitem{GriffithsMorgan-book}
P.~Griffiths and J.~Morgan.
\newblock {\em Rational homotopy theory and differential forms}, volume~16 of
  {\em Progress in Mathematics}.
\newblock Springer, New York, second edition, 2013.

\bibitem{HoferCuplength}
H.~{Hofer}.
\newblock {Lusternik-Schnirelman theory for Lagrangian intersections.}
\newblock {\em {Ann. Inst. Henri Poincar\'{e}}}, 5(5):465--499, 1988.

\bibitem{HoferMetric}
H.~Hofer.
\newblock On the topological properties of symplectic maps.
\newblock {\em Proc. Roy. Soc. Edinburgh Sect. A}, 115(1-2):25--38, 1990.

\bibitem{HWZ-SC}
H.~Hofer, K.~Wysocki, and E.~Zehnder.
\newblock sc-smoothness, retractions and new models for smooth spaces.
\newblock {\em Discrete Contin. Dyn. Syst.}, 28(2):665--788, 2010.

\bibitem{HWZ-GW}
H.~Hofer, K.~Wysocki, and E.~Zehnder.
\newblock Applications of polyfold theory {I}: {T}he polyfolds of
  {G}romov-{W}itten theory.
\newblock {\em Mem. Amer. Math. Soc.}, 248(1179), 2017.

\bibitem{HWZ-PF17}
H.~Hofer, K.~Wysocki, and E.~Zehnder.
\newblock Polyfolds and {F}redholm theory.
\newblock 2017.
\newblock Preprint, arXiv:1707.08941 [math.FA].

\bibitem{H-PF17}
H.~H.~W. Hofer.
\newblock Polyfolds and {F}redholm theory.
\newblock In {\em Lectures on geometry}, Clay Lect. Notes, pages 87--158.
  Oxford Univ. Press, Oxford, 2017.

\bibitem{huLalonde}
S.~Hu and F.~Lalonde.
\newblock A relative {S}eidel morphism and the {A}lbers map.
\newblock {\em Trans. Amer. Math. Soc.}, 362(3):1135--1168, 2010.

\bibitem{huLalondeLeclercq}
S.~Hu, F.~Lalonde, and R.~Leclercq.
\newblock Homological {L}agrangian monodromy.
\newblock {\em Geom. Topol.}, 15(3):1617--1650, 2011.

\bibitem{HLS-coisotropic}
V.~Humili\`ere, R.~Leclercq, and S.~Seyfaddini.
\newblock Coisotropic rigidity and {$C^0$}-symplectic geometry.
\newblock {\em Duke Math. J.}, 164(4):767--799, 2015.

\bibitem{hyvrier}
C.~Hyvrier.
\newblock Lagrangian circle actions.
\newblock {\em Algebr. Geom. Topol.}, 16(3):1309--1342, 2016.

\bibitem{Kat-Mil:PSS}
J.~Kati\'c and D.~Milinkovi\'c.
\newblock {P}iunikhin-{S}alamon-{S}chwarz isomorphisms for {L}agrangian
  intersections.
\newblock {\em Differential Geom. Appl.}, 22(2):215--227, 2005.
\newblock MR2122744, Zbl 1064.37041.

\bibitem{Lalonde-McDuff-Energy}
F.~Lalonde and D.~McDuff.
\newblock The geometry of symplectic energy.
\newblock {\em Ann. of Math. (2)}, 141(2):349--371, 1995.

\bibitem{Leclercq-spectral}
R.~Leclercq.
\newblock Spectral invariants in {L}agrangian {F}loer theory.
\newblock {\em J. Mod. Dyn.}, 2(2):249--286, 2008.

\bibitem{LeclercqZapolsky}
R.~{Leclercq} and F.~{Zapolsky}.
\newblock Spectral invariants for monotone {L}agrangians.
\newblock {\em J. Topol. Anal.}
\newblock Available at https://doi.org/10.1142/S1793525318500267.

\bibitem{LekiliPascaleff-equivariant}
Y.~Lekili and J.~Pascaleff.
\newblock Floer cohomology of {$\mathfrak{g}$}-equivariant {L}agrangian branes.
\newblock {\em Compos. Math.}, 152(5):1071--1110, 2016.

\bibitem{Lerman-cuts}
E.~Lerman.
\newblock Symplectic cuts.
\newblock {\em Math. Res. Lett.}, 2(3):247--258, 1995.

\bibitem{LSV-conj}
F.~Leroux, S.~Seyfaddini, and C.~Viterbo.
\newblock Barcodes and area-preserving homeomorphisms.
\newblock Preprint arXiv:1810.03139 [math.SG], 2018.

\bibitem{LisiRieser}
S.~Lisi and A.~Rieser.
\newblock Coisotropic {H}ofer-{Z}ehnder capacities and non-squeezing for
  relative embeddings.
\newblock 2013.
\newblock Preprint, arXiv:1312.7334 [math.SG].

\bibitem{LiuTian}
G.~Liu and G.~Tian.
\newblock Floer homology and {A}rnold conjecture.
\newblock {\em J. Differential Geom.}, 49(1):1--74, 1998.

\bibitem{MelissaLiu}
M.~Liu.
\newblock {Moduli of $J$-holomorphic curves with Lagrangian boundary conditions
  and open Gromov-Witten invariants for an $S^1$-equivariant pair}.
\newblock Preprint arXiv:math/0210257, 2004.

\bibitem{GcLu}
G.~Lu.
\newblock Gromov-{W}itten invariants and pseudo symplectic capacities.
\newblock {\em Israel J. Math.}, 156:1--63, 2006.

\bibitem{McDuff-K17}
D.~McDuff.
\newblock {Constructing the virtual fundamental class of a Kuranishi atlas}.
\newblock 2017.
\newblock Preprint, arXiv:1708.01127 [math.SG].

\bibitem{McDuffPolterovichPacking}
D.~McDuff and L.~Polterovich.
\newblock Symplectic packings and algebraic geometry.
\newblock {\em Invent. Math.}, 115(3):405--434, 1994.
\newblock With an appendix by Yael Karshon.

\bibitem{McDuffSalamonBIG}
D.~{McDuff} and D.~{Salamon}.
\newblock {\em {$J$-holomorphic curves and symplectic topology. 2nd ed.}},
  volume~52.
\newblock Providence, RI: American Mathematical Society (AMS), 2nd ed. edition,
  2012.

\bibitem{McDuffTolman}
D.~McDuff and S.~Tolman.
\newblock Topological properties of {H}amiltonian circle actions.
\newblock {\em IMRP Int. Math. Res. Pap.}

\bibitem{MonznerVicheryZapolsky}
A.~Monzner, N.~Vichery, and F.~Zapolsky.
\newblock Partial quasimorphisms and quasistates on cotangent bundles, and
  symplectic homogenization.
\newblock {\em J. Mod. Dyn.}, 6(2):205--249, 2012.

\bibitem{oancea}
A.~Oancea.
\newblock A survey of {F}loer homology for manifolds with contact type boundary
  or symplectic homology.
\newblock In {\em Symplectic geometry and {F}loer homology. {A} survey of the
  {F}loer homology for manifolds with contact type boundary or symplectic
  homology}, volume~7 of {\em Ensaios Mat.}, pages 51--91. Soc. Brasil. Mat.,
  Rio de Janeiro, 2004.

\bibitem{Oh-RiemannHilbert}
Y.-G. Oh.
\newblock Riemann-{H}ilbert problem and application to the perturbation theory
  of analytic discs.
\newblock {\em Kyungpook Math. J.}, 35(1):39--75, 1995.

\bibitem{Oh:SpectralSequence}
Y.-G. Oh.
\newblock Floer cohomology, spectral sequences, and the {M}aslov class of
  {L}agrangian embeddings.
\newblock {\em Internat. Math. Res. Notices}, 1996(7):305--346, 1996.

\bibitem{Oh-spec-lagr}
Y.-G. Oh.
\newblock Symplectic topology as the geometry of action functional. {II}.
  {P}ants product and cohomological invariants.
\newblock {\em Comm. Anal. Geom.}, 7(1):1--54, 1999.

\bibitem{Oh-Turkish}
Y.-G. Oh.
\newblock Floer homology and its continuity for non-compact {L}agrangian
  submanifolds.
\newblock {\em Turkish J. Math.}, 25(1):103--124, 2001.

\bibitem{Oh-chain}
Y.-G. Oh.
\newblock Chain level {F}loer theory and {H}ofer's geometry of the
  {H}amiltonian diffeomorphism group.
\newblock {\em Asian J. Math.}, 6(4):579--624, 2002.

\bibitem{Oh-construction}
Y.-G. Oh.
\newblock Construction of spectral invariants of {H}amiltonian paths on closed
  symplectic manifolds.
\newblock In {\em The breadth of symplectic and {P}oisson geometry}, volume 232
  of {\em Progr. Math.}, pages 525--570. Birkh\"auser Boston, Boston, MA, 2005.

\bibitem{Oh-specnorm}
Y.-G. {Oh}.
\newblock {Spectral invariants, analysis of the Floer moduli space, and
  geometry of the Hamiltonian diffeomorphism group.}
\newblock {\em {Duke Math. J.}}, 130(2):199--295, 2005.

\bibitem{OhBook}
Y.-G. Oh.
\newblock {\em Symplectic topology and {F}loer homology. {V}ols. 1 and 2},
  volume 27 and 28 of {\em New Mathematical Monographs}.
\newblock Cambridge University Press, Cambridge, 2015.
\newblock Symplectic geometry and pseudoholomorphic curves.

\bibitem{Oh-Zhu:PSS}
Y.-G. Oh and K.~Zhu.
\newblock Floer trajectories with immersed nodes and scale-dependent gluing.
\newblock {\em J. Symplectic Geom.}, 9:483--636, 2011.
\newblock MR2900788, Zbl 1257.53117.

\bibitem{PardonHam}
J.~Pardon.
\newblock An algebraic approach to virtual fundamental cycles on moduli spaces
  of pseudo-holomorphic curves.
\newblock {\em Geom. Topol.}, 20(2):779--1034, 2016.

\bibitem{PSS}
S.~Piunikhin, D.~Salamon, and M.~Schwarz.
\newblock Symplectic {F}loer-{D}onaldson theory and quantum cohomology.
\newblock In {\em Contact and symplectic geometry ({C}ambridge, 1994)},
  volume~8 of {\em Publ. Newton Inst.}, pages 171--200. Cambridge Univ. Press,
  Cambridge, 1996.

\bibitem{Polterovich-isotopy}
L.~Polterovich.
\newblock Symplectic displacement energy for {L}agrangian submanifolds.
\newblock {\em Ergodic Theory Dynam. Systems}, 13(2):357--367, 1993.

\bibitem{P-book}
L.~Polterovich.
\newblock {\em The geometry of the group of symplectic diffeomorphisms}.
\newblock Lectures in Mathematics ETH Z\"urich. Birkh\"auser Verlag, Basel,
  2001.

\bibitem{PolterovichRosen}
L.~Polterovich and D.~Rosen.
\newblock {\em Function theory on symplectic manifolds}, volume~34 of {\em CRM
  Monograph Series}.
\newblock American Mathematical Society, Providence, RI, 2014.

\bibitem{PolShe}
L.~Polterovich and E.~Shelukhin.
\newblock Autonomous {H}amiltonian flows, {H}ofer's geometry and persistence
  modules.
\newblock {\em Selecta Math. (N.S.)}, 22(1):227--296, 2016.

\bibitem{PolSheSto}
L.~Polterovich, E.~Shelukhin, and V.~Stojisavljevi\'c.
\newblock Persistence modules with operators in {M}orse and {F}loer theory.
\newblock {\em Mosc. Math. J.}, 17(4):757--786, 2017.

\bibitem{RobbinSalamonIndex}
J.~Robbin and D.~Salamon.
\newblock The {M}aslov index for paths.
\newblock {\em Topology}, 32(4):827--844, 1993.

\bibitem{SchwarzCuplength}
M.~{Schwarz}.
\newblock {A quantum cup-length estimate for symplectic fixed points.}
\newblock {\em {Invent. Math.}}, 133:353--397, 1998.

\bibitem{Schwarz:action-spectrum}
M.~Schwarz.
\newblock On the action spectrum for closed symplectically aspherical
  manifolds.
\newblock {\em Pacific J. Math.}, 193(2):419--461, 2000.

\bibitem{SeidelGraded}
P.~{Seidel}.
\newblock {Graded Lagrangian submanifolds.}
\newblock {\em {Bull. Soc. Math. Fr.}}

\bibitem{seidelInvertibles}
P.~Seidel.
\newblock {$\pi_1$} of symplectic automorphism groups and invertibles in
  quantum homology rings.
\newblock {\em Geom. Funct. Anal.}, 7(6):1046--1095, 1997.

\bibitem{SeidelBook}
P.~Seidel.
\newblock {\em Fukaya categories and {P}icard-{L}efschetz theory}.
\newblock Zurich Lectures in Advanced Mathematics. European Mathematical
  Society (EMS), Z\"urich, 2008.

\bibitem{SeidelDisjoinable}
P.~Seidel.
\newblock Disjoinable lagrangian spheres and dilations.
\newblock {\em Invent. Math.}, 197(2):299--359, 2014.

\bibitem{SeyfaddiniC0Limits}
S.~Seyfaddini.
\newblock {$C^0$}-limits of {H}amiltonian paths and the {O}h-{S}chwarz spectral
  invariants.
\newblock {\em Int. Math. Res. Not. IMRN}, (21):4920--4960, 2013.

\bibitem{ES-Viterbo}
E.~Shelukhin.
\newblock In preparation.

\bibitem{ShelukhinHZ}
E.~Shelukhin.
\newblock In preparation.

\bibitem{STV-Flux18}
E.~Shelukhin, D.~Tonkonog, and R.~Vianna.
\newblock {Geometry of symplectic flux and Lagrangian torus fibrations}.
\newblock {\em Preprint, arXiv:1804.02044 [math.SG]}.

\bibitem{SmithPencils}
I.~Smith.
\newblock Floer cohomology and pencils of quadrics.
\newblock {\em Invent. Math.}, 189(1):149--250, 2012.

\bibitem{Steenrod-book}
N.~Steenrod.
\newblock {\em The {T}opology of {F}ibre {B}undles}.
\newblock Princeton Mathematical Series, vol. 14. Princeton University Press,
  Princeton, N. J., 1951.

\bibitem{stevenson}
B.~Stevenson.
\newblock A quasi-isometric embedding into the group of {H}amiltonian
  diffeomorphisms with {H}ofer's metric.
\newblock {\em Israel J. Math.}, 223(1):141--195, 2018.

\bibitem{KimuraStasheffVoronov}
J.~S. T.~Kimura and A.~Voronov.
\newblock On operad structures of moduli spaces and string theory.
\newblock {\em Comm. Math. Phys.}

\bibitem{Usher-spec}
M.~Usher.
\newblock Spectral numbers in {F}loer theories.
\newblock {\em Compos. Math.}, 144(6):1581--1592, 2008.

\bibitem{Usher-Sharp}
M.~{Usher}.
\newblock {The sharp energy-capacity inequality.}
\newblock {\em {Commun. Contemp. Math.}}, 12(3):457--473, 2010.

\bibitem{Usher-private}
M.~Usher.
\newblock Private communication, 2011.

\bibitem{UsherBD1}
M.~Usher.
\newblock Boundary depth in {F}loer theory and its applications to
  {H}amiltonian dynamics and coisotropic submanifolds.
\newblock {\em Israel J. Math.}, 184:1--57, 2011.

\bibitem{UsherBD2}
M.~Usher.
\newblock Hofer's metrics and boundary depth.
\newblock {\em Ann. Sci. \'Ec. Norm. Sup\'er. (4)}, 46(1):57--128 (2013), 2013.

\bibitem{Usher-trick}
M.~Usher.
\newblock Observations on the {H}ofer distance between closed subsets.
\newblock {\em Math. Res. Lett.}, 22(6):1805--1820, 2015.

\bibitem{UsherZhang}
M.~Usher and J.~Zhang.
\newblock Persistent homology and {F}loer--{N}ovikov theory.
\newblock {\em Geom. Topol.}, 20(6):3333--3430, 2016.

\bibitem{ViterboMaslov}
C.~{Viterbo}.
\newblock Intersection de sous-vari\'{e}t\'{e}s lagrangiennes, fonctionnelles
  d’action et indice des syst\`{e}mes hamiltoniens.
\newblock {\em Bull. Soc. Math. France}, 115(3):361--390, 1987.

\bibitem{Viterbo-specGF}
C.~Viterbo.
\newblock Symplectic topology as the geometry of generating functions.
\newblock {\em Math. Ann.}, 292(4):685--710, 1992.

\bibitem{Zap:Orient}
F.~Zapolsky.
\newblock {The Lagrangian Floer-quantum-PSS package and canonical orientations
  in Floer theory}.
\newblock {\em Preprint}.

\bibitem{Zhang}
J.~{Zhang}.
\newblock {p-cyclic persistent homology and Hofer distance}.
\newblock {\em J. Symp. Geom.}
\newblock to appear. Available at arXiv:1605.07594.

\bibitem{CarlZom}
A.~Zomorodian and G.~Carlsson.
\newblock Computing persistent homology.
\newblock {\em Discrete Comput. Geom.}, 33(2):249--274, 2005.

\end{thebibliography}

\end{document}